\documentclass[11pt]{amsart}
\usepackage{graphicx}
\usepackage{marginnote}
\usepackage[pagebackref=true]{hyperref}
\usepackage{comment}
\usepackage{soul}
\usepackage[normalem]{ulem}

\usepackage[dvipsnames]{xcolor}
\usepackage{pinlabel,enumitem,comment,amssymb}
\usepackage[margin=1.25in]{geometry}
\usepackage{thm-restate}
\usepackage{outlines}
\usepackage[textsize=small]{todonotes}
\usepackage{thm-restate}

\makeatletter
\@mparswitchfalse
\makeatother
\normalmarginpar
    \definecolor{purple}{rgb}{.5, 0, 0.5}
    \definecolor{mybrown}{RGB}{144, 85, 1}
    \definecolor{mygreen}{RGB}{0, 101, 46}
    \definecolor{myblue}{RGB}{0, 64, 113}

\newtheorem{theorem}{Theorem}[section]
\newtheorem{corollary}[theorem]{Corollary}
\newtheorem{lemma}[theorem]{Lemma}
\newtheorem{proposition}[theorem]{Proposition}

\theoremstyle{definition}
\newtheorem{definition}[theorem]{Definition}
\newtheorem{remark}[theorem]{Remark}
\newtheorem{observation}[theorem]{Observation}

\theoremstyle{definition}
\newtheorem{question}[theorem]{Question}
\newtheorem*{question*}{Question}

\newtheorem{notation}[theorem]{Notation}
\newtheorem{example}[theorem]{Example}

\newtheorem{procedure}[theorem]{Procedure}
\newtheorem{construction}[theorem]{Construction}

\newcommand\ttimes{\mathbin{%
    \stackrel{\sim}{\smash{\times}\rule{0pt}{0.7ex}}%
    }}
\newcommand\cut{\setminus\!\!\setminus}

\makeatother

\author[Karimi]{Homayun Karimi}
\address{McMaster University\\Hamilton, ON L8S 4L8}
\email{homayun.karimi@gmail.com}

\author[Kim]{Seungwon Kim}
\address{Department of Mathematics\\Sungkyunkwan University\\Suwon, Gyeonggi, 16419 Republic of Korea}
\email{seungwon.kim@skku.edu}

\author[Kindred]{Thomas Kindred}
\address{Department of Mathematical Sciences \\ Smith College \\ Northampton, MA 01063 USA}
\email{tkindred@smith.edu}

\author[Miller]{Maggie Miller}
\address{Department of Mathematics\\The University of Texas at Austin\\Austin, TX 78712 USA}
\email{maggie.miller.math@gmail.com}

\author[Naylor]{Patrick Naylor}
\address{Department of Mathematics and Statistics\\McMaster University\\Hamilton, ON L8S 4L8}
\email{patrick.naylor@mcmaster.ca}

\title{Spanning solids and Murasugi sum in dimension four}

\begin{document}

\begin{abstract}
    We study various constructions of spanning solids for knotted surfaces in the 4-sphere. In particular, we consider a 4-dimensional analogue of Murasugi sum, which we use to define a notion of arborescent {knotted surface}s. We give a variety of examples and applications to broken surface diagrams, such as deciding which resolutions of the crossings in these diagrams yield spanning solids.
\end{abstract}

\maketitle

\section{Introduction}

In this paper, we study various constructions of {\emph{spanning solids}}, i.e., compact 3-manifolds with nonempty boundary that are smoothly embedded in $S^4$. In general, we do not require spanning solids to be orientable -- although if a spanning solid is oriented, we may call it a {\emph{Seifert solid}}. The boundary of a spanning solid, whether connected or disconnected, is a {\it knotted surface} in $S^4$. 
Our principal aims are (1) to adapt several theorems and constructions from the classical setting of spanning surfaces in $S^3$ 
to the 4-dimensional setting; (2) to highlight many of the questions that arise naturally from these adaptations; and (3) to produce interesting families of (spanning solids for) knotted surfaces in $S^4$. 

In particular, many classical constructions and results can be understood in terms of an operation called {\it Murasugi sum}, which is a certain way of gluing two spanning surfaces along a disk, and we adapt this operation and several of its applications to the 4-dimensional setting.  For example, we use Murasugi sum of spanning solids to define and examine notions of arborescent {knotted surface}s, which have spanning solids constructed from simple pieces via Murasugi sum in the pattern of a tree; state solids from projections of knotted surfaces to $S^3$, which are analogous to state surfaces from link diagrams on $S^2$; and $\pi_1$-essentiality of spanning solids. 

In Section \ref{sec:murasugi}, we introduce an analogue of Murasugi sum in dimension four (Definition \ref{def:murasugi_sum}). In the classical setting, a Murasugi sum of two spanning surfaces $S_1, S_2$ in $S^3$ is performed along $2n$-gons $P_i\in S_i$. The resulting surface is contained in $S^3=S^3\# S^3$, with $S_1,S_2$ in the corresponding summand, each intersecting the connected sum sphere in precisely $P_i$. In our setting, Murasugi sum will similarly glue 
two spanning solids along an embedded 3-ball whose boundary is subdivided in a manner analogous to that of the $2n$-gon.

Such an operation was previously studied by Ozbagci and Popescu-Pampu \cite{ozbagci_popescupampu}, who gave a version for solids in arbitrary 4-manifolds and also allowed the gluing region to be a general submanifold rather than specifically a 3-ball (in fact, they worked with Murasugi sum in arbitrary dimension). We focus on Murasugi sums in the restricted setting of $S^4$.

In Section \ref{sec:arborescent}, we define a notion of an {\emph{arborescent spanning solid}}, and define a {knotted surface} to be arborescent when it bounds such a solid. A reader familiar with {knotted surface}s might guess that these are simply the spins of classical arborescent links, but in fact, spinning an arborescent spanning surface generally does {\emph{not}} yield an arborescent spanning solid. (While we have not yet provided a definition, one should note that when spanning surfaces are plumbed along disks, spinning will transform this disk into a solid torus rather than a ball unless the spinning axis is carefully chosen.)

\newtheorem*{arborescent_determined}{Theorem \ref{thm:arborescent_determined}}
\begin{arborescent_determined}
    Our notion of arborescent {knotted surface} is well-defined, with each link determined by a weighted graph associated to a spanning solid. The class of arborescent surfaces includes spins of 2-bridge links (Proposition \ref{prop:2stripespin}) as well as non-spun surfaces such as the genus-2 link in Example \ref{ex:interesting_example}.
\end{arborescent_determined}

A classical link diagram gives rise, via Kauffman states, to state surfaces, each of which is a Murasugi sum of checkerboard surfaces.  We adapt this construction to broken surface diagrams (projections with crossing information) of knotted surfaces $K\subset S^4$. Namely, we resolve the crossings of the diagram in $S^3$ to get a system of closed surfaces called a {\emph{state}}, cap off each surface in the state with a solid, and connect these solids with standard pieces near the crossings to get a spanning solid for $K$. We call this solid a {\emph{state solid}}.

One difficulty in working with broken surface diagrams is that, in contrast to classical link diagrams(whose self-intersections are all double points), the self-intersections of these diagrams may include not just double points, but also triple points, as well as branch points or cusps where the surface is locally a cone on a small loop with one self-intersection. 
It turns out that, whereas all Kauffman states of a classical link diagram yield state surfaces, not all crossing resolutions in a broken surface diagram yield state solids: the difficulties arise at cusps and triple points. We give a classification of exactly which states of a broken surface diagram yield state solids. 

\newtheorem*{state_solids}{Theorem \ref{thm:state_solid}}
\begin{state_solids}
We provide a complete classification of which resolutions of the crossings in broken surface diagrams give rise to state solids.
\end{state_solids}

We are not able to give the exact statement of Theorem \ref{thm:state_solid} here without introducing much terminology, so we delay the precise statement to Section \ref{sec:states}. The full version includes distinct classifications for several subclasses of state solids.

Because of the extra difficulties presented by triple points and cusps, we also devote specific attention to knotted surfaces with broken surface diagrams whose crossings are only double points. 
These diagrams and the surfaces that they represent are called {\emph{pseudo-ribbon}}. 

Every classical knot or link has a spanning surface that is a Murasugi sum of trivial M\"obius bands (spanning the unknot).  We adapt this fact as follows.

\newtheorem*{ribbonstate}{Theorem \ref{thm:ribbonstate}}
\begin{ribbonstate}
    Every ribbon 2-knot has a spanning solid that is a Murasugi sum of spun trivial M\"obius bands.
\end{ribbonstate}

(The adjective {\emph{ribbon}} is another common term (not defined here) which for general surfaces is not equivalent to pseudo-ribbon, but for 2-knots ``ribbon" and ``pseudo-ribbon" agree \cite{yajima1964}.)

See the statement of Theorem \ref{thm:ribbonstate} in \textsection\ref{sec:ribbon} for a specific description of how these Murasugi sums are performed.

We also define and characterize {\emph{ribbon-alternating}} broken surface diagrams, a special case of pseudo-ribbon diagrams. Such a diagram has intersection set with no triple or branch points, and its double points cut the surface into disks and annuli, such that each annulus is incident to an overpass and an underpass. That is, in each component of the preimage of the immersion, the intersection set lifts to parallel circles that 
alternate between under and overpasses. A surface $K$ is {\emph{ribbon-alternating}} if it admits a ribbon-alternating diagram. 
We characterize ribbon-alternating knotted surfaces as follows.

\newtheorem*{thm:alternating}{Theorem \ref{thm:alternating}}

\begin{thm:alternating}
    Let $K\subset S^4$ be a knotted surface admitting a connected, ribbon-alternating diagram $E$. 
    \begin{enumerate}[label=(\alph*)]
        \item If some component of $K$ is a 2-sphere, then $E$ is the result of spinning an alternating classical tangle diagram (without deformation), and in particular $K$ is a union of spheres and tori. 
        \item If $K$ is orientable and contains no 2-spheres, then it is a union of tori, and 
        there is an alternating link diagram $D\subset D^2$ such that $E$ is obtained by spinning $D$ while turning $D$ about some point $x\in D^2\setminus D$ of rotational symmetry  
        (thus, preserving orientations). 
        \item If $K$ is nonorientable, then $K$ has exactly one Klein bottle component, the other components of $K$ are tori, and there is an alternating link diagram $D$ such that $E$ is obtained by spinning $D$ while rotating $\pm 180^\circ$ about a crossing of $D$, where the rotation fixes $D$ setwise.
    \end{enumerate}
\end{thm:alternating}

We also prove the following intuitive fact, although the related Question \ref{q:non_alternating} and Example \ref{ex:35_nugatory} are less straightforward and perhaps counterintuitive.

\newtheorem*{thm:non_alternating}{Theorem \ref{thm:non_alternating}}

\begin{thm:non_alternating}
If $K$ is the spin of a 
classical nonalternating knot $L$, then $K$ is not ribbon-alternating. 
\end{thm:non_alternating}

We define {\emph{flypes}} of broken surface diagrams and adapt two of the three classical Tait conjectures to ribbon-alternating knotted surfaces---see Definition \ref{def:flype_T2} and Questions \ref{Q:Tait_1} and \ref{q:flypes}.

Murasugi sum of spanning surfaces interacts well with the notion of {\emph{$\pi_1$-essentiality}}, and we show that the same is true for spanning solids. Let $M$ be a spanning solid for a surface $K\subset S^4$, and let $X$ denote the exterior of $M$ in $S^4$. The boundary of $X$ contains a copy of $K$, with the rest of $\partial X$ consisting of a twofold cover of $M$ (the copy of $K$ in $\partial X$ will separate $\partial X$ if and only if $M$ is orientable). The solid $M$ is said to be $\pi_1$-essential if every loop representing an element of $\ker(\pi_1(\partial X)\to X)$ either intersects $K$ at least twice or represents a trivial element of $\pi_1(M)$. In other words, each loop in $\partial X$ that is contractible in $X$ either (1) intersects $K$ at least twice or (2) lies in, and is contractible in, $M$. In particular, a Seifert solid is $\pi_1$-essential if and only if the only loops in $M$ that are contractible in $X$ are those that are contractible in $M$.

\newtheorem*{essentialthm}{Theorem \ref{thm:essential}}
\begin{essentialthm}
   Any Murasugi sum $M=M_0*M_1$ of $\pi_1$-essential spanning solids is $\pi_1$-essential.
\end{essentialthm}

In Section \ref{sec:essential}, we also prove a generalization of Theorem \ref{thm:essential} in which we allow a more general version of Murasugi sum. Specifically, we allow $M_0, M_1$ to be plumbed along a 3-manifold with boundary that is not necessarily a 3-ball. See Section \ref{sec:essential} for details.

\subsection*{Organization} The paper is organized as follows. 
\begin{itemize}[leftmargin=1in]
    \item[\bf{Section
\ref{sec:murasugi}:}] We define Murasugi sum of spanning solids in $S^4$.
\item[\bf{Section \ref{sec:arborescent}:}] We use Murasugi sum to define a notion of arborescent {knotted surface}s, which exhibit behavior analogous to that of arborescent links in $S^3$.
\item[\bf{Section \ref{sec:states}:}] We discuss state solids for broken surface diagrams of knotted surfaces, determining exactly which states yield embedded spanning solids. 
\item[\bf{Section \ref{sec:ribbon}:}] We investigate pseudo-ribbon diagrams and their state solids. 
\item[\bf{Section \ref{sec:alternating}:}]We understand ribbon-alternating surface diagrams, a particular class of pseudo-ribbon diagrams. 
\item[\bf{Section 
\ref{sec:essential}:}] We define what it means for a spanning solid to be $\pi_1$-essential, and show that this property is preserved under Murasugi sum.
\item[\bf{Section \ref{sec:questions}:}] We collect questions that we have asked throughout the paper.
\end{itemize}

\subsection*{Acknowledgments} This collaboration began at the 2024 Trisectors Workshop, which was held at the University of Nebraska-Lincoln in June 2024. Some work also took place in a Fields Opportunities for Collaborations US (FOCUS) visit in April 2026. MM was supported by a Clay Research Fellowship, NSF grant DMS-2404810, Simons Foundation Gift MPS-TSM-00007679, a Sloan Research Fellowship and a Packard Fellowship for Science and Engineering. PN was supported by an NSERC Discovery Grant and a CRM-Simons Scholar Grant.

\section{Murasugi sum}\label{sec:murasugi}

\subsection{Preliminaries}\label{subsec:preliminaries} 
We work in the smooth category throughout. Unless otherwise specified, all manifolds considered in this paper are compact; they may be orientable or non-orientable. If $M$ is a manifold, then $\mathring{M}$ will denote its interior. If $M$ is a submanifold of some ambient manifold, we write $\nu M$ for a closed regular neighborhood of $M$ and $\mathring\nu M$ for its interior. If $M$ is a 3-manifold then $M^\circ$ will denote $M$ with an open 3-ball deleted from its interior, and if $F$ is a surface with boundary then $F^\circ$ will denote $F$ with an open half-disk deleted near its boundary (if $\partial F$ has corners, this half disk will be disjoint from them).

We will call (the image of) an embedding of a surface $\Sigma\hookrightarrow S^4$ a \emph{knotted surface}, and reserve the term \emph{2-knot} for the case that $\Sigma$ is a 2-sphere. Knotted surfaces need not be connected nor orientable. We will refer to any 3-manifold smoothly embedded in $S^4$ with boundary $\Sigma$ as a {\emph{spanning solid}} for $\Sigma$. We regard two knotted surfaces $K_1,K_2\subset S^4$ as (smoothly) equivalent if there is a pairwise diffeomorphism $(S^4,K_1)\to(S^4,K_2)$ that preserves the orientations of $S^4$, and we regard two spanning solids $M_1,M_2\subset S^4$ as equivalent if there is a pairwise diffeomorphism $(S^4,M_1)\to(S^4,M_2)$ that preserves the orientation of $S^4$.

For convenience, we will use the following notation, adapted from the classical setting. 

\begin{definition}\label{def:cut}
    Suppose that $M\subset S^4$ is a spanning solid. We write $S^4\cut M$ to denote $S^4$ \emph{cut along} $M$. As a topological space, this is given by $S^4\setminus \mathring{\nu} M$; its boundary contains a copy of $\partial M$, which, in the case that $M$ is orientable, cuts $\partial(S^4\cut M)$ into two copies of $M$.\footnote{One could define $S^4\cut M$ precisely as the metric closure of $S^4\setminus M$, where points are equivalence classes of Cauchy sequences that are identified if their interpolation converges.} We will also use this notation in other similar contexts.
\end{definition}

\subsection{A 4-dimensional Murasugi sum} In this section, we define the notion of {\emph{Murasugi sum}} for spanning solids in $S^4$. We first remind the reader of the classical definition introduced by Murasugi \cite{murasugi} (although explored in more similar notation to ours by Gabai \cite{gabai_murasugisumnatural}).

\begin{definition}\label{def:classicalmurasugisum}
    Let $M_1, M_2$ each denote a copy of $S^3$ and let $S_1, S_2$ be spanning surfaces for some knots or links in $M_1, M_2$, respectively. For some $n\in \mathbb{N}$, let $P$ be an oriented $2n$-gon, where the boundary of $P$ consists of the ordered edges $e_1,e_2,\ldots, e_{2n}$. Choose embeddings $f_i:P\to S_i$ such that $f_1$ sends $e_1,e_3,\ldots, e_{2n-1}$ to $\partial S_1$ but the other edges and interior of $P$ to the interior of $S_1$, while $f_2$ sends $e_2,e_4,\ldots, e_{2n}$ to $\partial S_2$ but the other edges and interior of $P$ to the interior of $S_2$. 

    Delete an open ball from each $M_1\setminus S_1, M_2\setminus S_2$ so that $S_1, S_2$ meet the boundary of $M_1'=M_1\setminus \mathring{B}^3$ and $M_2'=M_2\setminus\mathring{B}^3$ in exactly $f_1(P)$ and $f_2(P)$, respectively. In $M_1$, obtain this open ball by thickening $f_1(P)$ according to the positive normal (using the induced orientation from $f_1$). In $M_2$, obtain this open ball by thickening $f_2(P)$ according to the negative normal. 
    The positive normal to $f_1(P)$ points out of the ball $M_1'$, whereas the positive normal to $f_2(P)$ points into the ball $M_2'$. 
    
    Now, write $S^3$ as $M_1'\cup_g M_2'$, where the gluing map $g$ is an orientation-reversing diffeomorphism $g:\partial M_1'\to\partial M_2'$, chosen so that $g\circ f_1=f_2$. The spanning surface $S_1\cup S_2$ in $S^3$ is said to be obtained from $S_1, S_2$ by {\emph{Murasugi sum.}} See Figure \ref{fig:murasugi}.
\end{definition}

Note that the Murasugi sum of two spanning surfaces depends on the gluing polygons in each surface and the identification between them.

\begin{figure}
\labellist
\pinlabel{$S_1$} at 25 90
\pinlabel{$S_2$} at 245 90
\endlabellist
    \includegraphics[width=90mm]{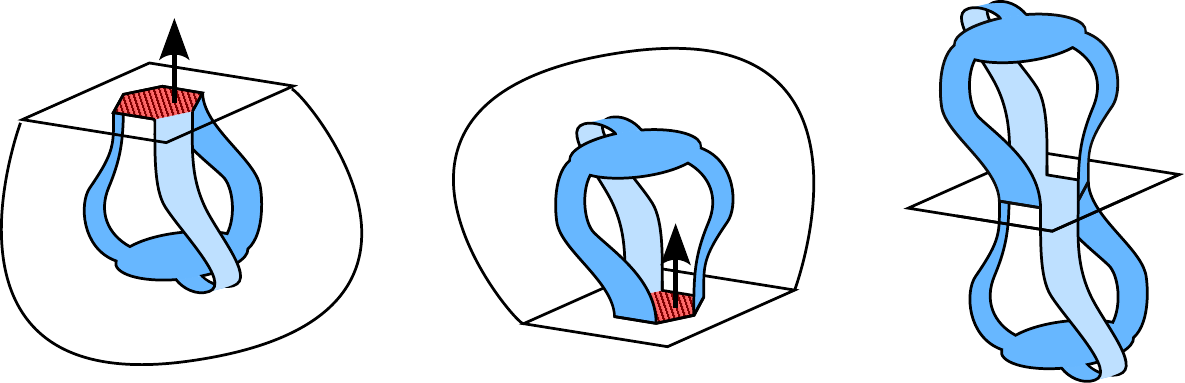}
    \caption{The leftmost two images show two surfaces $S_1, S_2$ embedded in 3-balls. Each surface meets the boundary of the 3-ball in a shaded $2n$-gon (here, $n=3$). We glue the surfaces 
    $S_1, S_2$ (simultaneously gluing the two 3-balls) to form a new surface in $S^3$ obtained by Murasugi sum. This new surface is illustrated in the rightmost image.}\label{fig:murasugi}
\end{figure}

We will now define Murasugi sum for spanning solids in 4-dimensional space. The construction will be analogous to the condensed classical definition above.

\begin{definition}\label{def:colored_ball}
    A {\emph{colored 3-ball}} $\mathcal{B}$ is a triple $\mathcal{B}=(B,X,Y)$, where $B$ is a 3-ball, and $X,Y\subset \partial B$ are surfaces with the property that $\partial B=X\cup Y$ and $X\cap Y=\partial X=\partial Y$.
\end{definition}

One might consider a more general version of the preceding definition, in which $X\cap Y$ is allowed to have singular points. (That is, $X\cap Y$ includes components that are graphs rather than closed loops.) We postpone the discussion of why the given definition avoids any loss of generality to Remark \ref{prop:singular}.

\begin{example}\label{ex:colored_balls}
    Figure \ref{fig:colored_3ball} illustrates some examples of colored 3-balls $(B,X,Y)$; in each case, $X$ (lighter; orange) and $Y$ (darker; blue) are subsurfaces of $\partial B$ whose union is $\partial B$, and which meet along a collection of closed curves. 

    \begin{figure}[!ht]
        \centering
        \hspace{0.1\textwidth}
        \includegraphics[width=.2\textwidth]{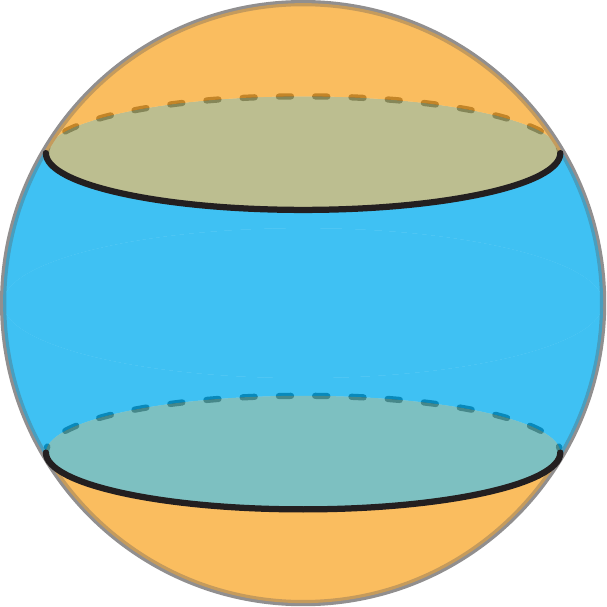}
        \hfill
        \includegraphics[width=.2\textwidth]{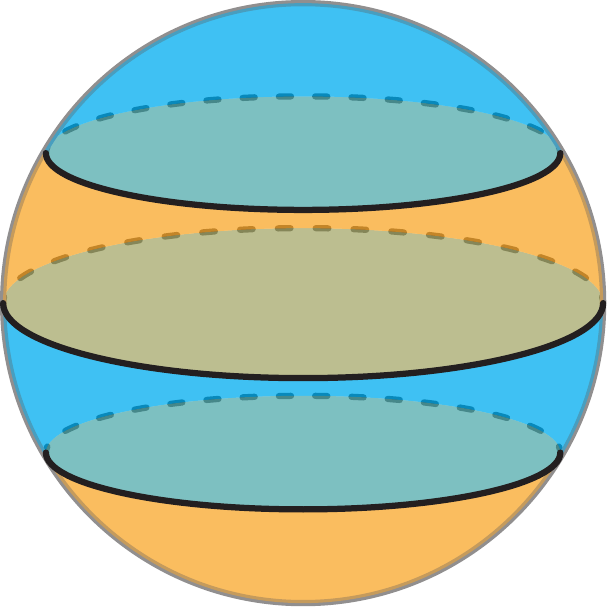}
        \hfill
        \includegraphics[width=.2\textwidth]{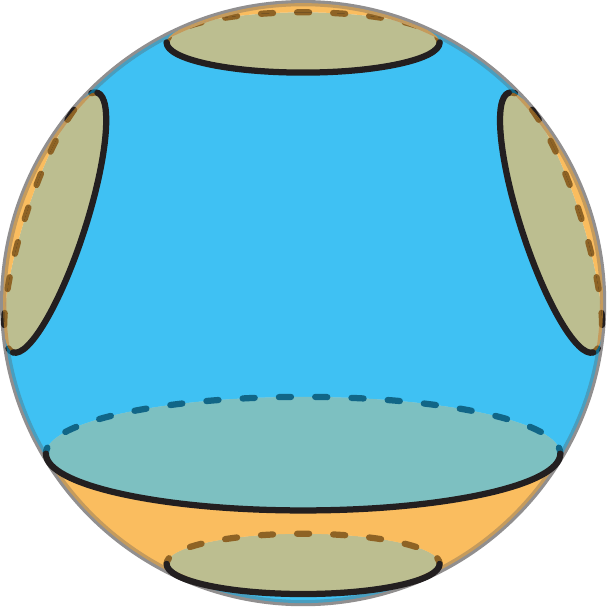}
        \hspace{0.1\textwidth}
       \caption{Examples of 3-balls with a decomposition of their boundary into two pieces $X$ (in lighter orange) and $Y$ (in darker blue). The left two images depict 1-striped and 2-striped balls.}
        \label{fig:colored_3ball}
    \end{figure}
\end{example}

\begin{definition}\label{def:x_embedding}
    Given a colored 3-ball $\mathcal{B}=(B,X,Y)$ and a smooth embedding $f:B\to M$ of $B$ in some 3-manifold $M$, we say $f$ is an {\emph{$X$-embedding}} if $f(B)\cap\partial M=f(X)$, i.e., if $f$ maps $X$ to $\partial M$ and maps the interior of $Y$ to the interior of $M$. In this case we call $f(B)$ a \emph{Murasugi $X$-ball}. Similarly, we say $f$ is a {\emph{$Y$-embedding}}, and $f(B)$ is a \emph{Murasugi $Y$-ball}, if $f(B)\cap\partial M=f(Y)$, i.e., if $f$ maps $Y$ to $\partial M$ and the interior of $X$ to the interior of $M$. 
\end{definition}

Figure \ref{fig:xembedding} illustrates simple examples of both kinds of embeddings. 

\begin{figure}[!ht]
    \centering
    \labellist \small
    \pinlabel{$(B,X,Y)$} at 150 -30
    \pinlabel{$X$-embedding} at 670 -30
    \pinlabel{$Y$-embedding} at 1300 -30
    \endlabellist
    \hspace{0.1\textwidth}
        \raisebox{.035\textwidth}{\includegraphics[width=.15\textwidth]{figures/1_Stripe}}
        \hfill
        \raisebox{.045\textwidth}{\includegraphics[width=.25\textwidth]{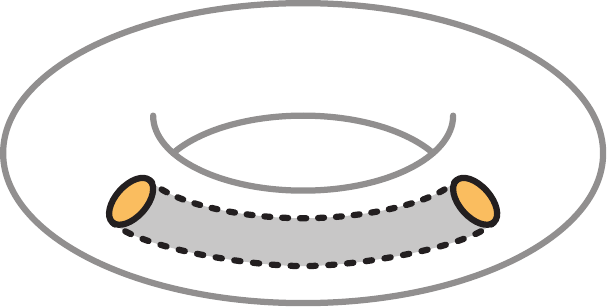}}
        \hfill
        \raisebox{.045\textwidth}{\includegraphics[width=.25\textwidth]{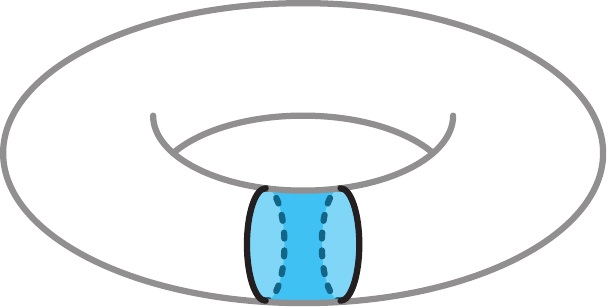}}
        \hspace{0.1\textwidth}
    
    \caption{An illustration of $X$-embedding and a $Y$-embedding of a colored 3-ball $(B,X,Y)$ in a 3-manifold $M$. The corresponding colored 3-ball is illustrated on the left, with $X$ given by a disjoint union of two disks, and $Y$ given by an annulus. Examples of $X$- and $Y$-embeddings are shown in the center and right images, respectively.} 
    \label{fig:xembedding}
\end{figure}

\begin{definition}\label{def:murasugi_x_ball}
    Let $M$ be a 3-manifold smoothly embedded in a closed 4-manifold $W$, fix a colored 3-ball $\mathcal{B}=(B,X,Y)$, and let $f:B\to M$ be an $X$-embedding (respectively, $Y$-embedding). Pushing $f(B)$ off of $M$ in either normal direction in $W$ sweeps out a 4-ball $V$ with $M\cap\mathring{V}=\varnothing$ and $M\cap\partial V=f(B)$. In this case, we say $V$ {\emph{realizes}} $\mathcal{B}$ as an $X$-ball (resp. $Y$-ball) in $M$.
\end{definition}

\begin{definition}[4D Murasugi Sum]\label{def:murasugi_sum}
    Let $M_1$ and $M_2$ be 3-manifolds smoothly embedded in closed 4-manifolds $W_1, W_2$, respectively. Suppose that for some colored 3-ball $\mathcal{B}=(B,X,Y)$, there are 4-balls $V_1, V_2$ embedded in $W_1, W_2$, respectively, realizing $\mathcal{B}$ as an $X$-ball in $M_1$ and as a $Y$-ball in $M_2$. Let $f_1:B\to M_1$ and $f_2:B\to M_2$ be the associated $X$- and $Y$-embeddings.

    Choose an orientation-reversing diffeomorphism $g:\partial V_1\to\partial V_2$ such that $g(f_1(B))=f_2(B)$, $g(f_1(X))=f_2(X)$, and $g(f_1(Y))=f_2(Y)$ (all as oriented submanifolds, with orientation induced from $B$). Obtain the connected sum of $W_1$ and $W_2$ as $(W_1\setminus\mathring{V}_1)\cup_g(W_2\setminus\mathring{V}_2)$. We call the submanifold $N=M_1\cup_g M_2\subset W_1\# W_2$ the {\emph{Murasugi sum}} of $M_1$ and $M_2$ along $f_1,f_2$. 

    Typically, we will refer to the Murasugi sum as instead taking place along the images $f_1(B)$ and $f_2(B)$, with the identification between these balls implicit; we may also write this as $M_1\ast_\mathcal{B} M_2$ or simply $M_1\ast M_2$ when the identification is clear from context. 
\end{definition}

While the spanning solid $M$ constructed in Definition \ref{def:murasugi_sum} above generally depends on the choice of local orientations on $M_i$ near $f_i(B)$, the {knotted surface} $L=\partial M$ is independent of these choices; see Proposition \ref{prop:replumb}.

\begin{definition}\label{def:spine}
    Let $\mathcal{B}=(B,X,Y)$ be a colored 3-ball. An \emph{$X$-spine} of $\mathcal{B}$ is a properly embedded 2-complex $S_X\subset B$ such that $B$ deformation retracts to $S_X$ and $X$ deformation retracts to $\partial B\cap S_X$. A \emph{$Y$-spine} is defined similarly. Furthermore, if $f_1$ and $f_2$ are $X$- and $Y$-embeddings of $(B,X,Y)$ in spanning solids $M_1$ and $M_2$, respectively, then we call $f_1(S_X)$ a {\it spine} of the Murasugi $X$-ball $f_1(B)\subset M_1$, and we call $f_2(S_Y)$ a {\it spine} of the Murasugi $Y$-ball $f_1(B)\subset M_2$.
\end{definition}

\begin{definition}\label{def:pattern}
    If $\mathcal{B}=(B,X,Y)$ is a colored 3-ball and  $X\cap Y$ cuts $\partial B$ into two disks and $n$ annuli, we call $\mathcal{B}$ an \emph{$n$-striped ball}. If we perform a Murasugi sum using an $n$-striped ball, we will describe it as ``performing a Murasugi sum according to an $n$-striped pattern.''
\end{definition}

For example, the leftmost image of Figure \ref{fig:colored_3ball} is a 1-striped ball. We will be particularly interested in the cases where $n\in\{1,2\}$, which are most analogous to the simplest case of the classical Murasugi sum (i.e., along a square embedded in a spanning surface for a knot or link). Note that Murasugi sum along a 0-striped ball is boundary sum.

To be explicit, when $n=1$, we have $X=S^0\times D^2$ and $Y=S^1\times D^1$ (or vice versa). In this case, we can take the spine $S_X$ to be an arc in $B$ with one endpoint in the interior of each component of $X$, and $S_Y$ to be a properly embedded disk in $B$ whose boundary is a core circle of the annulus $Y$. See Figure \ref{fig:12stripe}, left.

When $n=2$, $X$ and $Y$ each consists of a disjoint union of a disk and an annulus. Each of $S_X, S_Y$ is a copy of $D^2\vee I$, where the wedge product takes place along a point in the boundary of $I$ and the interior of $D^2$. For $A\in\{X,Y\}$, the other endpoint of $I$ in $S_A$ lies in the interior of the disk component of $A$, while the boundary of the $D^2$ is a core circle in the annulus component of $A$. See the center image of Figure \ref{fig:12stripe}.

\begin{figure}[!ht]
    \centering
    \hspace{0.1\textwidth}
    \includegraphics[width=.2\textwidth]{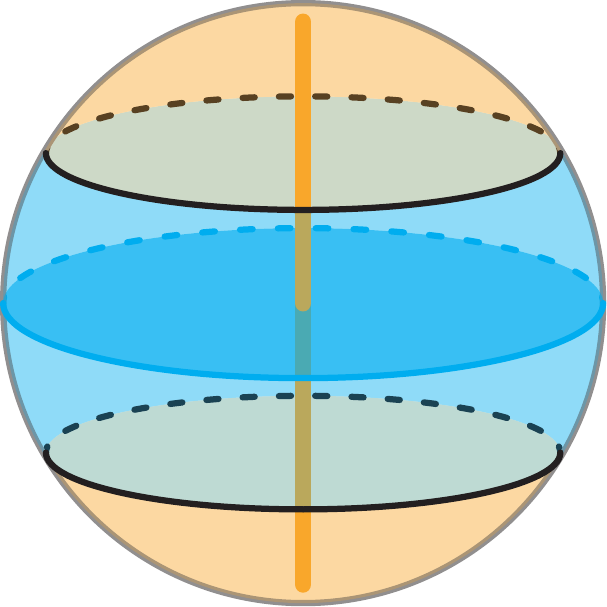}
    \hfill
    \includegraphics[width=.2\textwidth]{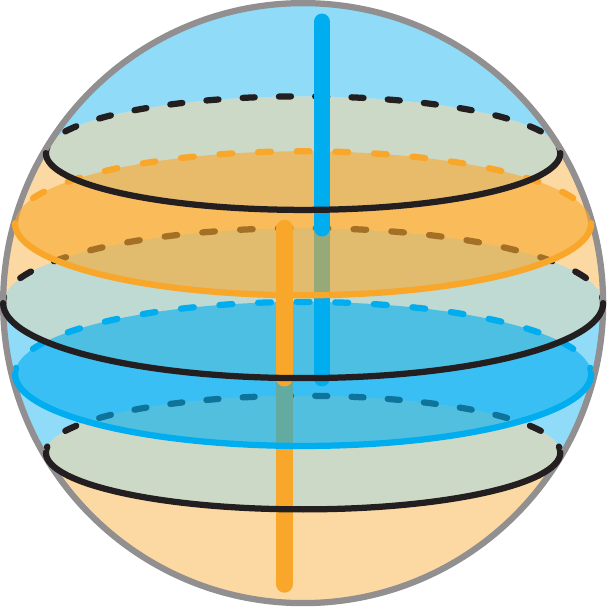}
    \hfill
    \includegraphics[width=.2\textwidth]{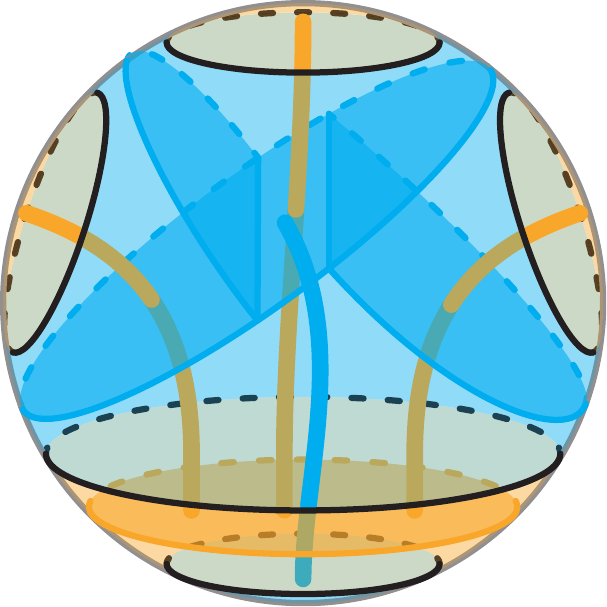}
    \hspace{0.1\textwidth}
    \caption{Spines of the colored balls from Figure \ref{fig:colored_3ball}. Left: in a 1-striped ball $(B,X,Y)$, $S_X$ is an arc and $S_Y$ is a disk. Center: in a 2-striped ball $(B,X,Y)$, $S_X$ and $S_Y$ are both $D^2\vee I$, where the wedge takes place at an endpoint of $I$ and an interior point in $D^2$. 
    Right: a more complicated example. This colored ball is not an $n$-striped ball for any $n$, as there is a region on the boundary that is neither a disk nor an annulus. }
    \label{fig:12stripe}
\end{figure}

\subsection{Examples} We now give some concrete examples of Murasugi sums of knotted surfaces, some of which will be obtained by spinning spanning surfaces for links in $S^3$. 

\begin{definition}\label{def:spun_seifert_surface}
    Let $L\subset S^3$ be a link with a specified component $L_0$, and let $F$ be a spanning surface for $L$. View $S^4=(B^3\times S^1)\cup( S^2\times D^2)$. The \emph{spun spanning surface} $S(F;L_0)\subset S^4$ of $F$ about $L_0$ is the 3-manifold
    \[(S^4,S(F;L_0))=(B^3\times S^1,F^\circ\times S^1)\cup(S^2\times D^2, (F^\circ\cap \partial B^3)\times D^2),\]
    where the gluing map $\partial (B^3\times S^1)\to \partial (S^2\times D^2)$ is the canonical one, and $(B^3,F^\circ)$ is the result of removing a small open ball centered at a point in $L_0$ from $(S^3,F)$. If $L=L_0$ is connected, we will denote the solid $S(F;L)$ simply by $S(F)$. The boundary of this solid is called the \emph{spin} of $L$ with respect to $K$, and is denoted $S(L;K)$ (or simply $S(L)$ if $L$ is connected). 
\end{definition}

Other solids may be obtained by spinning a surface with respect to multiple components, but the surface must be specified more precisely in this case.

\begin{example}
    As a warm-up example, let $L$ be an oriented Hopf link, let $F$ be its annular Seifert surface, and let $M\subset S^4$ be the spin of $F$ with respect to one of the two components. Note that $M$ is diffeomorphic to $(S^1\times D^2)^\circ$.
    Spinning a Murasugi square embedded in $F$ in the manner shown in Figure \ref{fig:spun_example}(a) produces a 3-ball $B$ which naturally lies in $M$ as a $2$-striped ball. 
    
    By performing a Murasugi sum of two copies of $M=S(F; L_0)$, with a copy of $B$ in each copy of $M$, as illustrated in Figure \ref{fig:spun_example}(b), we obtain the spin of the genus one Seifert surface for the trefoil knot. This resulting manifold is diffeomorphic to $(\#^2S^1\times S^2)^\circ$.
    
    Alternatively, performing the Murasugi sum of two copies of $M=S(F;L_0)$ along a (different) 1-stripe pattern as illustrated in Figure \ref{fig:spun_example}(c) yields a spun Seifert surface $S(F';K')$ for a 3-component link $L'$. Namely, $F'$ is a boundary sum of two positive Hopf bands, and $L'$ is a connect sum of two Hopf links. The manifold $S(F';K')$ is diffeomorphic to $(\#^2 S^1\times D^2)^\circ$.

     \begin{figure}[ht]
        \includegraphics[width=0.75\textwidth]{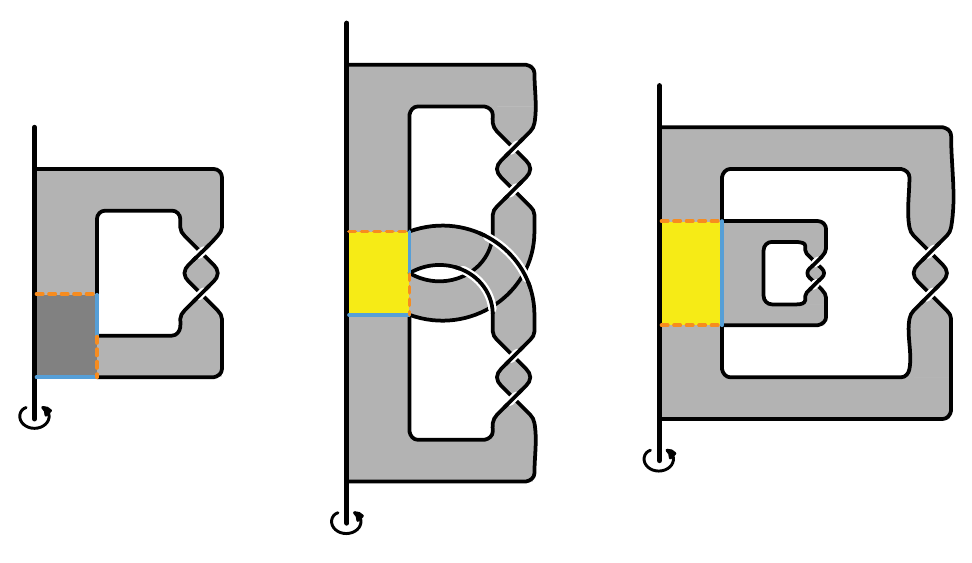}
        \put(-300,-5){(a)}
        \put(-308,73){$B$}
        \put(-180,-5){(b)}
        \put(-60,-5){(c)}
        \caption{In (a), a schematic illustration of the spin of an annulus $F$ bounded by the Hopf link, together with an embedded 2-striped 3-ball $B$ obtaining by spinning a square in $F$. In (b), the Murasugi sum of two spun Hopf annuli yielding the spin of the genus one Seifert surface for the trefoil. In (c), a Murasugi sum of two spun Hopf annuli along a 1-striped pattern. In (b) and (c), a slice of the identified 3-ball is highlighted (yellow).}
        \label{fig:spun_example}
    \end{figure}
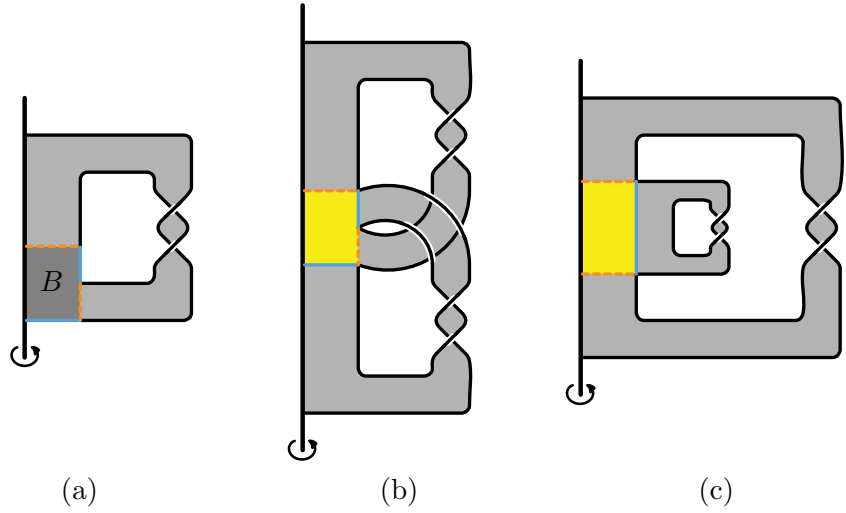
\end{example}

The previous example indicates that Murasugi sum is compatible with spinning in a natural way. 

\begin{proposition}\label{prop:spun_murasugi}
    Suppose that $K_1$ and $K_2$ are classical knots, with spanning surfaces $F_1$ and $F_2$, respectively. Then the spin of any (classical) Murasugi sum of $F_1$ and $F_2$ is given by a Murasugi sum of the spins $S(F_1)$ and $S(F_2)$.
\end{proposition}  

\begin{proof}
    The idea of the proof is essentially contained in the previous example. Suppose that the Murasugi sum $F$ of $F_1$ and $F_2$ takes place along identifications of a $2n$-sided polygon $P$. Build the spin of $F$ by puncturing $S^3$ at one of the vertices of $P$ in $\partial F$. Then $P^\circ\subset F^\circ\subset B^3$ is a polygon whose spin naturally lies in $S(F)$ as a $(2n-2)$-striped 3-ball. In particular, $S(F)$ decomposes as a Murasugi sum of $S(F_1)$ and $S(F_2)$. 
\end{proof}

\begin{remark}
    The preceding proof relies on the assumption that $K_1$ and $K_2$ are knots rather than links of multiple components. A statement analogous to Proposition \ref{prop:spun_murasugi} is true for spins of spanning surfaces for links if one replaces ``any'' with ``some''---the condition would be that the $2n$-gon in each $F_i$ abuts the specified boundary boundary component for spinning--- or if one changes the 4D Murasugi sum in the statement to {\it generalized Murasugi sum}, as in Definition \ref{def:generalized_murasugi_sum}.
\end{remark}

\begin{remark}\label{rem:non-additivity}
    In \cite{thompson1994}, Thompson observed that knot genus is neither sub- nor super-additive under the Murasugi sum operation for non-minimal genus spanning surfaces. In particular, she gives an example of a Murasugi sum of (Seifert surfaces for) two trivial knots yielding a trefoil, and an example of a Murasugi sum of two figure eight knots yielding the trivial knot. In both examples, the gluing region is a square.

    Thus, by spinning her examples and using Proposition \ref{prop:spun_murasugi}, we see that the minimal first Betti number among all Seifert solids for an oriented knotted surface is generally neither sub- nor super-additive under Murasugi sums, even along 2-stripe patterns. 
    \end{remark}

\subsection{Properties of Murasugi sum}\label{subsec:properties} We now record some of the important properties of the Murasugi sum in $S^4$.

\begin{proposition}\label{prop:replumb}
    While the smooth equivalence class of the spanning solid $M$ constructed in Definition \ref{def:murasugi_sum} generally depends on the choice of local orientations on $M_i$ near $f_i(B)$, the knotted surface $L=\partial M$ is independent of these choices (up to smooth equivalence). In other words, $L$ depends only on the diffeomorphism $g$ identifying the $X$-ball $f_1(B)$ in $M_1$ with the $Y$-ball $f_2(B)$ in $M_2$, not the orientation choices.
\end{proposition}

Note that if one reverses the local orientation on only one of the $M_i$ near $f_i(B)$, then $g$ no longer identifies $f_1(B)$ and $f_2(B)$ as oriented manifolds as required in Definition \ref{def:murasugi_sum}. Thus in Proposition \ref{prop:replumb}, the alternative orientation choice we are considering is that of reversing {\emph{both}} the local orientations of $M_1, M_2$ near the plumbing regions.

\begin{proof}
    Let $M$ be the Murasugi sum resulting from one choice of local orientations on the plumbing regions. Recall $M$ meets the splitting $S^3$ in a 3-ball $N$, which is the identified $f_1(B),f_2(B)$. Write $N'=S^3\setminus\mathring N$. Let $M'$ be the Murasugi sum obtained after reversing both orientations (using the same identification $g$). Through isotopy into a bounded $B^4$ (rel.\ boundary), $N'$ is isotopic to $N$ with reversed orientation. Therefore, we may arrange (up to smooth isotopy) $M'=(M\setminus N)\cup N'.$  While it is not clear whether $M',M$ are equivalent, in this description they have identical boundary and hence the {knotted surface} $L$ depends only on $g$, not the choice of local orientations on $M_i$ near the plumbing regions. 
\end{proof}

\begin{remark}\label{prop:singular}
    In Definition \ref{def:colored_ball}, we could have allowed the surfaces $X$ and $Y$ that form the boundary of a colored 3-ball to include wedge singularities, i.e.\ have $X\cap Y$ be a graph instead of disjoint circles. We use the (seemingly) simpler definition because allowing singularities yields no potential new outcomes. In other words, if $M$ is the Murasugi sum of spanning solids $M_1$ and $M_2$ along a colored 3-ball where the intersection $X\cap Y$ is allowed to have wedge singularities, then $M$ can also be obtained as a Murasugi sum of $M_1$ and $M_2$ along a colored 3-ball as in Definition \ref{def:colored_ball} (which has no singularities).
\end{remark}

\begin{proof}[Proof of claim in Remark \ref{prop:singular}]
    Let $\mathcal{B}=(B,X,Y)$ denote the singular colored 3-ball used in the Murasugi sum of $M_1, M_2$. In Figure \ref{fig:singular} we see a local model of the result of Murasugi sum near a wedge point $p$ in $X\cap Y$. Near $p$, we perturb the connected sum sphere to intersect the spanning solid transversely. This choice of perturbation is not unique, as illustrated in the figure. Up to diffeomorphism, we obtain the same spanning solid achieved as a Murasugi sum of $M_1, M_2$ along a colored 3-ball that agrees with $\mathcal{B}$ away from $p$, but near $p$ has its coloring replaced by a resolution of the singularity; the choice of resolution depends on the perturbation of the 3-sphere. Repeating this process for each wedge point in $X\cap Y$, we thus obtain the same spanning solid as a Murasugi sum of $M_1$ and $M_2$ along a nonsingular colored 3-ball.
\end{proof}

\begin{figure}
    \includegraphics[width=110mm]{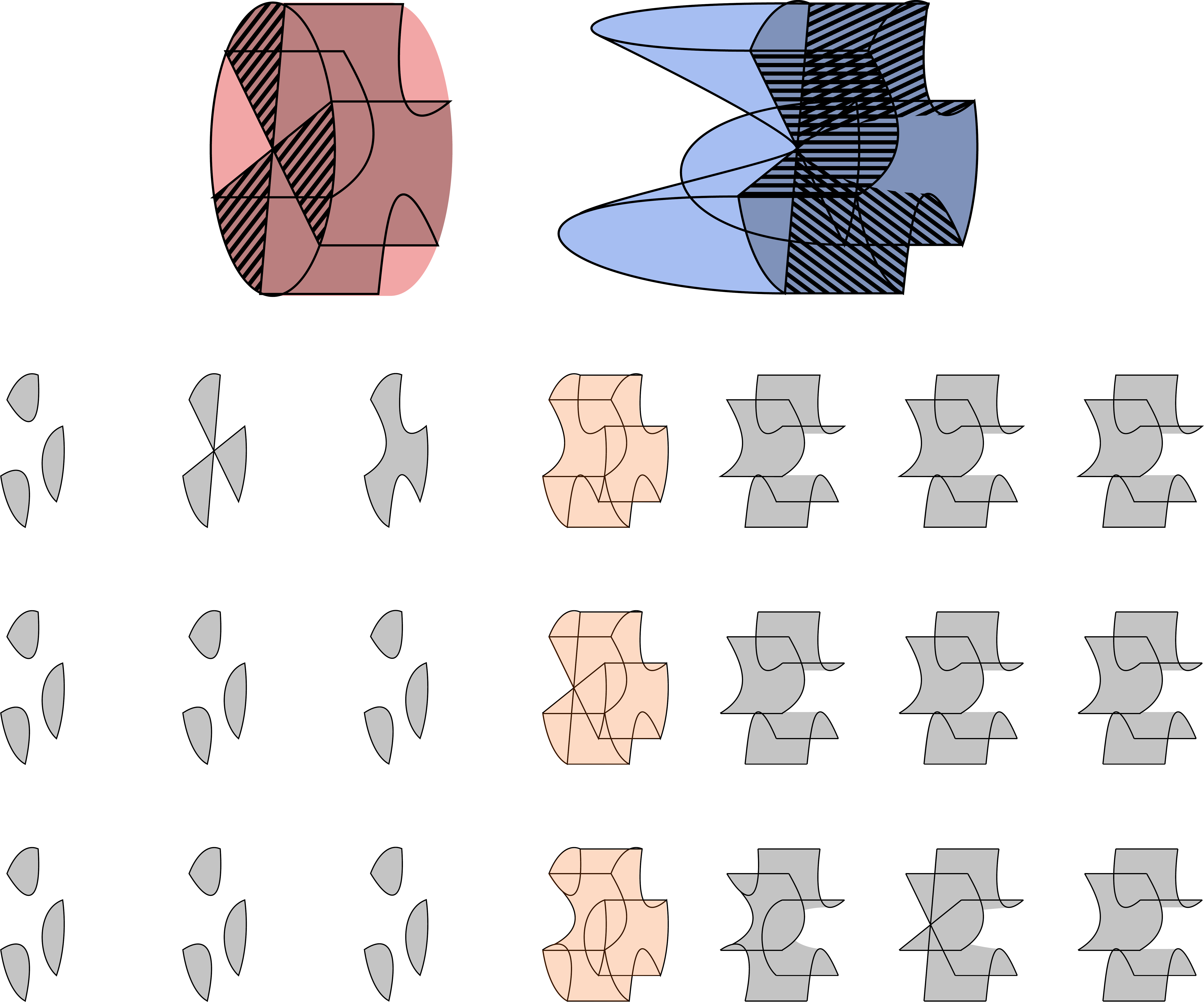}
    \caption{Top two images: Spanning solids $M_1, M_2$ near a singular point of $X\cap Y$, where $(B,X, Y)$ is a colored 3-ball $X$-embedded in $M_1$ and $Y$-embedded in $M_2$. We use diagonal lines to shade the portions of $\partial B$ in $\partial M_1$ and $\partial M_2$ -- note these regions are complementary. Bottom: In the middle row, we Murasugi sum $M_1, M_2$ along $(B,X,Y)$. We draw cross-sections of the resulting spanning solid, with $B$ appearing as one cross-section in the center. In the top and bottom rows, we perturb the embedding to find the same spanning solid obtained as a Murasugi sum of $M_1, M_2$ along a colored 3-ball with one fewer point of singularity.}
    \label{fig:singular}
\end{figure}

One of the most important features of the classical Murasugi sum operation, first proved by Stallings \cite{stallings1978}, is that it preserves the property of being fibered. In other words, if $F_1$ and $F_2$ are fiber surfaces in $S^3$, then any Murasugi sum of $F_1$ and $F_2$ is also a fiber surface in $S^3$. Later, Gabai \cite{gabai_murasugisumnatural} proved the converse (if a sum is fibered, so are both summands) and coined the name \textit{Murasugi sum} for the generalization to $2n$-gons as in Definition \ref{def:classicalmurasugisum}. He showed that this operation behaves naturally with respect to a number of useful features of fibrations of link complements, e.g.\ the monodromy. 

We note that Lines (\cite{lines1985}, \cite{lines1986}, \cite{lines1987}) also considered versions of plumbings of {codimension-2} submanifolds in ambient dimension strictly greater than four. For some history of the various notions of plumbing and the corresponding fibration theorems, see \cite[\S2]{ozbagci_popescupampu}.

Ozbagci and Popescu-Pampu \cite{ozbagci_popescupampu} prove the following very general result for summing pages of an open book along an embedded submanifold with no restriction on dimension. 

\begin{theorem}[{\cite[Theorem 9.3]{ozbagci_popescupampu}}]\label{thm:ozbagcic_pompescupampu}
   Let $(W_i, M_i, P)_{i = 1,2}$ be two summable patched Seifert hypersurfaces
   which are pages of open books on the closed manifolds $W_i$.
   Then the Seifert hypersurface associated to the sum
   $(W_1, M_1) \biguplus^P  (W_2, M_2)$
   is again a page of an open book. Moreover, the geometric monodromy
   of the resulting open book is the composition $\phi_1 \circ \phi_2$ of the monodromies
   of the initial open books. Here $\phi_i : M_i \to M_i$
   is extended to $M_1 \biguplus^P M_2$ by the identity on
   $(M_1 \biguplus^P M_2)\setminus M_i$.
\end{theorem}

Here, the pair $(W_i, M_i, P)_{i = 1,2}$ are called \emph{summable patched Seifert hypersurfaces} if $M_i$ are Seifert hypersurfaces whose co-orientations extend that of the \emph{patch} $P$ (an identified subset of each $M_i$ along with the gluing data). We refer the reader to \cite[\S9]{ozbagci_popescupampu} for more details. As a consequence, we obtain the following corollary: the Murasugi sum of (fiber solids for) two fibered knotted surfaces is fibered. 

\begin{corollary}[{Corollary of \cite{ozbagci_popescupampu}, Theorem 9.3}]
    \label{cor:fibered_murasugi}
    Suppose that $S_1$ and $S_2\subset S^4$ are fibered knotted surfaces embedded in $S^4$ with fibers $M_1$ and $M_2$, respectively. Then the boundary $S_1\ast S_2$ of any Murasugi sum of $M_1$ and $M_2$ is fibered, with fiber $M_1\ast M_2$. Moreover, the monodromy of $M_1\ast M_2$ is given by $\phi_1\circ \phi_2$, where $\phi_1$ and $\phi_2$ are the monodromies of $M_1$ and $M_2$, respectively. Here, $\phi_i$ is extended by the identity to $M_1\ast M_2$. 
\end{corollary}

\begin{remark}\label{rem:general_murasugi_sums}
    While have only discussed performing Murasugi sums along a 3-ball, we later generalize Definitions \ref{def:colored_ball} and \ref{def:murasugi_sum} to other convenient 3-dimensional solids in Definition \ref{def:generalized_murasugi_sum}. We note here, however, that if we perform such a ``generalized Murasugi sum'' of solids $M_1$ and $M_2$ with boundaries $\partial M_i=S_i$ along a 3-manifold $Y$ to obtain a solid $M$ with boundary $S=\partial M$, then we have
    \begin{align*}
    \chi(S) &= \chi(S_1\setminus\partial Y) + \chi(S_2\setminus \partial Y) \\
    &= \chi(S_1)+\chi(S_2) - (\chi(S_1\cap Y)+\chi(S_2\cap Y)) \\
    &= \chi(S_1)+\chi(S_2)-\chi(\partial Y).
    \end{align*}

    Thus, if $\partial Y$ is (connected but) not a 2-sphere, the resulting knotted surface $S$ cannot be fibered in $S^4$, since any knotted surface in $S^4$ which is fibered must have Euler characteristic equal to 2. In particular, this will mean that no strictly generalized Murasugi sum will preserve the property of being fibered.
\end{remark}

The next two lemmas allow one to simplify certain Murasugi sum descriptions. (We suspect that a stronger version of Lemma \ref{lem:knotted_arc} is true, but we have no occasion to use it---see Question \ref{Q:knotted_arc}.) They imply, for example, that if $M_i$, $i=1,2$, are boundary-irreducible spanning solids (meaning that every properly embedded disk in $M_i$ is boundary-parallel), then the only Murasugi sums $M_1*M_2$ are boundary sums $M_1\natural M_2$.

\begin{lemma}\label{lem:knotted_arc}
     Let $M=M_1*M_2$ be a Murasugi sum of spanning solids along a colored 3-ball $\mathcal{B}=(B,X,Y)$, let $f_1:B\to M_1$ and $f_2:B\to M_2$ be the associated $X$- and $Y$-embeddings, and let $S_X\subset M_1,S_Y\subset M_2$ be the $X$- and $Y$-spines of the plumbing. Let $\alpha$ be the union of the arcs in $S_X$, so that $S_X\cut \alpha$ is comprised of the closures of the 2-cells of $S_X$.  Now suppose $Z\subset M_1$ is a 3-ball containing $X$, and $\alpha'\subset Z$ is another union of arcs 
     that is obtained from $\alpha$ by a homotopy rel.\ boundary in $Z\cut(S_X\cut\alpha)$---in other words, $\alpha'$ is obtained from $\alpha$ by locally (in $Z$) changing crossings (including self-crossings) among the arcs of $\alpha$. Let $S_{X'}=(S_X\cut\alpha)\cup \alpha'$, take $X'=\nu S_{X'}$ such that $X'\cap\partial M_1=X\cap\partial M_1$, and choose an embedding $f_1':B\to M_1$ whose restriction to $f_1^{-1}(\partial M_1)$ matches that of $f_1$. Then $M$ is equivalent to the Murasugi sum of $M_1$ and $M_2$ along the colored 3-ball $\mathcal{B}'=(B,X',Y)$. 
\end{lemma}

\begin{proof}
First, note that the Murasugi sum of $M_1$ and $M_2$ along $\mathcal{B}$ identifies $X$ and $Y$ by the diffeomorphism $f_2\circ f_1^{-1}$, the Murasugi sum along $\mathcal{B'}$ identifies $X$ and $Y$ by $f_2\circ (f_1')^{-1}$, and these two diffeomorphisms have identical restrictions to $X\cap\partial M_1=X'\cap\partial M_1$. With this in mind, we will glue together $M_1$ and $M_2$ in a way that yields a solid $M'$ that, we will see, is equivalent to both $M_1*_\mathcal{B}M_2$ and $M_1*_{\mathcal{B}'}M_2$.
    
    Embed $M_1$ in $B^4_1$ such that $M_1\cap \partial B^4_1=Z$. Choose an orientation-reversing diffeomorphism $g:\partial B^4_1\to\partial B^4_2$ such that $g$ matches $f_2\circ f_1^{-1}$ and $f_2\circ (f_1')^{-1}$ when restricted to $X\cap\partial M_1\to \partial Y$. Construct a solid $M'=M_1\cup_g M_2$ in $S^4=B^4_1\cup_g B^4_2$.  We will show that $M'$ is isotopic to the Murasugi sums $M_1*_{\mathcal{B}} M_2$ and $M_1*_{\mathcal{B'}} M_2$.

Observe that pushing $Z\cut X$ slightly into the interior of $B^4_1$, while fixing the rest of $M'$, transforms $M'$ into a 3-manifold intersecting
\begin{itemize}
    \item $B^4_1$ in a copy of $M_1$,
    \item $B^4_2$ in a copy of $M_2$,
    \item $\partial B^4_i$ in a 3-ball identified with $f_1(B)\subset M_1$ and $f_2(B)\subset M_2$.
\end{itemize}
Thus, pushing $Z\cut X$ slightly into the interior of $B^4_1$ yields an  an isotopy between $M'$ and $M_1*_\mathcal{B}M_2$.

On the other hand, pushing $Z\cut X'$ into the interior of $B^4_1$ similarly transforms $M'$ into a Murasugi sum $M''$ intersecting $B^4_1$ in a copy of $M_1$ and $B^4_2$ in $(M_2\setminus f_2(B))\cup_{\overline{Y\cap{\mathring{M}_2}}} B^3$, where since $f_1^{-1}\circ f_2$ and $(f_1')^{-1}\circ f_2$ agree on $\overline{Y\cap{\mathring{M}_2}}$, we again have a copy of $M_2$ in $B^4_2$. The intersection of $M''$ with $\partial B^4_i$ is a 3-ball identified with $f_1'(B)$ in $M_1$ and $f_2(B)$ in $Y$, so $M''$ is the Murasugi sum $M_1*_{\mathcal{B'}} M_2$.

To summarize, pushing $Z\cut X\subset M'$ into the interior of $B^4_1$ yields an isotopy from $M'$ to $M_1*_{\mathcal{B}} M_2$, while pushing $Z\cut X'\subset M'$ into the interior of $B^4_1$ yields an isotopy from $M'$ to $M_1*_{\mathcal{B'}} M_2$. We conclude the Murasugi sums $M_1*_{\mathcal{B}} M_2$ and $M_1*_{\mathcal{B'}} M_2$ are equivalent.
\end{proof}

\begin{question}\label{Q:knotted_arc}
    Does Lemma \ref{lem:knotted_arc} hold without the assumption that the homotopy from $\alpha$ to $\alpha'$ is {\it local}, i.e., contained within a 3-ball $Z$?
\end{question}

\begin{lemma}\label{lem:boundary_parallel_connect_sum}
  Let $M=M_1*M_2$ be a Murasugi sum of spanning solids along a colored 3-ball $(B,X,Y)$, and let $S_X\subset M_1,S_Y\subset M_2$ be the $X$- and $Y$-spines of the plumbing. Suppose there exists one of the following.
  \begin{enumerate}[label=(\arabic*)]
      \item\label{case:ball} A properly embedded 3-ball $V\subset M_1\cut S_X$ whose boundary consists of a disk $D\subset S_X$ and a disk $D'\subset\partial M_1$,
      \item\label{case:disk} A properly embedded disk $D\subset M_1\cut S_X$ whose boundary consists of an arc $\alpha\subset S_X$ and an arc $\beta\subset\partial M_1$ whose endpoints lie on distinct components of $\partial S_X$ (Figure \ref{fig:Monkey_Reduce_Solid}).
    \end{enumerate}
     Then $M$ can be constructed as a simpler plumbing of $M_1$ and $M_2$. Namely, one can construct $M=M_1*M_2$ along a colored 3-ball $(B',X',Y')$ with $|X'\cap Y'|<|X\cap Y|$.
\end{lemma}

\begin{proof}
   Note first that $|X\cap Y|=|X|+|Y|-1$. In Case \ref{case:ball}, let $B'=B\cup V$; in Case \ref{case:disk}, let $B'=B\cup\nu(D)$. Then let $X'=\partial B'\cap\partial M_1$ and $Y'=\partial M_2$. In Case \ref{case:ball}, $X'$ is the result of attaching the disk $D'$ to $X$, so $|\partial X'|=|\partial X|$ and $|\partial Y'|<|\partial Y|$. In Case \ref{case:disk}, $X'$ is the result of attaching a band to $X$ along two distinct boundary components, so $|\partial X'|<|\partial X|$; meanwhile, $|\partial Y'|=|\partial Y|$. In either case, we conclude that $|X'\cap Y'|<|X\cap Y|$. 
\end{proof}

Figure \ref{fig:Monkey_Reduce_Solid} shows the idea behind the proof of Lemma \ref{lem:boundary_parallel_connect_sum}, and Figure \ref{fig:Monkey_Reduce} shows the analogous idea for Murasugi sums of spanning surfaces in $S^3$.

\begin{figure}[!ht]
    \centering
    \includegraphics[width=.85\textwidth]{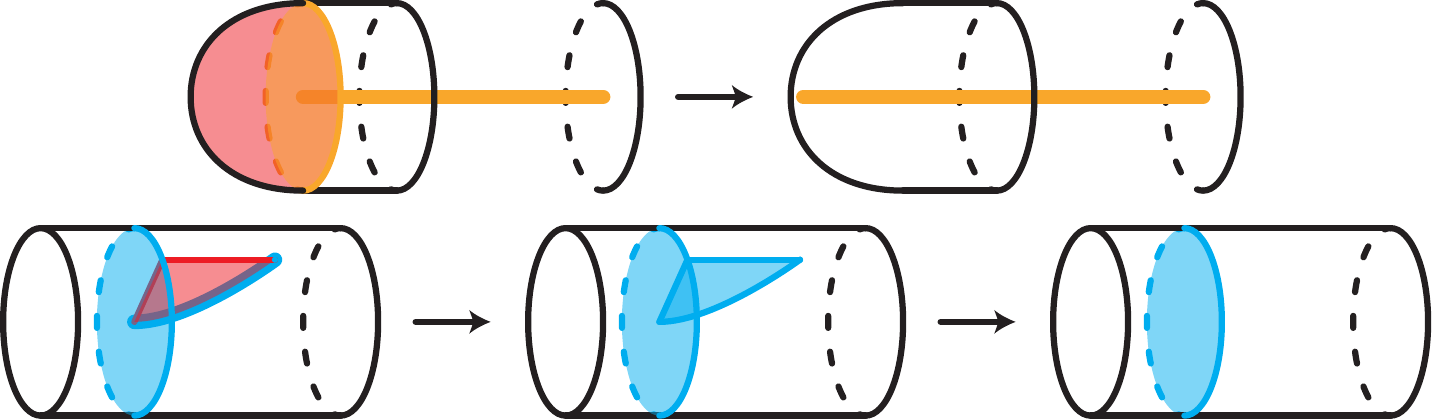}
    \caption{A schematic which gives the main idea of Lemma \ref{lem:boundary_parallel_connect_sum}---a kind of compressing disk or ball for the spine can be used to reduce the number of components in the boundary of a colored ball used for the Murasugi sum.}
    \label{fig:Monkey_Reduce_Solid}
\end{figure}

\begin{figure}[!ht]
    \centering
    \includegraphics[width=0.95\textwidth]{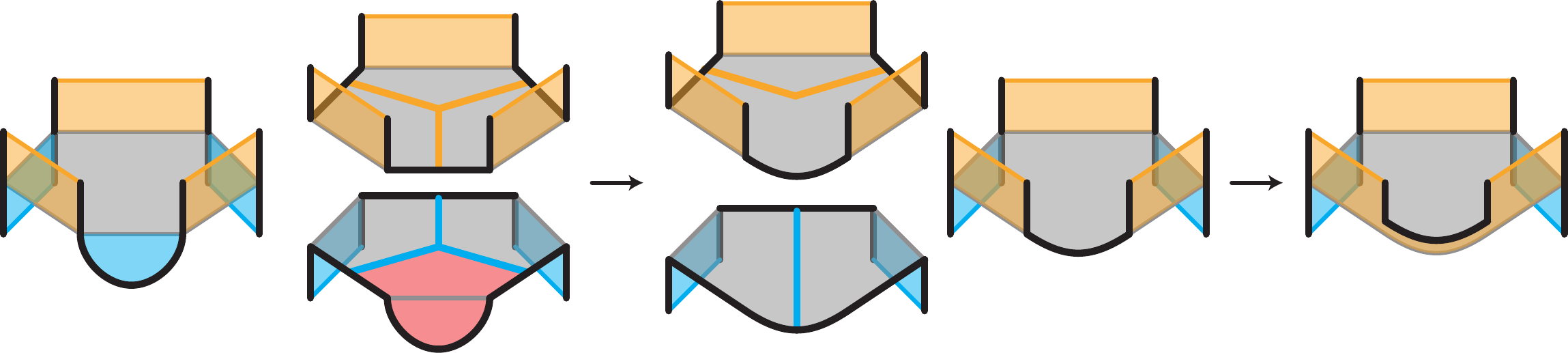}
    \caption{The idea of Lemma \ref{lem:boundary_parallel_connect_sum}, illustrated one dimension lower. The first image shows a surface built as a Murasugi sum, along with the corresponding ``spines.'' The red/shaded disk in the lower surface can be used to simplify the gluing region from a 6-sided polygon to a 4-sided polygon, as illustrated on the right. }
    \label{fig:Monkey_Reduce}
\end{figure}

\begin{restatable}{proposition}{prop:irreducible_sum}
\label{prop:irreducible_sum}
    If $M_1$ and $M_2$ are boundary-irreducible spanning solids (e.g., irreducible with possibly disconnected 2-sphere boundaries), then any Murasugi sum $M$ of $M_1$ and $M_2$ is a boundary sum. Moreover, any boundary-irreducible spanning solid that is a Murasugi sum is a boundary sum. 
\end{restatable}

\begin{proof}
    Let $S_X\subset M_1, S_Y\subset M_2$ be the $X$- and $Y$-spines of the plumbing. Then, since each $M_i$ is irreducible every 2-cell in $S_X,S_Y$ whose boundary is in $M_i$ is boundary parallel in $M_i$, albeit through a 3-ball whose interior might contain other arcs (or disks) in $S_X,S_Y$. Such arcs might be knotted or tangled, but Lemma \ref{lem:knotted_arc} allows us to unknot and untangle all these arcs without changing the resulting solid $M_1*M_2$ up to equivalence.

    Using Lemma \ref{lem:boundary_parallel_connect_sum} inductively, we then find that $M$ is obtained from a plumbing of $M_1, M_2$ using a colored ball whose $X$- and $Y$-spines are each 1-complexes; since they are dual we find 
    that the plumbing is along a colored ball $(B,X,Y)$ with $X\cong Y\cong D^2$. In other words, $M$ is equivalent to $M_1\natural M_2$.

    The last claim follows by similar reasoning. The key point is that for any component $X_0$ of $X$ that is not a disk and any circle $\gamma\subset \partial X_0$, there is a disk $Z\subset S_X$ and an isotopy of $(Z,\partial Z)$ through $(B,X_0)$ after which $\partial Z=\gamma$.
\end{proof}

\begin{remark} \label{prop:irreducible_fiber}
    There are many fibered knotted 2-spheres whose fiber is irreducible (e.g., the Cappell-Shaneson spheres constructed in \cite{cappell_shaneson76}); the the only way to take a Murasugi sum of such fiber solids is via boundary sum.
    Nevertheless, many spanning solids, such as the arborescent ones constructed in the next section, contain interesting disks; these solids may be plumbed onto boundary-irreducible solids to produce a variety of interesting spanning solids for knotted surfaces.   
\end{remark}  

\section{Arborescent {knotted surface}s}\label{sec:arborescent}

Conway \cite{conway1970} introduced arborescent spanning surfaces as a natural class of surfaces obtained by successively plumbing essential unknotted annuli and M\"obius bands along disjoint disks.
Each such surface is described by a planar tree with integer labels between consecutive edges at each vertex. These labels indicate half-twists in bands; see Figure \ref{fig:classical_arborescent_example}.

An arborescent surface constructed from a tree with at most two leaves is a spanning surface for a 2-bridge link. In fact, Hatcher--Thurston \cite{hatcher_thurston1985} prove that every essential spanning surface $F$ for a 2-bridge link $L$ is an arborescent surface (from a tree with at most two leaves). Spinning such a pair $F,L\subset S^3$ yields a spanning solid $M$ for a knotted surface $K\subset S^4$, where $K$ is diffeomorphic to either $S^2$ or $S^2\sqcup T^2$; the solid $M$ is a Murasugi sum of the spins of the bands that are plumbed together to form $F$. 

\subsection{Arborescent links} First, we remind the reader of the formal definition in the classical dimension. Recall that a \emph{Conway sphere} is a 2-sphere meeting a knot or link in four points.

\begin{definition}\label{def:arborescent_classical}
    A link $L$ is called \emph{arborescent} if there is a disjoint family of Conway spheres $\{Q_1,\dots,Q_n\}$ such that if $N$ is the closure of any component of $S^3\setminus(\bigcup_{i=1}^nQ_i)$, then up to isotopy in $S^3$, the pair $(N,N\cap K)$ takes the form of Figure \ref{fig:classical_arborescent_conway_sphere} below.
\end{definition}

\begin{figure}[h]
    \includegraphics[width=0.5\textwidth]{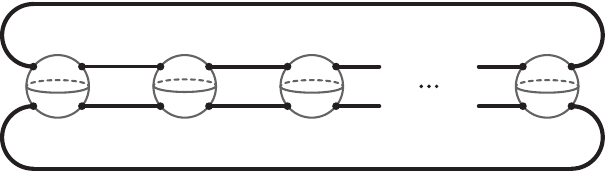}
    \caption{An ``atomic'' part of an arborescent link, obtained by cutting along a family of Conway spheres.}
    \label{fig:classical_arborescent_conway_sphere}
\end{figure}

\begin{definition}\label{def:weighted_tree_classical}
    A \emph{weighted tree} is a finite tree embedded in the plane along with, for each vertex $v$, an assignment of an integer $w(e_1,e_2)$ to each pair of of edges $e_1$ and $e_2$ adjacent at $v$. We call $w(e_1,e_2)$ a ``weight." In diagrams, we write this weight between $e_1$ and $e_2$ near the vertex $v$, and neglect to write a weight which is equal to zero. Note that for a vertex with valence $k$, we assign $k$ weights; any leaf is assigned one weight. 
\end{definition}

Starting from a weighted tree, the following construction produces an arborescent link (see \cite{montesinos1973}, \cite{gabai1986genera}); these links always bound a surface constructed in such a way. 

\begin{construction}\label{cons:arborescent_classical}
    ${ }$
    
    \begin{enumerate}[label=(\arabic*)]
        \item For a vertex $v$ adjoining $k$ edges $\{e_1,\dots,e_k\}$, we associate a band $b_v$ with $k$ labeled ``available'' plumbing squares. Between the patches corresponding to $e_i$ and $e_{i+1}$, we introduce $w(e_i,e_{i+1})$ half twists.
        \item If two vertices $v$ and $w$ share an edge $e$, we plumb the bands $b_v$ and $b_w$ along the plumbing squares corresponding to $e$ in each. 
    \end{enumerate}
\end{construction}

An example of a simple weighted tree and its corresponding arborescent link are illustrated in Figure \ref{fig:classical_arborescent_example}. 
Note that although the equivalence class of spanning surfaces produced by this construction may depend on the choice of plumbing direction \cite{hatcher_thurston1985}, the boundary knot does not; this follows from the classical version of Proposition \ref{prop:replumb}. For more details on arborescent links, we refer the reader to \cite{bonahon_siebenmann2010}. 

\begin{figure}[ht]
    \includegraphics[width=0.85\textwidth]{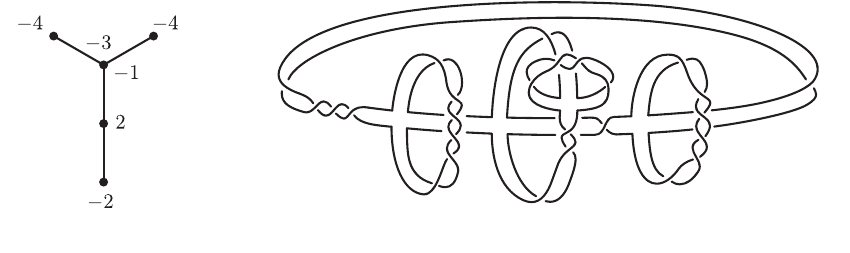}
    \put(-329,0){(a)}
    \put(-134,0){(b)}
    \caption{An example of an arborescent link. In (a), a weighted graph, and in (b), the corresponding plumbed surface.}
    \label{fig:classical_arborescent_example}
\end{figure}

\subsection{Arborescent knotted surfaces} We will now define arborescence for knotted surfaces, by performing repeated Murasugi sum of simple atomic pieces. We will use spun Seifert surfaces for the $T(2,n)$ torus links as the basic building blocks. 

\begin{definition}\label{def:standard_objects}
    For $n\geq0$ even, let $Y_n$ be the orientable solid obtained by spinning the annulus spanning the $(2,n)$-torus link about one boundary component. For $n>0$ odd, let $Y_n$ be the non-orientable solid obtained by spinning the M\"obius band spanning the $(2,n)$-torus knot. Both cases are illustrated in Figure \ref{fig:buildingblock}. Note that $Y_n$ is diffeomorphic to $(S^1\times D^2)^\circ$ when $n$ is even, and $(S^1\ttimes S^2)^\circ$ when $n$ is odd. The solid $Y_i$ is called a \emph{standard spun annulus} or \emph{standard spun M\"obius band}.
\end{definition}

\begin{remark}
   The spin of a surface $F$ is equivalent to the spin of its mirror $\overline{F}$. For this reason, we only consider the case $n\ge 0$ in Definition \ref{def:standard_objects}. 
\end{remark}

\begin{figure}[ht]
    \centering
    \includegraphics[width=.125\textwidth]{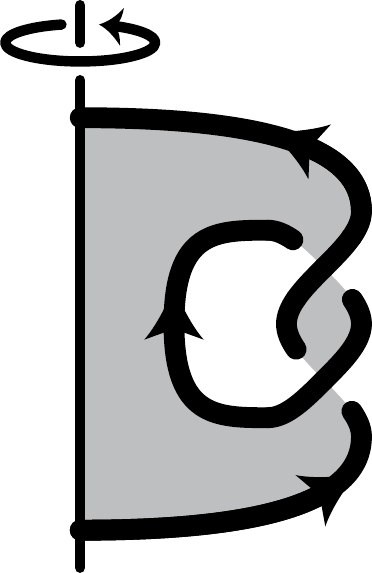}\hspace{.15\textwidth}
    \includegraphics[width=.125\textwidth]{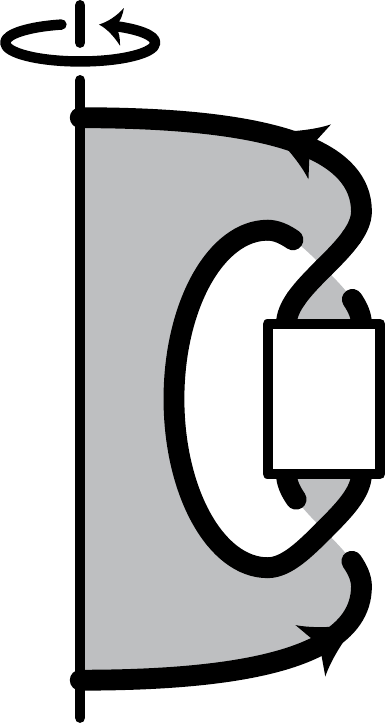}\hspace{.15\textwidth}
    \includegraphics[width=.125\textwidth]{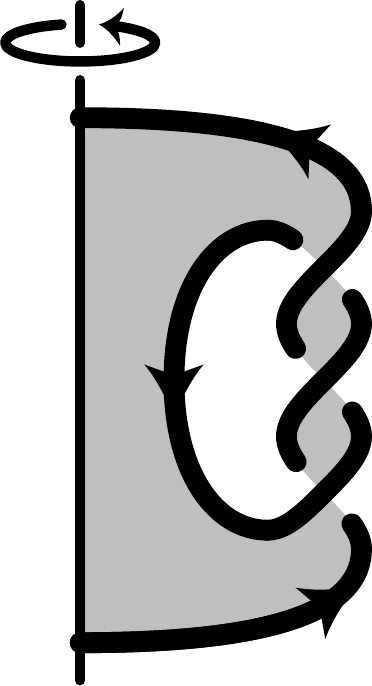}
    \put(-134,43){$n$}
    \caption{Illustrations of the standard spun annulus or M\"obius band $Y_n$. Left: the spin $Y_2$ of the annulus spanning the negative Hopf link. Center: A schematic illustration of $Y_n$. Right: the spin $Y_3$ of the M\"obius band spanning the right-handed trefoil.}
    \label{fig:buildingblock}
\end{figure}

In Figure \ref{fig:buildingblock}, we also indicate orientations of the boundaries of the spun annuli and M\"obius bands. For $n$ even, $Y_n$ inherits an orientation induced by the orientation of the annulus. For $n$ odd, a collar neighborhood of $\partial Y_n$ in $Y_n$ inherits an orientation from the specified orientation of the boundary of the M\"obius band. 

\begin{definition}\label{def:standard_subsets}
    Let $n>0$, and let $Y_n$ be the corresponding standard spun annulus or M\"obius band. 
    
    \begin{enumerate}[label=(\arabic*)]
        \item If $n$ is even, the \emph{standard disk} $D_n\subset Y_n$ is a disk fiber $\{\ast\}\times D^2$ of $(S^1\times D^2)^\circ$. If $n$ is odd, the standard disk is a punctured sphere fiber $(\{\ast\}\times S^2)^\circ$ of $(S^1\ttimes S^2)^\circ$. In either case, $D_n$ is the unique properly embedded nonseparating disk in $Y_n$ up to isotopy.
        \item If $n$ is even, the \emph{standard arc} $A_n$ is the unique (up to isotopy, using the light bulb trick) properly embedded arc whose endpoints lie on opposite components of $\partial Y_n$. If $n$ is odd, the \emph{standard arc} $A_n$ is an arc of the form $(S^1\ttimes \{\ast\})^\circ$, i.e., the unique (up to isotopy) unknotted properly embedded arc representing a generator of $H_1(Y_n,\partial Y_n)$.
         \item For any $n$, a \emph{standard disk-arc pair} $P_n^\pm$ is the union of $D_n$ with an unknotted, properly embedded, non-boundary-parallel arc $a_n$ in $Y_n\setminus D_n$ with one endpoint on $D_n$ and one on $\partial Y_n$. There are two such configurations, depending on the the local behavior of $a_n$ near $D_n$. We define $P_n^+$ (respectively, $P_n^-$) to be the configuration in which $a_n$ leaves $D_n$ in the same direction as a positively (respectively, negatively) orientated generator for $H_1(Y_n;\mathbb{Z})$.
    \end{enumerate}
    See Figures \ref{fig:standard_subsets_even} and \ref{fig:standard_subsets_odd} for illustrations of these subspaces.
\end{definition}

\begin{figure}[ht]
    \centering
    \includegraphics[width=0.2\textwidth]{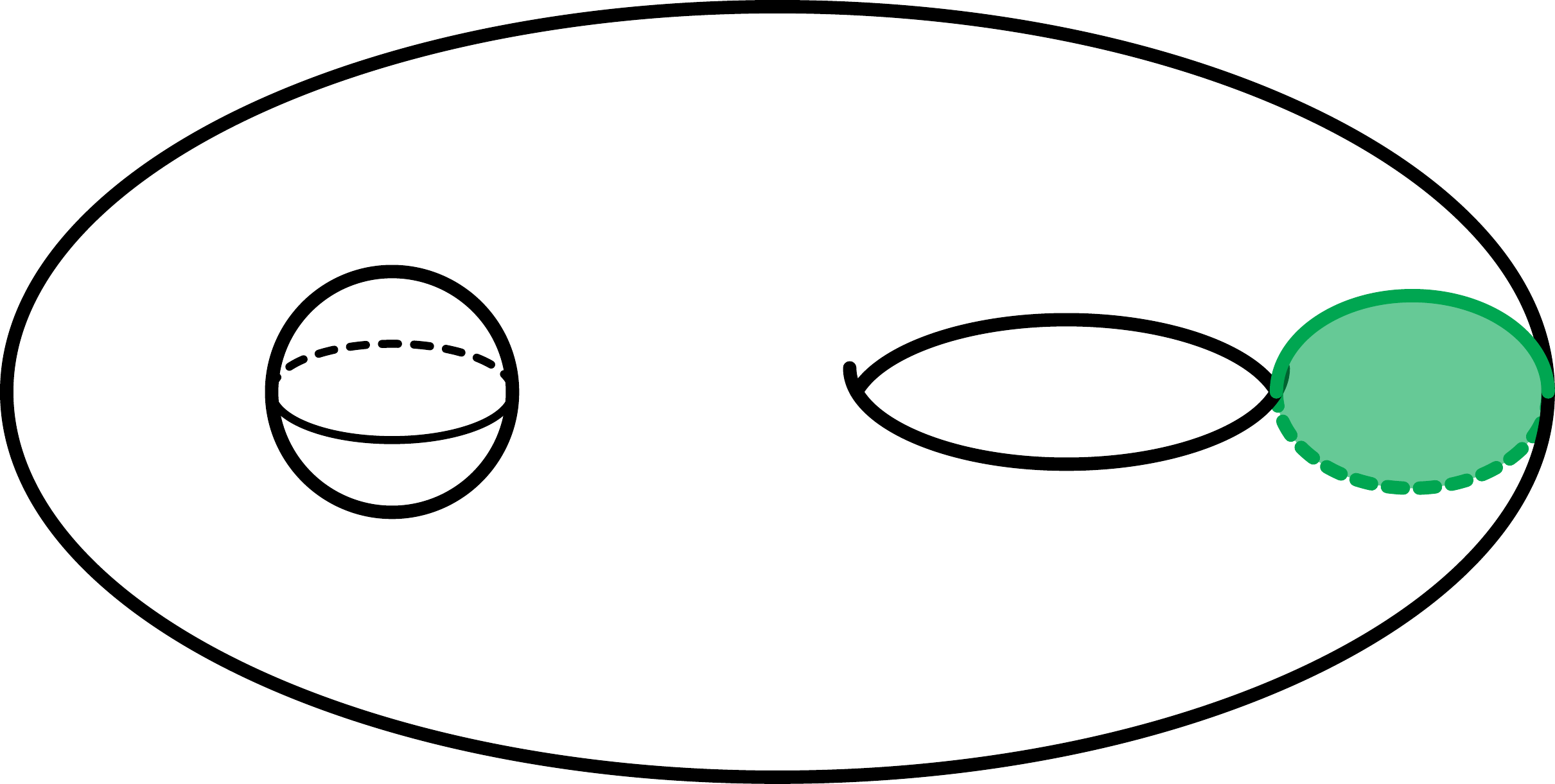}\hspace{.1\textwidth}
    \includegraphics[width=0.2\textwidth]{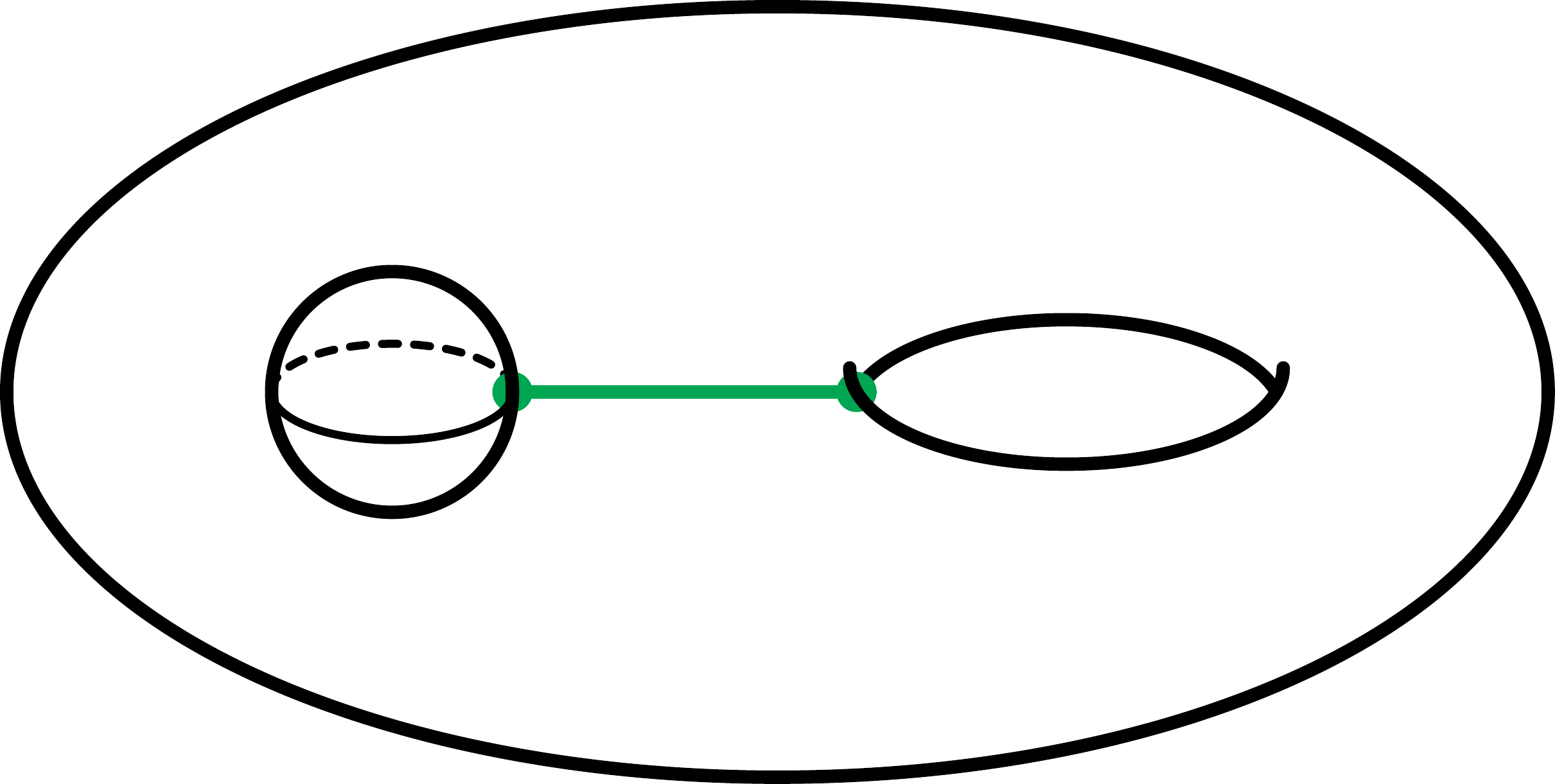}\hspace{.1\textwidth}
    \includegraphics[width=0.2\textwidth]{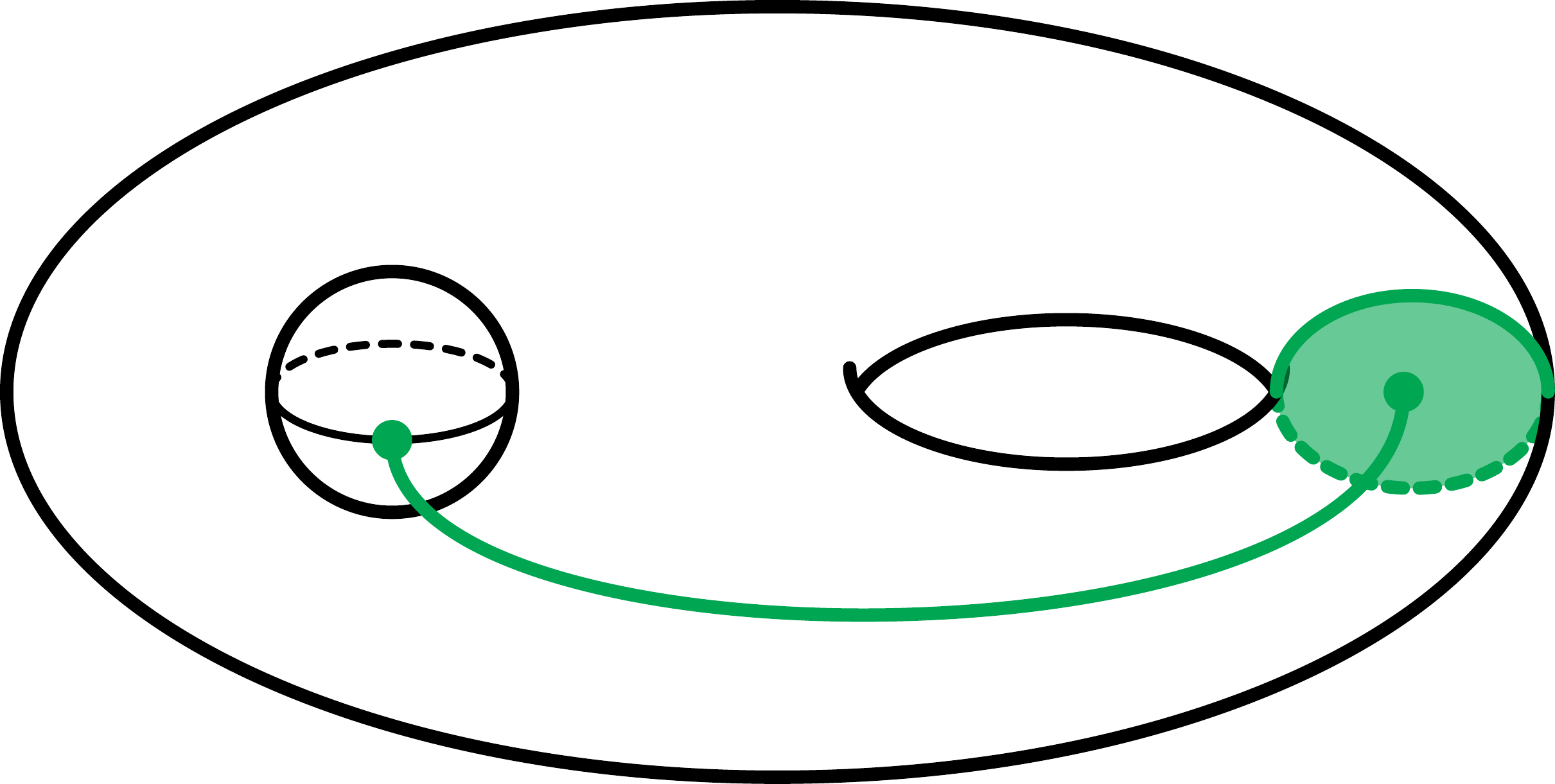}
    \caption{An illustration of the standard disk $D_n$ (left), arc $A_n$(center), and one of the disk-arc pairs $P_n^\pm$ (right) in $Y_n\cong (S^1\times D^2)^\circ$ in the case that $n$ is even. Here, $Y_n$ is drawn as a punctured solid torus.}
    \label{fig:standard_subsets_even}
\end{figure}

\begin{figure}[ht]
    \centering
    \includegraphics[width=0.2\textwidth]{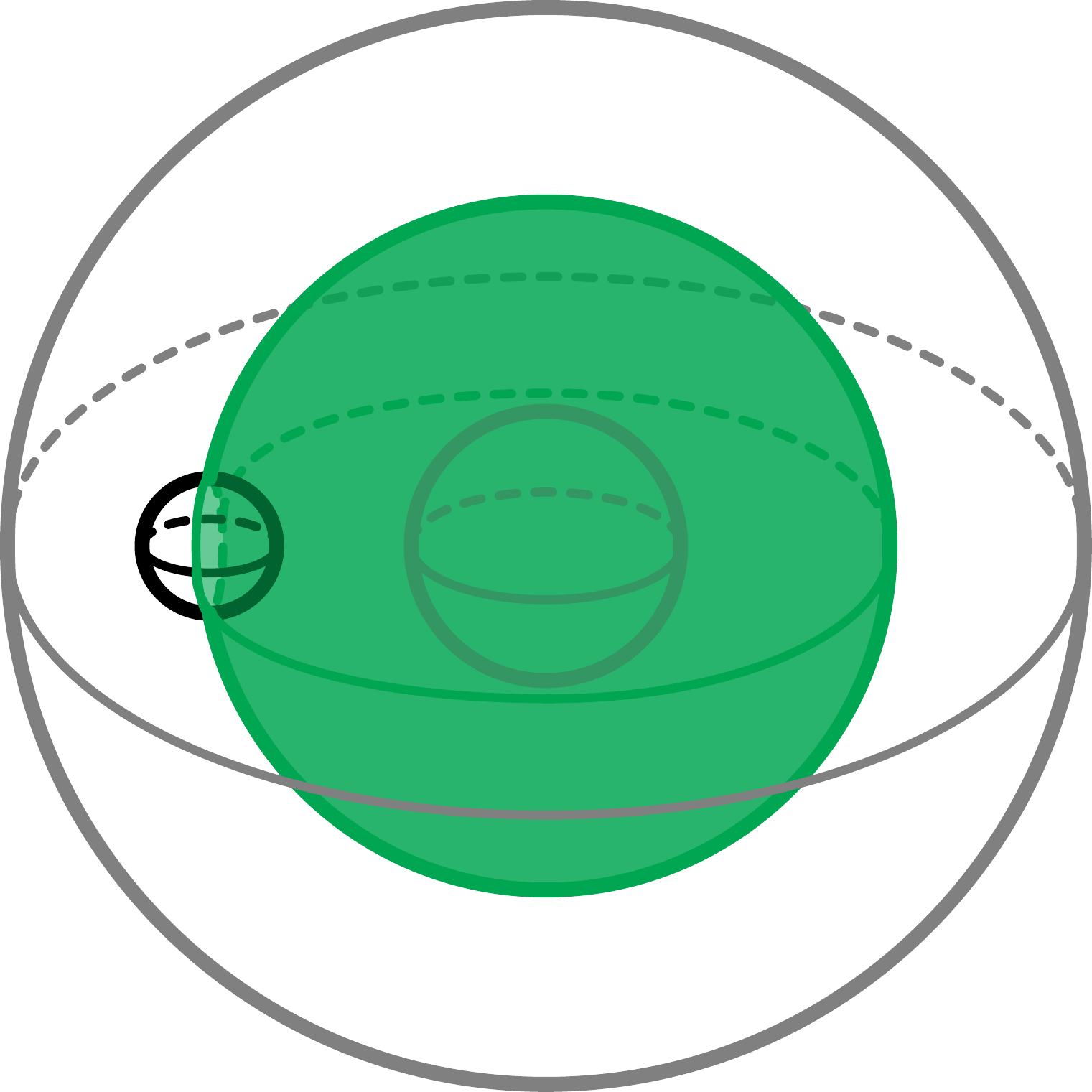}\hspace{.1\textwidth}
    \includegraphics[width=0.2\textwidth]{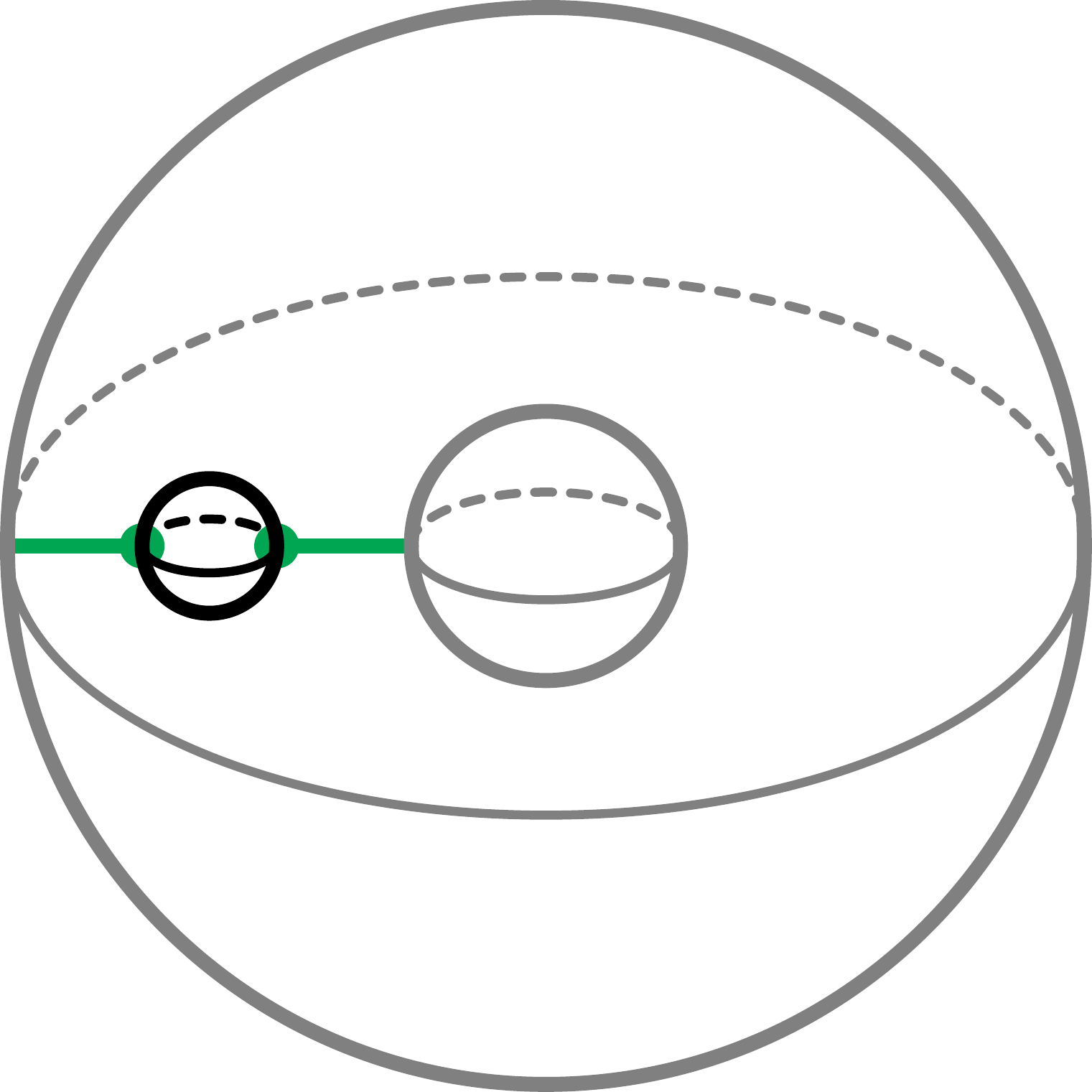}\hspace{.1\textwidth}
    \includegraphics[width=0.2\textwidth]{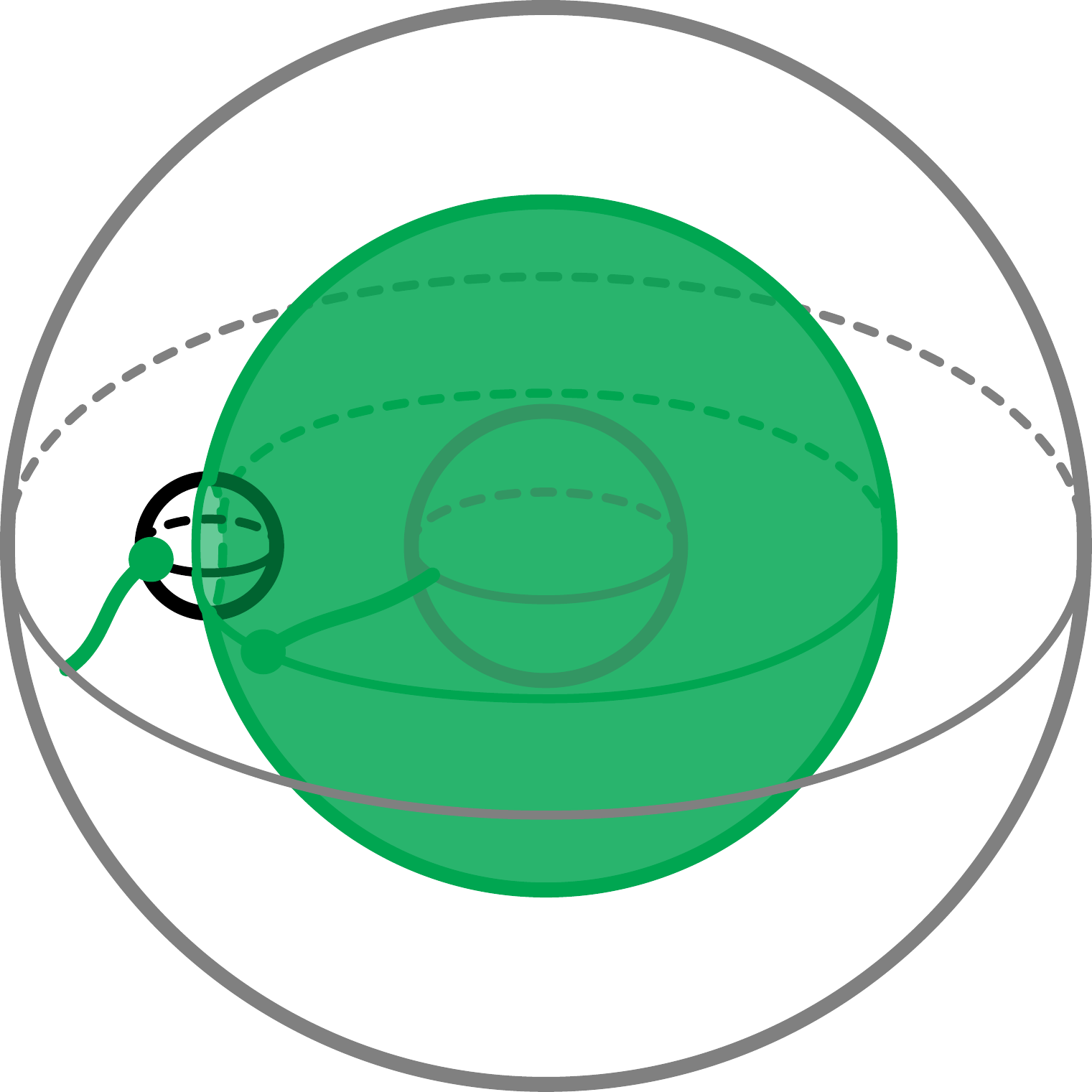}
    \caption{An illustration of the standard disk $D_n$ (left), arc $A_n$ (center), and one of the disk-arc pairs $P_n^\pm$ (right) in $Y_n\cong (S^1\ttimes S^2)^\circ$ in the case that $n$ is odd. Here, $Y_n$ is drawn as a punctured copy of $S^2\times I$, in which the two sphere boundary components are identified by a reflection.}
    \label{fig:standard_subsets_odd}
\end{figure}

We can now define arborescent spanning solids. Performing the Murasugi sum of orientable standard spun annuli along only 1-stripe patterns never yields a solid with a connected boundary, so we will use both 1- and 2-stripe patterns.

\begin{definition}\label{def:arborescent_spanning_solid}
    An {\emph {arborescent spanning solid}} is a compact 3-manifold $M\subset S^4$ for which there exists a disjoint union $Q=\bigsqcup_iQ_i$ of standard 3-spheres such that for each $i$, the intersection $Q_i\cap M$ is a 3-ball in a 1- or 2-stripe pattern, and the intersection of $M$ with each component of $S^4\cut  Q$ is a standard spun annulus or M\"obius band.
\end{definition}

Note that unlike the classical case, we must specify that $Q$ consists of standard 3-spheres.

\begin{definition}\label{def:arborescent_surface}
    A surface $S\subset S^4$ is {\emph{arborescent}} if $S$ bounds an arborescent spanning solid.
\end{definition}

As in the classical case, we can prescribe an arborescent spanning solid using a decorated tree. 

\begin{definition}\label{def:sunrise_tree}
    A {\emph{sunrise tree}} $G$ is a finite tree in which each vertex $v$ is assigned a label $n_v\in \mathbb{Z}^+$ and a choice of partition of the edges $\{e_1,\ldots, e_{m}\}$ incident to $v$ into subsets $A_v,B_v,L_v,R_v$. These partitions must satisfy the following properties:

    \begin{enumerate}[label=(\arabic*)]
        \item\label{case:ordering} The set $A_v$ is assigned a fixed ordering. 
        \item Each of $L_v,R_v$  contains at most one element.
        \item Suppose edge $e$ is incident to vertices $v,v'$.  If $e\in A_v$ then $e\in B_{v'}$; if $e\in B_v$ then $e\in A_{v'}$.
        \item Suppose edge $e$ is incident to vertices $v,v'$. If $e\in L_v\cup R_v$ then $e\in L_{v'}\cup R_{v'}$. Note that it is allowable for $e$ to be in both $R_v, R_{v'}$ or both $L_v, L_{v'}$ or both $R_v$, $L_{v'}$ or both $L_v$, $R_{v'}$.
    \end{enumerate}
    
    We indicate the partition $A_v,B_v,L_v,R_v$ and the ordering from \ref{case:ordering} via a planar embedding of $G$ as shown schematically in the left illustration in Figure \ref{fig:sunrise}. Each vertex $v$ is drawn as a semicircle containing its label $n_v$, with a straight line segment at the bottom and a curved arc at the top. The edges in the unordered set $B_v$ are incident to $v$ at the midpoint of the straight line, while the edges in $A_v$ are incident to $v$ along the interior of the curved arc in order from left to right. If $L_v\neq\varnothing$, the edge in $L_v$ is incident to $v$ at the left corner. If $R_v\neq\varnothing$, the edge in $R_v$ is incident to $v$ at the right corner. 

    We also include a weighting $w_i\in\mathbb{Z}/2\mathbb{Z}$ between adjacent edges in $L_v$, $A_v$, and $R_v$ (but not in the region between $R_v, L_v$ containing $B_v$, even if $B_v$ is empty). 
\end{definition}

An example of a sunrise tree is given in the right image of Figure \ref{fig:sunrise}.

\definecolor{darkgreen}{rgb}{0,0.67,0}
\definecolor{midblue}{rgb}{.1,.2,1}

\begin{figure}[ht]
    \labellist
    \pinlabel{$n_v$} at 45 45
    \pinlabel{\textcolor{darkgreen}{$B_v$}} at 45 -10
    \pinlabel{\textcolor{red}{$A_v$}} at 45 107
    \pinlabel{\textcolor{midblue}{$R_v$}} at 100 32
    \pinlabel{$L_v$} at 0 32
    \pinlabel{1} at 264 12
    \pinlabel{4} at 264 62
    \pinlabel{3} at 221 108
    \pinlabel{1} at 335 88
    \pinlabel{1} at 280 78
    \pinlabel{$w_1$} at 16 76
    \pinlabel{$w_2$} at 34 82
    \pinlabel{$w_k$} at 83 54
    \endlabellist
    \includegraphics[width=100mm]{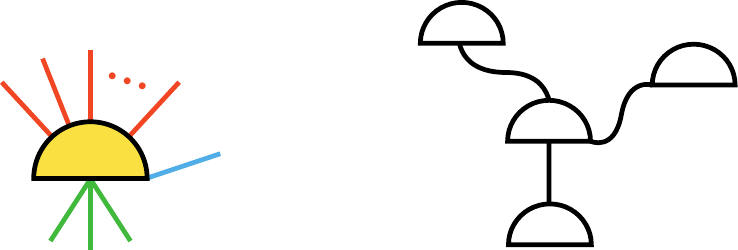}
    \vspace{.1in}
    \caption{Left: a schematic illustration of a vertex $v$ in a sunrise tree. Note that in this case, $L_v=\varnothing$. Right: a example of a sunrise tree with four vertices.}
    \label{fig:sunrise}
\end{figure}
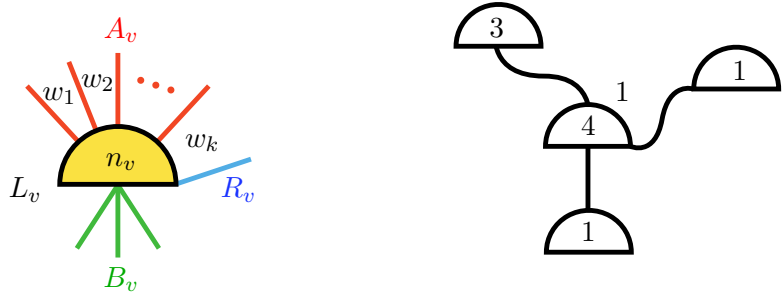

\begin{remark}\label{rem:G_from_M}
    Every arborescent spanning solid $M$ encodes a sunrise tree $G_M$. Specifically, the vertices correspond to the components of $M\cut Q$, and the edges correspond to the 3-spheres $Q_i$ in Definition \ref{def:arborescent_spanning_solid}. If two components $M_v$ and $M_{v'}$ (corresponding to vertices $v$ and $v'$) are glued according to a 1-stripe pattern, and the gluing identifies a neighborhood of a disk in $M_v$ with a neighborhood of a arc in $M_{v'}$, then we record this with an edge $e$ connecting $A_v$ and $B_{v'}$. If $M_v$ and $M_{v'}$ are identified with a 2-stripe pattern, then we record this with an edge $e$ which joins $L_v$ or $R_v$ to $L_{v'}$ or $R_{v'}$, depending on the handedness of the 2-stripe pattern (see Construction \ref{constr:arborescent} below). See Figure \ref{fig:sunrise_graph_example} for a schematic illustration.
    
    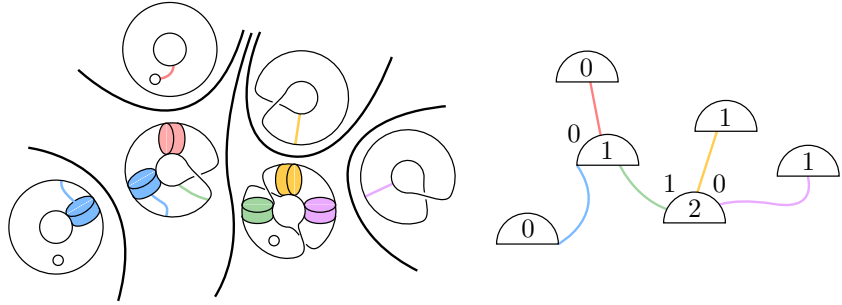
\begin{figure}[ht]
    \centering
    \labellist \small
    \pinlabel{$0$} at 625 88
    \pinlabel{$1$} at 720 182
    \pinlabel{$0$} at 680 202
    \pinlabel{$0$} at 695 282
    \pinlabel{$2$} at 825 112
    \pinlabel{$1$} at 795 142
    \pinlabel{$0$} at 855 142
    \pinlabel{$1$} at 862 222
    \pinlabel{$1$} at 962 168
    \endlabellist
        \includegraphics[width=110mm]{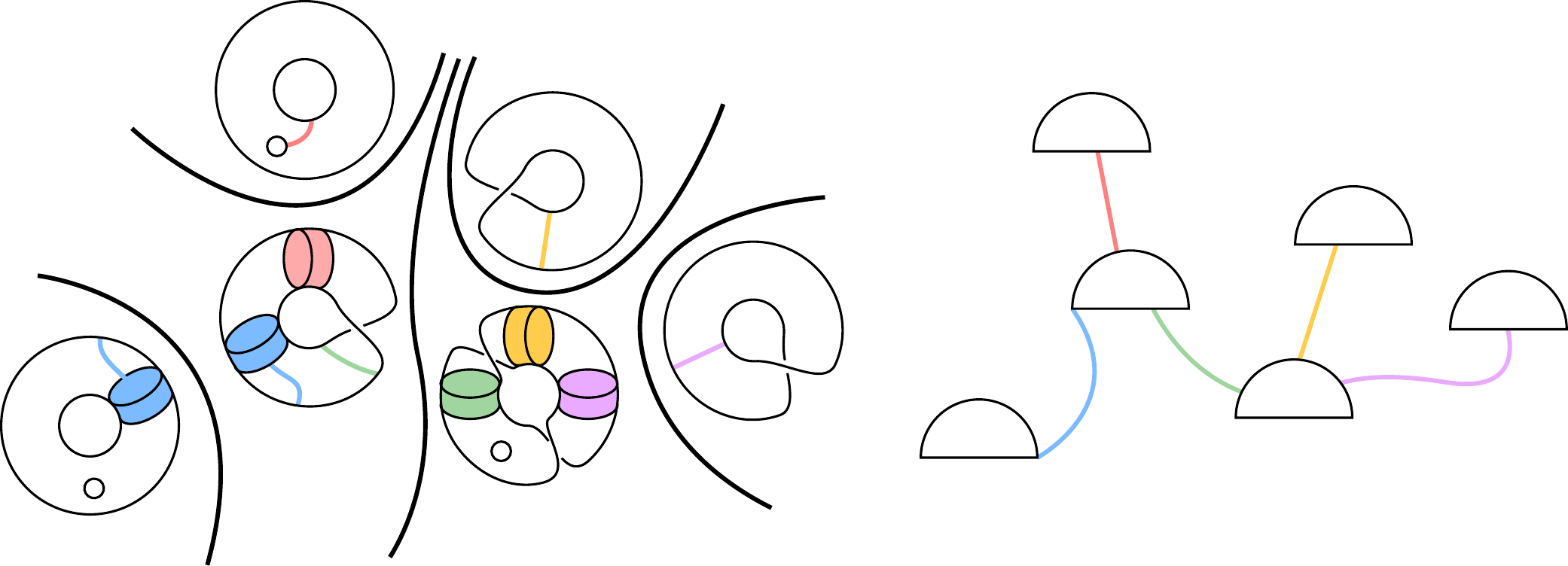}
        \caption{A schematic illustration of an arborescent spanning solid and its associated sunrise tree. The annular parts indicate standard spun annuli (with some number of twists), while the M\"obius band parts indicate standard spun M\"obius bands. Note that the lower-left edge corresponds to a 2-stripe pattern (of opposite handedness in the two components), while the remaining edges correspond to a 1-stripe pattern.}
        \label{fig:sunrise_graph_example}
    \end{figure}

     A core circle in each $M_v$ representing a generator of $H_1(M_v)$ intersects the disks and disk-arc pairs in $M_v$ in the same order that the corresponding edges appear on the curved part of $v$. The weighting between a pair of edges $e,e'\in L_v\cup A_v\cup R_v$ indicates whether the local orientation in the plumbing region of $Y_n$ corresponding to $e$ has the same (weight $0\in\mathbb{Z}/2\mathbb{Z}$) or opposite $(1\in\mathbb{Z}/2\mathbb{Z}$) orientation as the plumbing region corresponding to $e'$ (compared by transport through the segment of $Y_n$ corresponding to the arc between $e, e'$ with the assigned weight).
\end{remark}

Conversely, we show how to construct a spanning solid using a sunrise tree.

\begin{construction}\label{constr:arborescent}
    Given a sunrise tree $G$, we construct a spanning solid $M_G$ which we call the {\emph{standard arborescent spanning solid} obtained from $G$}.  We use $G$ to construct $M_G$ as a Murasugi sum of the standard solids $\{Y_{n_v}\mid v$ a vertex of $G\}$ as follows: 
    
    \begin{enumerate}[label=(\arabic*)]
        \item When $v$ and $v'$ are adjacent via an edge connecting $A_v$ to $B_{v'}$, we sum $Y_{n_v}$ and $Y_{n_v'}$ using a 1-stripe pattern. In $Y_{n_v}$, the spine of this plumbing is the standard disk $D_{n_v}$, and in $Y_{n_v'}$, the spine of this plumbing is the standard arc $A_{n_v'}$.
        \item When $v$ and $v'$ are adjacent via an edge connecting one of $L_v$ or $R_{v'}$ to one of $L_{v'}$ or $R_{v}$, we sum $Y_{n_v}$ and $Y_{n_v'}$ using a 2-stripe pattern. In both $Y_{n_v}$ and $Y_{n_v'}$, the spine of this pattern is a standard disk-arc pair $P_{n_v}^\pm$ or $P_{n_v'}^\pm$, respectively. The sign of the disk-arc pair in $Y_{n_v}$ is determined by whether the edge is in $R_v$ (positive) or $L_v$ (negative). Similarly, the sign of the disk-arc pair in $Y_{n_v'}$ is determined by whether the edge is in $R_{v'}$ (positive) or $L_v$ (negative). 
        \item We take all Murasugi sum regions in any $Y_{n_v}$ to be disjoint. The disk and disk-arc pairs appearing in the plumbing regions for a fixed $Y_{n_v}$ are arranged in the order $L_v, A_v,R_v$ (with respect to the direction prescribed by a generator for $H_1(Y_n;\mathbb{Z}))$. 
        \item As in Remark \ref{rem:G_from_M}, the $\mathbb{Z}/2\mathbb{Z}$ weights in between the edges of $L_v,R_v,A_v$ indicate whether the local orientation on $Y_{n_v}$ near the two associated disk regions are the same $(0)$ or opposite $(1)$. 
        Note that if the sum of weights on $n_v$ has the same parity at $n_v$, then the disk neighborhoods of regions associated to first and last edges on $n_v$ have the same local orientation when compared via  the last segment of $Y_{n_v}$. If the parities are opposite, then these orientations are opposite. 
       \item  The orientations on the Murasugi sum regions are those induced by the orientations on $Y_{2n}$ or collar of $\partial Y_{2n+1}$ as in Definition \ref{def:standard_objects}. 
    \end{enumerate}
\end{construction}

A sunrise tree determines an arborescent {knotted surface} up to smooth equivalence. 

\begin{theorem}\label{thm:arborescent_determined}
   Suppose $S_1, S_2$ are arborescent knotted surfaces bounding arborescent spanning solids $M_1$ and $M_2$, associated to sunrise graphs $G_1$ and $G_2$, respectively.
   \begin{enumerate}[label=(\arabic*)]
    \item\label{case:solid} If $G_1=G_2$, then $M_1$ and $M_2$ are smoothly equivalent.
    \item\label{case:surface} If $G_1$ and $G_2$ are the same when one ignores the labels from $\mathbb{Z}/2\mathbb{Z}$ at their vertices, then $S_1$ and $S_2$ are smoothly equivalent.
   \end{enumerate}
\end{theorem}

\begin{proof}
In order to prove part \ref{case:solid}, we must analyze what pieces of data were not specified in Definition \ref{def:arborescent_spanning_solid}. To prove part \ref{case:surface}, we must analyze the effect of changing the binary labels.

\begin{enumerate}[label=(\arabic*)]
\item We specified that the arcs and disk-arc pairs appearing as spines of 1-stripe and 2-stripe patterns are individually standard. Let $\alpha_1,\ldots,\alpha_n$ be the arcs in some $Y_i$ that are spines of 1-stripe patterns or the arc parts of 2-stripe patterns. Order the $\alpha_1,\ldots,\alpha_n$ so that in the construction, the Murasugi sums along these arcs occur in increasing order. Then perform ambient isotopy rel.\ boundary in $Y_i$ to straighten $\alpha_1$. (If $\alpha_1$ is part of a 2-stripe pattern spine, simultaneously standardize the attached disk.) After performing the Murasugi sum along $\alpha_1$, similarly apply ambient isotopy in $Y_i$ (now a subset of the spanning solid; we work rel.\ boundary in this solid) to straighten $\alpha_2$. Repeat iteratively. Thus, we may replace the set of arcs $\alpha_1,\ldots,\alpha_n$ with any set of arcs $\alpha'_1,\ldots,\alpha'_n\subset Y_i$ where each $\alpha'_j$ is isotopic rel.\ boundary (and rel.\ attached disk if $\alpha_j$ is in a 2-stripe pattern spine) to $\alpha_j$. Since we specify that the arcs are individually standard, this determines the spanning solid. 

\item 
In Definition \ref{def:arborescent_spanning_solid}, in order to perform the Murasugi sums that form $M$ from $Q$ we must specify local orientations near each Murasugi sum region. When a ball in $Q_i$ and a ball in $Q_j$ are identified, they have a relative orientation, i.e., making a choice of orientation the ball in $Q_i$ determines the orientation on the ball in $Q_j$. By Proposition \ref{prop:replumb}, these two choices may yield distinct spanning solids, but with equivalent boundaries.
\end{enumerate}
\end{proof}

\begin{remark}\label{rem:push_into_5ball}
    Analogous to the classical case, arborescent spanning solids which differ by reversing the local orientations on each summand near the plumbing regions become isotopic after pushing them into the 5-ball. 
\end{remark}

\begin{restatable}{question}{sunrisemoves}\label{Q:Sunrise_Moves}
    In analogy with the classical setting \cite{bonahon_siebenmann2010,cklsv26}, what moves on sunrise trees preserve the corresponding arborescent surfaces?
\end{restatable}

In order to construct interesting examples, it is essential that we use non-orientable spanning solids.

\begin{observation}
    If $M$ is an arborescent spanning solid in which at least one plumbing is done using a 1-stripe pattern with the disk spine in an orientable $Y_{2n}$, then $\partial M$ is disconnected.
\end{observation}

\begin{proposition}\label{prop:2stripespin}
    If $M$ is an arborescent spanning solid in which every plumbing is done along a 2-stripe pattern, then $M$ is the spin of a spanning surface for a 2-bridge link.
\end{proposition}
\begin{proof}
  Because we specify that the 2-stripe plumbings are performed with spines given by the standard disk-arc pair in each $Y_i$, one can see this directly. This is essentially illustrated in Figure \ref{fig:spun_example}(b).
\end{proof}

\begin{corollary}\label{cor:arborescent_2-knot}
    If $M$ is an oriented, arborescent spanning solid with connected boundary, then $M$ is the spin of a Seifert surface for a 2-bridge knot.
\end{corollary}

\begin{figure}[!ht]
\begin{center}
\labellist
\pinlabel{0} at 22 11
\pinlabel{0} at 96 11
\pinlabel{0} at 58 60
\endlabellist
\raisebox{15mm}{\includegraphics[width=30mm]{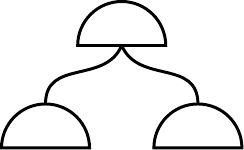}}
\hspace{1in}
\labellist \tiny
\pinlabel{drilled-out} at 215 155
\pinlabel{solid torus} at 215 110
\endlabellist
\includegraphics[width=80mm]{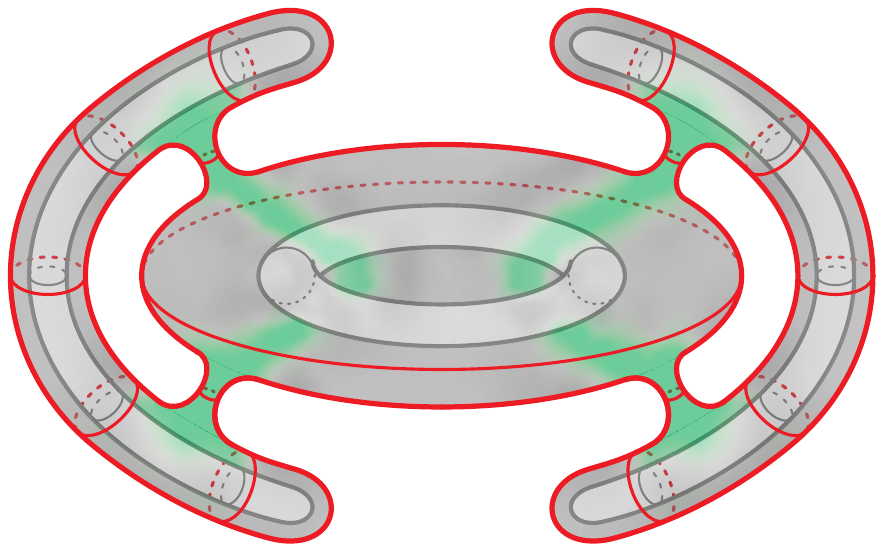}
\caption{The boundary of this arborescent solid has a component of genus 2 (shown in bold) and thus is not spun. The solid consists of a 3-ball with an unknotted solid torus drilled out, attached by two solid tubes (each) to two thickened 2-spheres. The 3-balls along which the Murasugi sums are performed are shaded.}
\label{fig:000}
\end{center}
\end{figure}
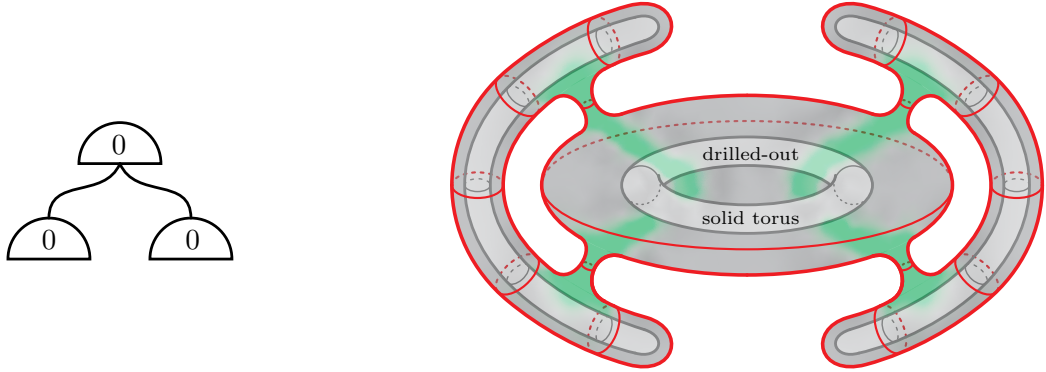

\begin{example}\label{ex:interesting_example}
The solid $Y_0$ has boundary consisting of one unknotted 2-sphere and one unknotted torus, embedded as a spun trivial link of two components. Let $M$ be a solid with boundary surface $S$. Murasugi summing a copy of $M$ with $Y_0$ using a 1-stripe pattern with arc spine in $M$ and standard disk $D_0$ spine in $Y_0$ yields a solid whose boundary $S'$ is obtained from $S$ by surgery along a framed arc (``attaching a tube") and adding a disjoint 2-sphere component. 
In particular, the sunrise tree shown left in Figure \ref{fig:000} describes a solid, shown right,  whose boundary is an 
arborescent surface with a genus-2 component (in addition to two spheres and a torus) and therefore is not spun. The same conclusion holds if we replace the 0's in the figure with any triple of non-negative even integers. For example, see Figure \ref{fig:222}.
\end{example}

\begin{figure}[!ht]
\begin{center}
\labellist \small
\pinlabel{2} at 22 11
\pinlabel{2} at 96 11
\pinlabel{2} at 58 60
\endlabellist
\raisebox{12.5mm}{\includegraphics[width=25mm]{figures/nonspinexample}}
\hspace{.125in}
\labellist \tiny
\endlabellist
\includegraphics[width=70mm]{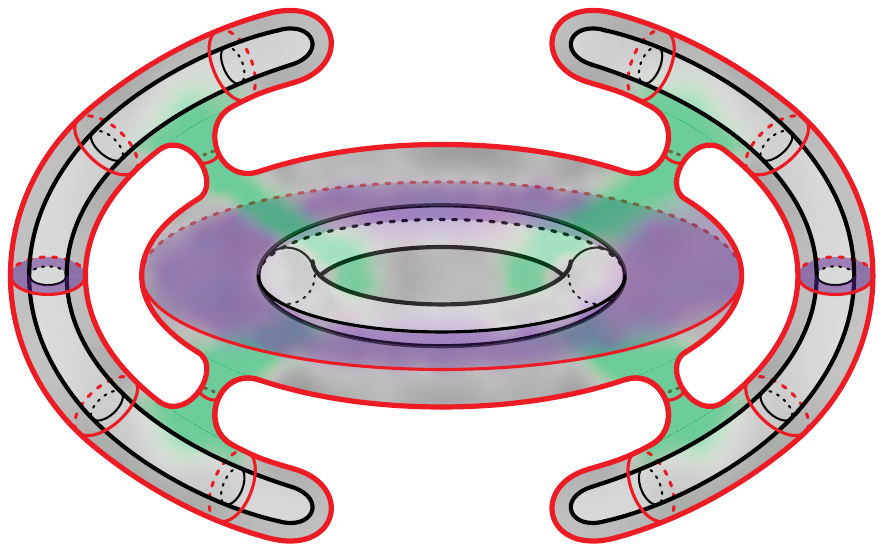}
\hspace{.125in}
\raisebox{10mm}{\includegraphics[width=40mm]{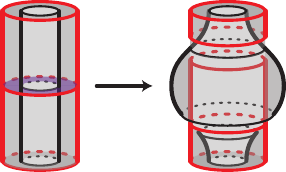}}
\caption{This arborescent solid, like the one in Figure \ref{fig:000}, is not spun.  The three annuli shown center indicate that one should transform the pictured solid (from Figure \ref{fig:000}) in the manner shown right to obtain the arborescent solid described by the sunrise tree shown left.}
\label{fig:222}
\end{center}
\end{figure}

\section{State solids for broken surface diagrams}\label{sec:states}

We begin by briefly reviewing states and state surfaces for classical link diagrams. Then, we adapt these ideas to dimension four.

\begin{definition}
   Given a diagram $E\subset S^2$ of a link $L\subset S^3$, a {\emph{Kauffman state}}  $X$ is obtained by resolving each crossing of $E$ by either of two possible smoothings (traditionally called $A$- or $B$-type).  The state $X$ is a disjoint union of circles in $S^2$ that agrees with $E$ away from crossings.
   
   One obtains a \emph{state surface} $F_X$ by capping off each circle $Y$ of a Kauffman state $X$ with a disk $U_Y$, typically on either side of $S^2$ in $S^3$, and then connecting the disks with half-twisted bands, one at each crossing \cite{ozawa11,adams_kindred2013}.  
\end{definition}

In the literature, a Kauffman state $X$ is sometimes defined to include  $A$- and $B$-labeled arcs that mark the crossings in $E$. 

\begin{remark}
    When obtaining a state surface from a Kauffman state $X$, the disks that cap off the circles of $X$ are usually taken to lie on a single side of $S^2$, but different choices of layering are also possible. 
    These choices often yield inequivalent surfaces, sometimes with notably different properties (for example, one surface can be incompressible while the other is compressible \cite{kindred_essential}). 
\end{remark}

\begin{remark}\label{rem:state_plumb_3D}
    For any circle $Y$ of a state $X$, any state surface $F_X$ de-plumbs along the disk $U_Y$. Assuming that $E$ has no nugatory crossings and the interior of $U_Y$ is disjoint from $S^2$, such a de-plumbing decomposes $F_X$ as a {\it nontrivial} Murasugi sum (meaning that neither summand is a 3-ball) whenever there are circles of $X$ on both sides of $Y$ in $S^2$.
\end{remark}

One can make the state surface construction more explicit by inserting, a l\`a Menasco \cite{menasco}, a small 3-ball $C_i$ about each crossing point of $E$. We write $C=\bigsqcup_iC_i$, and embed $L$ in $(S^2\setminus C)\cup\partial C$ such that each hemisphere of each $\partial C_i\cut S^2$ contains an overpass or underpass. See Figure \ref{fig:bubble}, which also shows how, in this context, one can define a Kauffman state to be a 1-manifold $X$ with 
\[L\cap S^2\subset X\subset (L\cup \partial C)\cap S^2.\]

\begin{figure}[ht]
    \begin{center}
        \includegraphics[width=\textwidth]{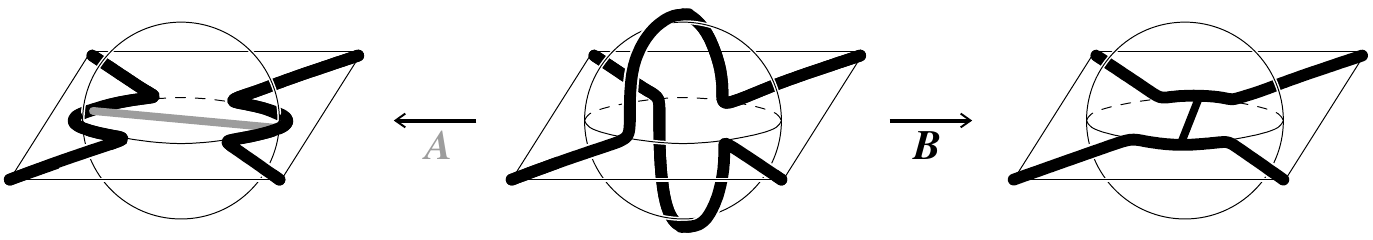}
        \caption{Embedding a link and smoothing a crossing at a Menasco bubble}\label{fig:bubble}
    \end{center}
\end{figure}

\begin{definition}
    We call a state surface {\it all-up} if each state circle is capped off with a disk whose interior lies entirely above $S^2\cup C$. Similarly, we call a state surface {\it all-down} if the interior of each such disk lies below $S^2\cup C$, or {\it up-down} if the interior of each such disk is disjoint from $S^2\cup C$, possibly with some above and some below.
\end{definition}

In this section, we will explore an analogous construction for broken surface diagrams $E\subset S^3$ of knotted surfaces in $K\subset S^4$. 
First, we will carefully and explicitly present the general construction of states and state solids from broken surface diagrams, in part by adapting Menasco's crossing bubbles to the context of broken surface diagrams. It will then be interesting to ask which states of the crossings in broken surface diagrams are {\it viable}, in the sense that they extend to state solids. The answer, unlike in dimension three, is not ``all of them.'' 
In fact, the answer depends subtly on the ``layering'', and 
we will answer several versions of this viability question in Theorems \ref{thm:up_solid} and \ref{thm:state_solid}.

\subsection{States of broken surface diagrams}\label{sec:states_subsection}

We will describe a method of constructing spanning solids for a knotted surface $K\subset S^4$ from a broken surface diagram $E\subset S^3$ for $K$. First, in \textsection\ref{sec:states_subsection}, we will describe how to ``smooth'' the crossings of $E$ to obtain a closed (possibly disconnected) surface $X$ called a \emph{state} (see Definition \ref{def:state}). 
Next, in \textsection\ref{sec:up_solid}, we will describe how to construct an ``all-up'' \emph{state solid} $M_X$ (see Definition \ref{def:state-solid}) by capping off each component of $X$ with a solid that is parallel into $S^3$ and whose interior lies entirely above $S^3$, and then connecting these solids with model pieces near the smoothed crossings. 
Lastly, in \textsection\ref{sec:layered_solid}, we will describe how to construct a state solid $M_X$ with other choices of layering.

We will use the following notation throughout this section.

\begin{notation}
Let $E\subset S^3$ be a connected broken surface diagram of a surface-link $K\subset S^4$, let $c$ be the union of the (marked) self-intersection points in $E$, and let $c_t\subset c$ and $c_b\subset c$ respectively be the set of all triple points and branch points of $E$. Write $c_d=c\setminus(c_t\cup c_b)$ for the set of double points in $E$.
\end{notation}

Take a regular neighborhood $S^3\times [-1,1]$ of $S^3\times\{0\}=S^3$ in $S^4$. Let $C_t$ and $C_b$ respectively be regular neighborhoods of $c_t$ and $c_b$ in $S^3\times[-1/2,1/2]$, and let $C_d$ be a regular neighborhood of $c_d\cut (C_t\cup C_b)$ in $(S^3\times[-1/2,1/2])\cut (C_t\cup C_b)$. Then each component of $C_d$ has the form $\alpha\times B^3$ for some arc or circle $\alpha$ of $c_d\cut (C_t\cup C_b)$, where each point in $\mathring{\alpha}$ has a neighborhood that intersects $E$ as shown in Figure \ref{fig:double_point}. Similarly, each component of $C_t$ is a ball that intersects $c$ as shown left in Figure \ref{fig:triple_branch_points}, each component of $C_b$ is a ball that intersects $c$ as shown right in Figure \ref{fig:triple_branch_points}, and $C=C_t\cup C_b\cup C_d$ is a regular neighborhood of $c$ in $S^3\times[-1/2,1/2]$. We encourage the reader to think of each component of $C_t$ and $C_b$ as a 4-dimensional Menasco ball \cite{menasco}, its boundary as a 3-dimensional Menasco bubble, and similarly for the components of $C_d$. We write
$$
\widehat{S^3}=\left( S^3\setminus\mathring{C}\right)\cup\partial C.$$

\begin{figure}
    \begin{center}
        \labellist
        \pinlabel{$\approx$} at 250 75
        \pinlabel{$\times I$} at 385 75
        \endlabellist
            \includegraphics[height=1.25in]{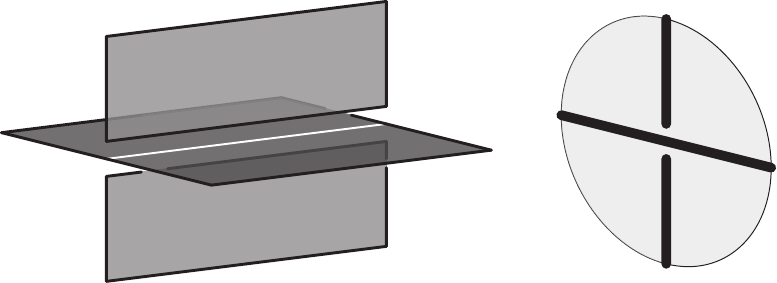}
                \caption{The neighborhood of a double point in a broken surface diagram}
        \label{fig:double_point}
    \end{center}
\end{figure}

\begin{figure}
    \begin{center}
        \;\hfill
        \includegraphics[height=1.25in]{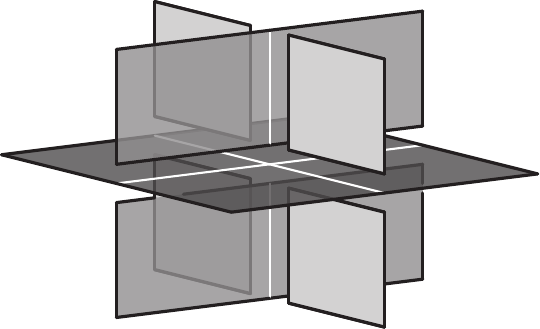}
        \hfill
        \includegraphics[height=1.25in]{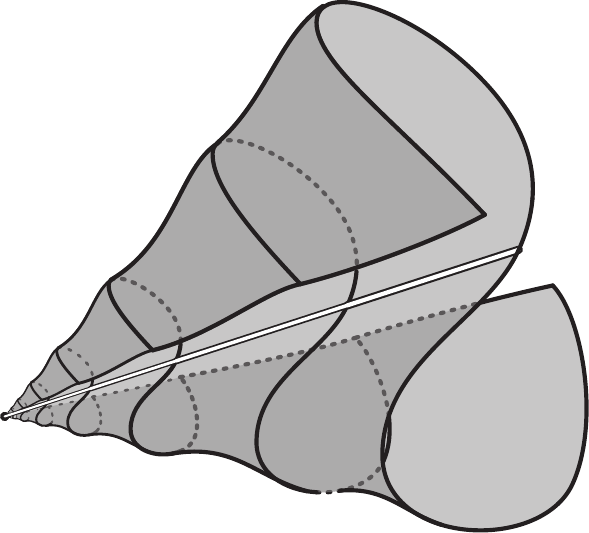}
        \hfill\;
        \caption{The neighborhood of a triple point (left) and a branch point (right) in a broken surface diagram. The branch point neighborhood may also be mirrored.}
        \label{fig:triple_branch_points}
    \end{center}
\end{figure}

\begin{definition}\label{def:Menasco_embedding}
While preserving $E$ (i.e., the projection of $K$ with crossing information), use the Menasco-style neighborhood $C$ of the crossing points of $E$ to isotope $K$ into $\widehat{S^3}\cup C_t$, so that near each double point, cusp, and triple point of $E$, $K$ appears as in Figures \ref{fig:Menasco_double_point}, \ref{fig:Menasco_branch_point}, and \ref{fig:Menasco_triple_point}, respectively.\footnote{In several figures, including Figures \ref{fig:Menasco_double_point}, \ref{fig:Menasco_branch_point}, and \ref{fig:Menasco_triple_point}, we view $\nu S^3$ as $S^3\times [-1,1]$, where $S^3$ is pictured spatially and $[-1,1]$ is the \textsc{roy-g-biv} color spectrum, with $S^3=S^3\times\{0\}$ shown in green (and mostly gray, for simplicity), and with $S^3\times(0,1)$ and $S^3\times(-1,0)$, respectively, shown in warm and cool colors. If viewing in black and white, we strongly recommend referring to the online version.} Namely:

\begin{enumerate}[label=(\arabic*)]
    \item For each component $\alpha\times B^3$ of $C_d$, $S^3$ cuts $\alpha\times S^2$ into two hemispheres, $\alpha\times S^2_\pm$, and we take $K$ to intersect each hemisphere in a rectangle, as shown in Figure \ref{fig:Menasco_double_point}.
    \item For the neighborhood $\{b\}\times D^4= D^4_b$ of each branch point $b$, $\partial D^4_b$ is a 3-sphere that intersects $S^3$ in a 2-sphere $Q_b$, which contains an endpoint $b'$ of an arc of $c_d\cut (C_t\cup C_b)$. View $\partial D^4_b=(Q_b\times[-1/2,1/2])/\sim$ as the suspension of $Q_b$, and denote the projection map $\pi:Q_b\times(-1/2,1/2)\to Q_b$.
    Then, as shown in Figure \ref{fig:Menasco_branch_point}, we take $K$ to intersect $\partial D^4_b$ in a disk $D$ that looks like a checkerboard surface for a 1-crossing diagram of the unknot, in the sense that $D$ contains a ``vertical'' arc $v$ with $\pi(v)=\{b'\}$ and $\pi|_{D\setminus v}$ is injective. Also see Figure \ref{fig:branched}
    \item\label{case:Menasco_Ct} For the neighborhood $\{t\}\times D^4= D^4_t$ of each triple point $t$, 
    $\partial D^4_t$ is a 3-sphere that intersects $S^3$ in a 2-sphere $Q_t$, which contains six endpoints $t_i$ of arcs of $c_d\cut (C_t\cup C_b)$. View $\partial D^4_t=(Q_t\times[-1/2,1/2])/\sim$ as the suspension of $Q_t$. Take a  3-dimensional (i.e., traditional) Menasco ball $B^3_i\subset \partial D^4_t$ about each point $t_i$, and consider
    \[\widehat{Q_t}=Q_t\cup \bigsqcup_iB^3_i\subset\partial D^4_t.\] 
    As shown in Figure \ref{fig:trivial3_saddle}, we take $K$ to intersect $D^4_t$ in three properly embedded disks that are simultaneously boundary-parallel (say, near $\partial D^4_t$, with just one critical point relative to the radial morse function from $t$); in the Figure, $K\cap \widehat{Q_t}$ is black.
\end{enumerate}
\end{definition}    

\begin{figure}
    \begin{center}
        \labellist
        \pinlabel{$\approx$} at 540 150
        \pinlabel{$\times I$} at 840 150
        \endlabellist
        \includegraphics[height=1in]{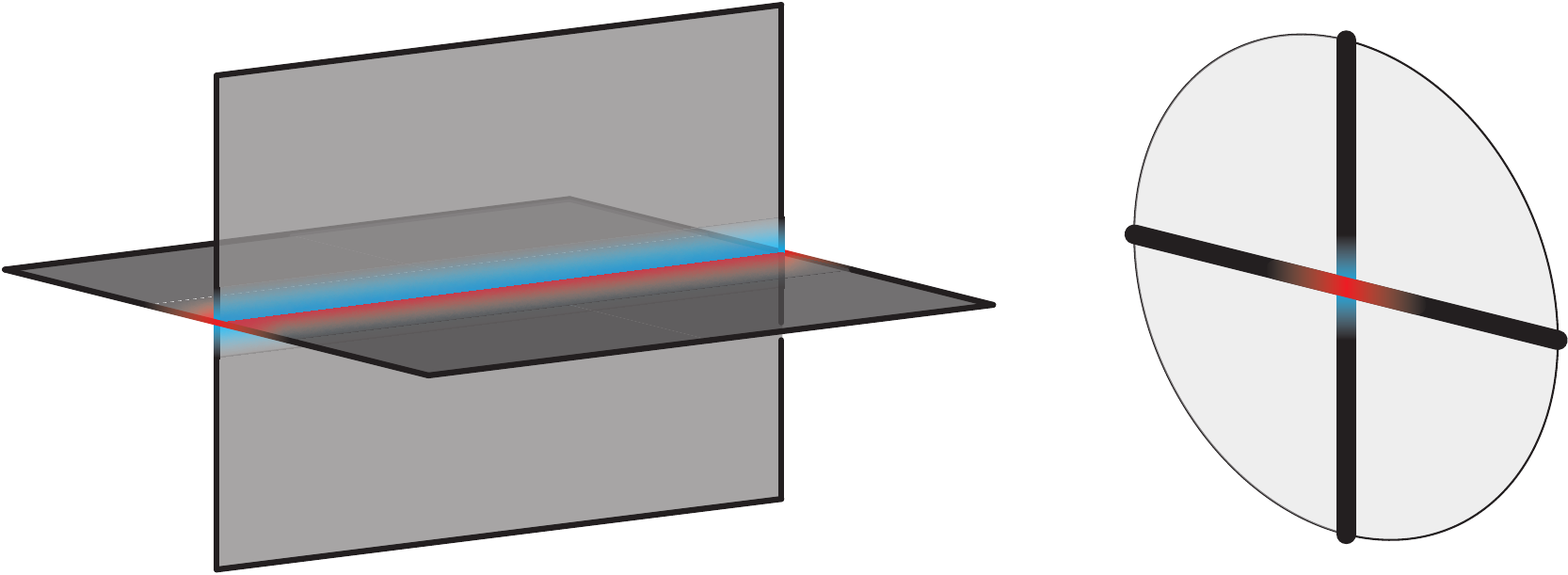}
        \caption{The neighborhood of a double point after $K$ has been pushed into the union of $S^3$ and the Menasco bubbles}
        \label{fig:Menasco_double_point}
    \end{center}
\end{figure}

\begin{figure}
    \begin{center}
        \labellist
        \pinlabel{$=$} at 485 90
        \pinlabel{$\cup$} at 930 90
        \pinlabel{$\cup$} at 1455 90
        \endlabellist
        \includegraphics[scale=.25]{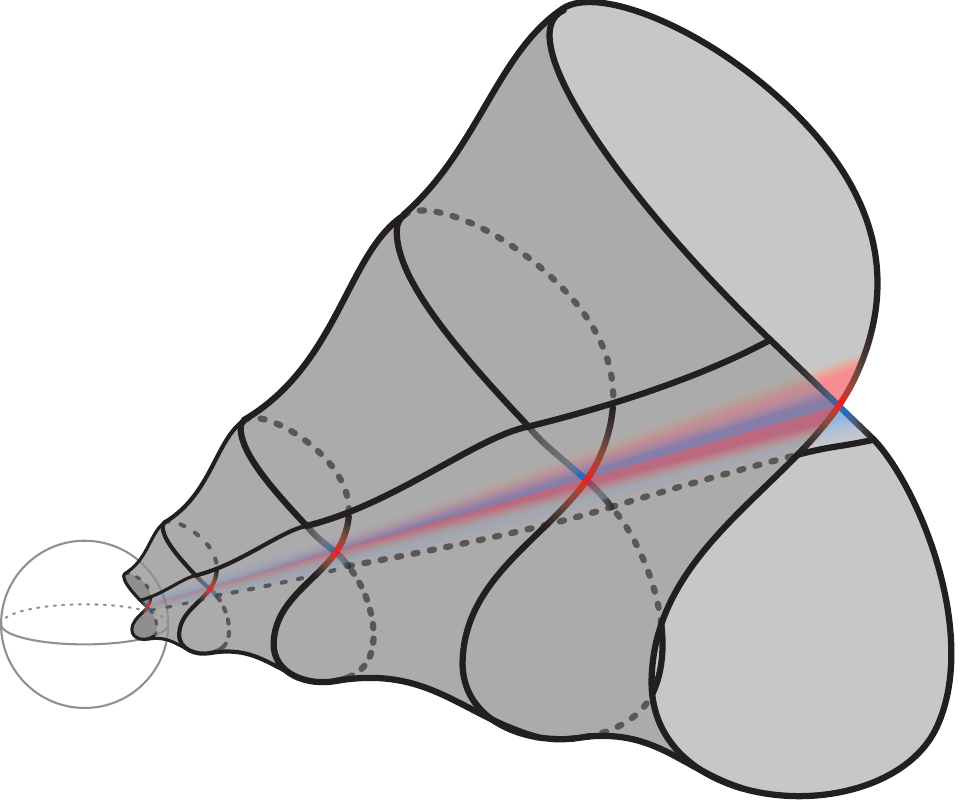}\hfill
        \raisebox{0in}{\includegraphics[scale=.25]{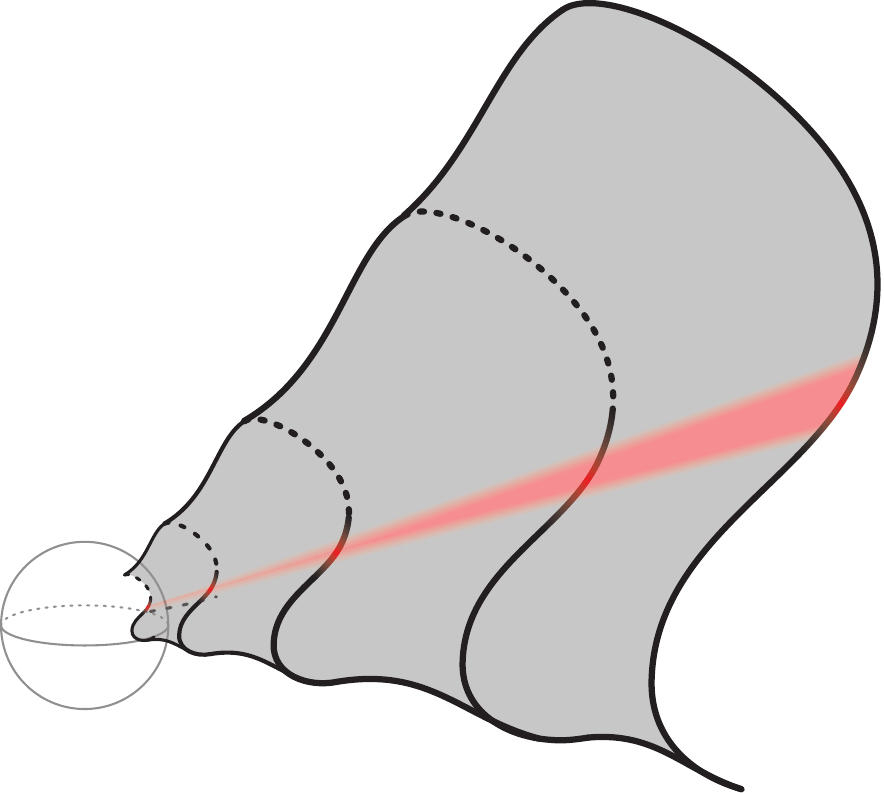}}
        \includegraphics[scale=.25]{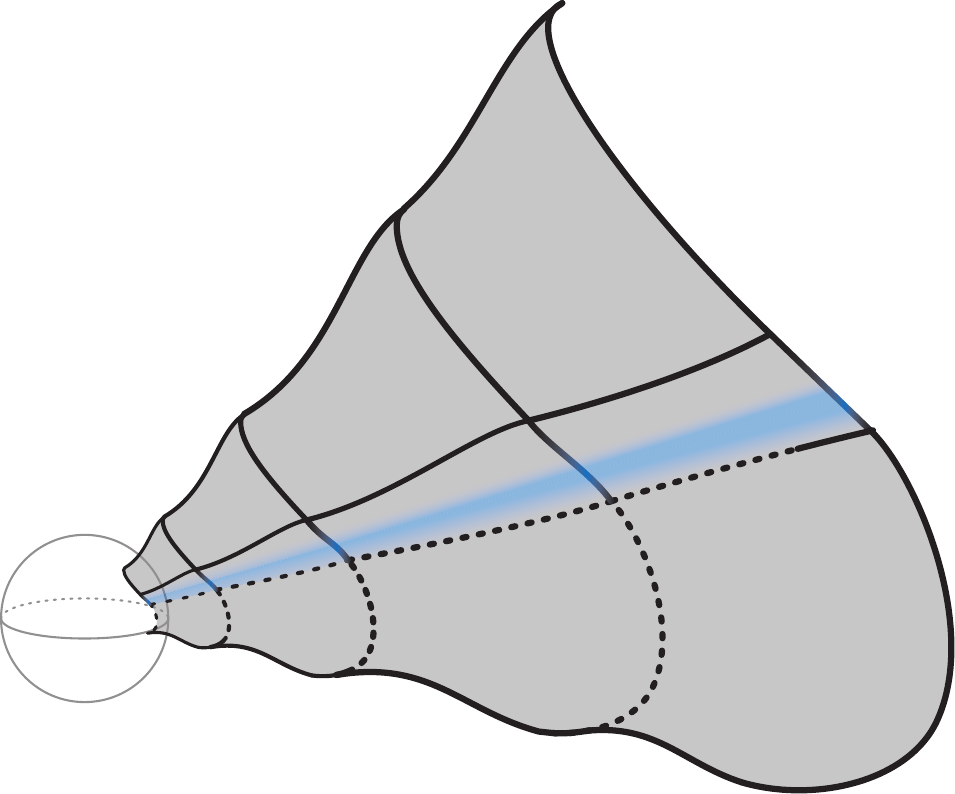}\hfill
        \includegraphics[scale=.75]{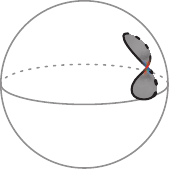}
        \caption{The neighborhood of a branch point after $K$ has been pushed into the union of $S^3$ and the Menasco bubbles (the bubble here is the suspension of the depicted 2-sphere)}
        \label{fig:Menasco_branch_point}
    \end{center}
\end{figure}

\begin{figure}
    \begin{center}
        \labellist
        \pinlabel{$=$} at 540 150
        \pinlabel{$\cup$} at 1110 150
        \pinlabel{$\cup$} at 1470 150
        \endlabellist
        \includegraphics[scale=.25]{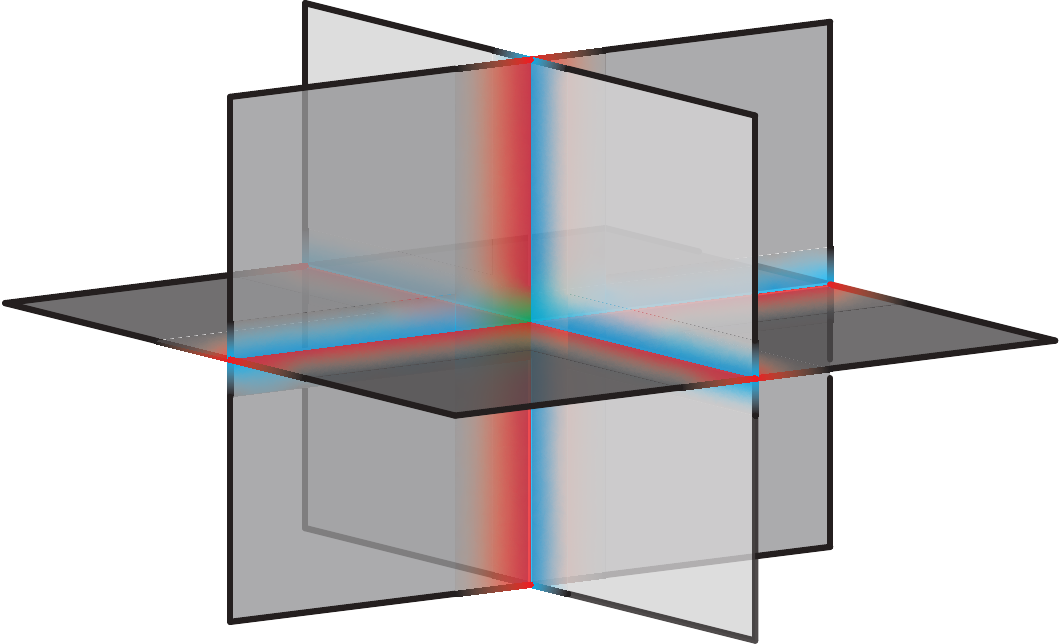}\hfill
        \raisebox{.35in}{\includegraphics[scale=.25]{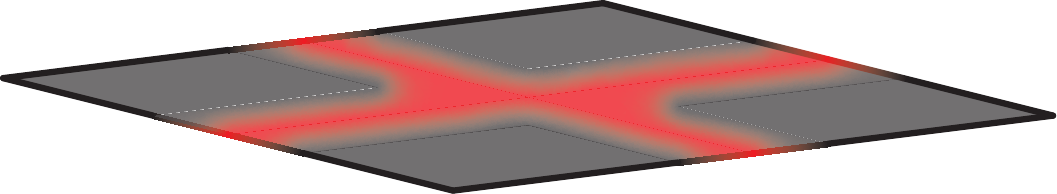}}\hfill
        \includegraphics[scale=.25]{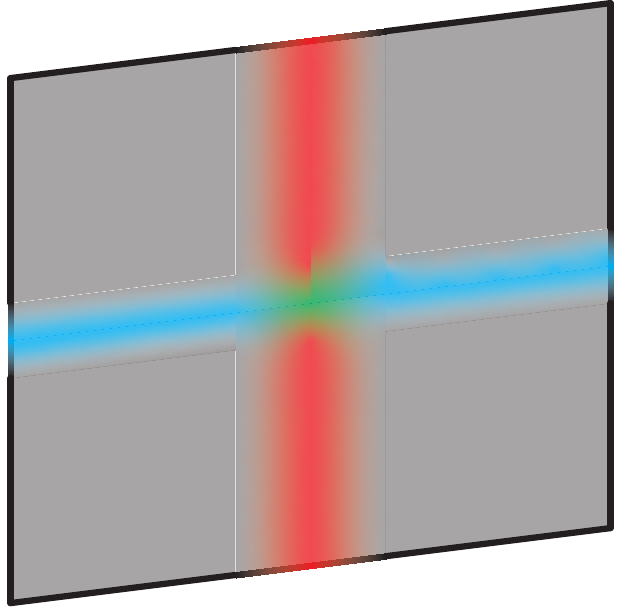}\hfill
        \includegraphics[scale=.25]{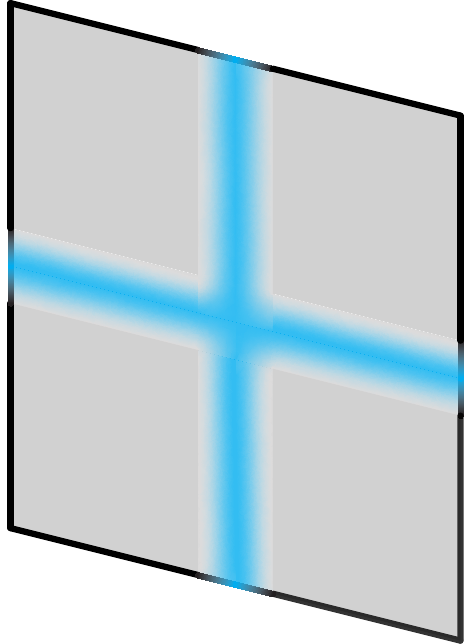}
        \caption{The neighborhood of a triple point after $K$ has been pushed into the union of $S^3$ and the Menasco bubbles}
        \label{fig:Menasco_triple_point}
    \end{center}
\end{figure}

    \begin{figure}
        \centering
        \labellist
        \pinlabel{$\approx$} at 192 32
        \pinlabel{$\approx$} at 450 32
        \endlabellist
        \includegraphics[width=\textwidth]{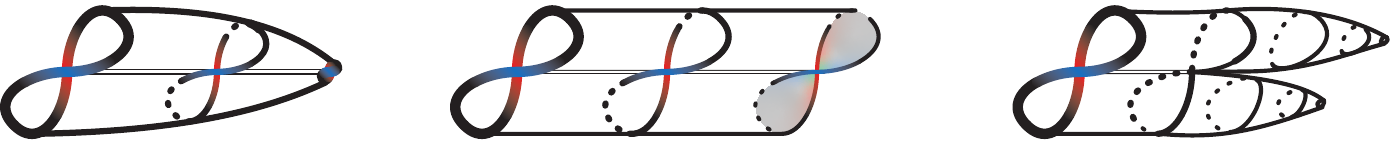}
        \caption{Local isotopies near a branch point}
        \label{fig:branched}
    \end{figure}

    \begin{figure}
        \begin{center}
        \labellist
        \pinlabel{$=$} at 1360 225
        \pinlabel{$\cup$} at 1780 225
        \pinlabel{$\cup$} at 2155 225
        \endlabellist
            \includegraphics[scale=.175]{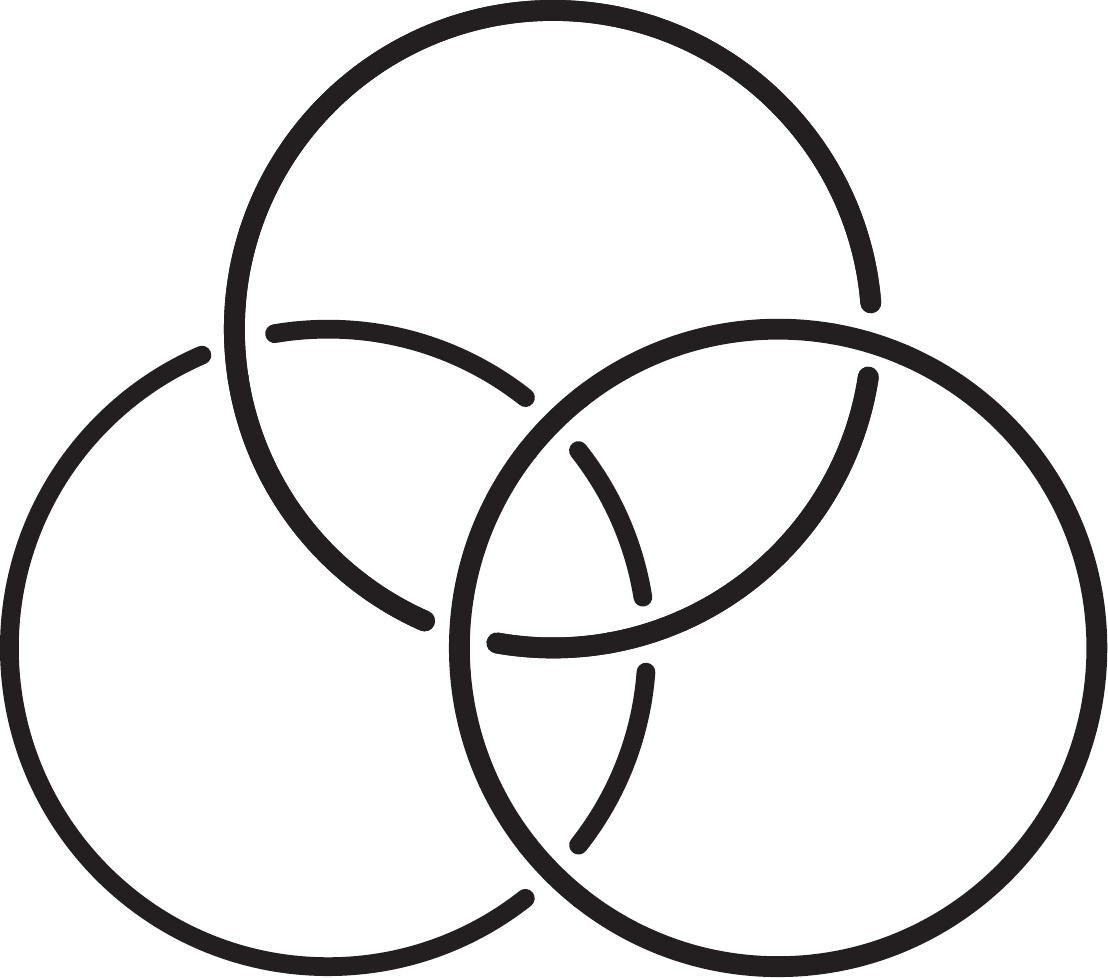}\hspace{.5in}
            \includegraphics[scale=.175]{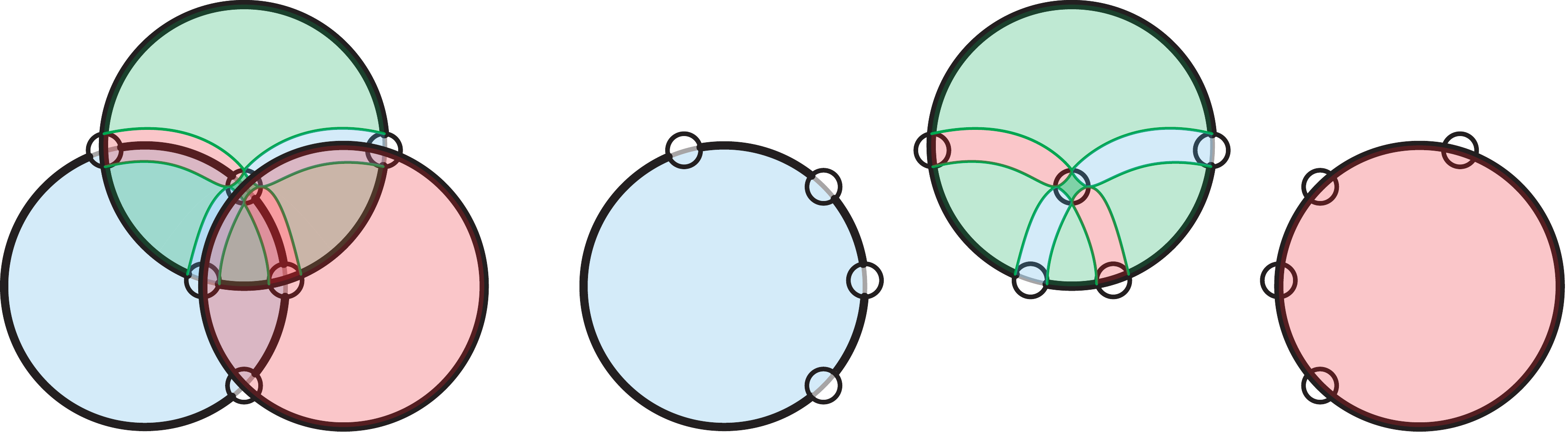}
            \caption{Left: The diagram $P$ of the trivial link of three components. Right: Embedding a {knotted surface} in the boundary of a Menasco neighborhood near a triple point}
            \label{fig:trivial3_saddle}
        \end{center}
    \end{figure}

    \begin{remark}In part \ref{case:Menasco_Ct} of Definition \ref{def:Menasco_embedding}, we could take all three disks of $K\cap D^4_t$ to lie in $\partial D^4_t$, specifically with the disks bounded by the top and bottom circles lying in its upper and lower hemispheres, respectively, both with interiors disjoint from all 3-dimensional Menasco balls $B^3_i$. The disk of $K\cap D^4_t$ bounded by the middle circle, however, would still intersect the equatorial 2-sphere $\partial D^4_t\cap S^3$, as it is incident to overpasses and underpasses from double points.  More importantly, it would need to intersect the interior of one of the 3-dimensional Menasco balls $B^3_i$. This would be not only inconvenient, but actually prohibitive during the state solid construction. For this and other technical reasons (like guaranteeing that a state and a knotted surface intersect in a predictable way near a triple point), we elect to perturb the interiors of all three disks into the interior of $D^4_t$. 
    There may be other settings in which it proves useful to embed $K$ entirely in $\widehat{S^3}$.
\end{remark}
    
\begin{notation}
    Write $K_0=K\setminus{C}$. 
    Note that $K_0\subset S^3$ and $ K\subset \widehat{S^3}\cup C_t$.
\end{notation}

\begin{definition}\label{def:state}
A {\emph{state}} of $D$ is a closed, possibly disconnected, surface $X$ embedded in $\widehat{S^3}$ such that: 

\begin{enumerate}[label=(\arabic*)]
    \item $K_0\subset X\subset \widehat{S^3}$, 
    \item For each (circle or arc) component of $c_d\cut (C_b\cup C_t)$, $X$ intersects the corresponding (torus or cylinder) component $S$ of $\partial C\cut (C_b\cup C_t)$ in one of the two opposite pairs of annuli or rectangles of $S\cut K$, as in Figure \ref{fig:Menasco_double_point_smoothed}---also see Figure \ref{fig:bubble}.
    \item\label{case:branched} For each component $B$ of $C_b$, $X\cap \partial C_b$ consists of two disks in $K\cap \partial C_b$, as in Figure \ref{fig:Menasco_branch_point_state}---$K\cap \partial C_b$ is a disk that projects to a checkerboard surface for the one-crossing diagram of the unknot, and $X\cap \partial C_b$ is comprised of the two checkerboard disks for this surface.
    \item\label{case:Ct}For each component $T$ of $C_t$, each component $X_0$ of $X\cap T$ is a disk which lies in $\partial T$ and whose interior is disjoint from $K$ and the interiors of the six 3-dimensional Menasco balls in $\partial T$. If all six points of $\partial T\cap c_d$ lie on the same side of $\partial X_0$ in the 2-sphere $\partial T\cap S^3$, then we place $X_0$ in $S^3$.
\end{enumerate}
\end{definition}

\begin{figure}[ht]
    \begin{center}
        \labellist
        \pinlabel{$\longleftarrow$} at 300 130
        \pinlabel{$\longrightarrow$} at 660 130
        \endlabellist
        \includegraphics[height=.75in]{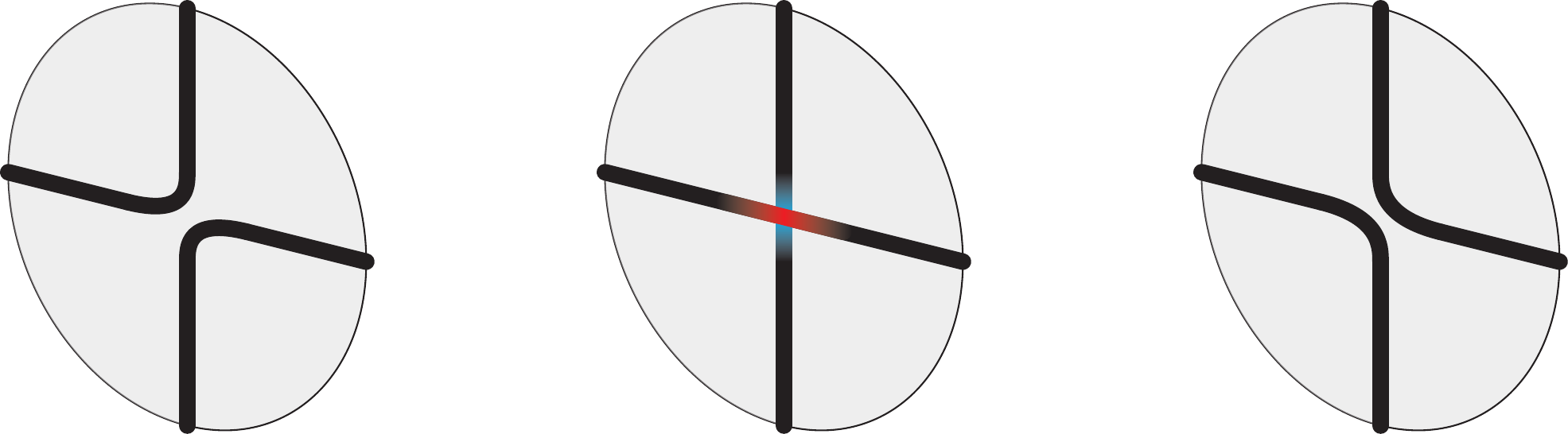}
        \caption{To smooth an arc or circle of double points, take the product of this picture with $I$ or $S^1$}
        \label{fig:Menasco_double_point_smoothed}
    \end{center}
\end{figure}

\begin{figure}[ht]
        \centering
        \includegraphics[width=.4\textwidth]{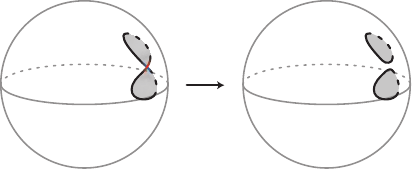}
        \caption{Smoothing a broken surface diagram near a branch point (also see Figure \ref{fig:Menasco_branch_point})}
        \label{fig:Menasco_branch_point_state}
\end{figure}

We consider two states $X$ and $X'$ of a broken surface diagram $E$ to be equivalent if $X\setminus C_t=X'\setminus C_t$ and if $X\cap C_t$ and $X'\cap C_t$ are isotopic rel.\ boundary in $C_t\cut K$. 
A state $X$ of $E$ is thus determined by (i) a binary choice of smoothing for each arc or circle of double points and (ii) a choice of isotopy class (rel.\ boundary in $C_t\cut K$) for each component of $X\cap C_t$. The choice in (ii) includes how the components are ``layered'' relative to $S^3$; this is reflected in Figures \ref{fig:ThreeGoodStatesUp}--\ref{fig:ThreeGoodStatesDown} by our coloring of each state circle on $\partial C_t$ as light highlights (orange or blue), according to whether the incident solid lies above (orange, lightest) or below (blue, darker) $S^3$. State circles for which this choice is immaterial, up to isotopy, are colored dark gray. 

\begin{remark}\label{rem:state_branched_condition}
    Not all choices in (2) above will yield states. For example, there is at most one permissible smoothing of any arc of double points which has an endpoint at a cusp. In particular, if any arc of double points has both endpoints at cusps of the same sign, then no smoothing to a state is possible, since the two branch points give incompatible restrictions. 
     See Figure \ref{fig:RP2}.  
    
    \begin{figure}[ht]
        \centering
        \includegraphics[width=\textwidth]{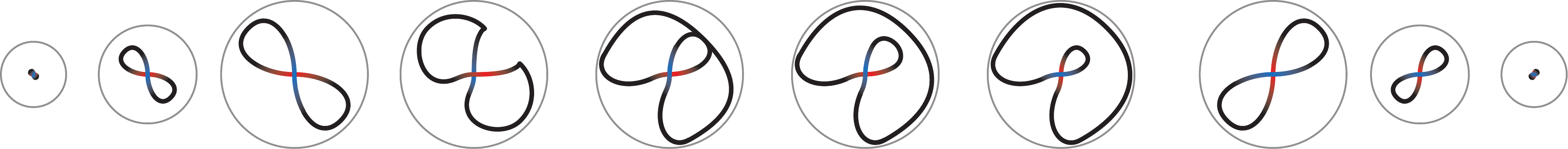}
        \caption{Viewing $S^2\times I\subset S^3$, this movie of knot diagrams gives a broken surface diagram $E$ of an unknotted projective plane; $E$ has one arc of double points, whose endpoints are cusps of the same sign, so $E$ has no state.}
        \label{fig:RP2}
    \end{figure}
\end{remark}

\begin{definition}\label{def:state-solid}
      Let $E$ be a broken surface diagram for a knotted surface $K$. A \emph{state solid} $M_X$ for a state $X$ of $E$ is a 3-dimensional spanning solid for $K$ obtained as follows.  First, take closed solids in $S^3$, one bounded by each component of $X$, and perturb them rel.\ boundary in $S^4$ to be disjoint. Second, glue on the following subsets of the Menasco neighborhood $C$:
    \begin{enumerate}[label=(\arabic*)]
        \item For each arc or circle $\alpha$ of double points, glue on 
        \[\text{(classical crossing band)}\times(I\text{ or }S^1).\]
        \item (Glue on nothing near each branch point.)
        \item\label{case:triple_cobordism} In each component $T$ of $C_t$, take $F\subset\partial C_t$ to be the union of the disk(s) of $X\cap T$ and the classical crossing bands in $\partial C_t$ coming from the smoothings of the incident arcs of double points. Construct a cobordism-with-boundary $M$ from $F$ to the empty set, such that the critical points of $M$ with respect to the radial morse function from $t$ are compressions of $F$ and three deaths; glue on $M$.
    \end{enumerate}
    A state solid $M_X$ is \emph{all-up} if the interior of each perturbed solid in the first part of the construction lies above $S^3$, \emph{all-down} if each such solid lies below $S^3$, and \emph{up-down} if each such solid is disjoint from $S^3$.
\end{definition}%

\begin{definition}\label{def:viable}
    A state $X$ for a surface $K$ is called {\emph{viable}} if there exists a state solid $M_X$.
\end{definition}

The rest of the section is devoted mainly to questions about viability. There is no concern or obstruction related to double points per se, and while part \ref{case:branched} of Definition \ref{def:state} generally restricts the states one may construct from a given diagram involving branched points (see Remark \ref{rem:state_branched_condition}), it also ensures that branched points present no obstructions when extending states to state solids. Thus, the only such obstructions will arise near triple points, coming from the requirement in part \ref{case:triple_cobordism} of Definition \ref{def:state-solid} that there is a cobordism from the state surface $F$ to the empty set.  

\begin{observation}\label{obsv:slopes}
    Because $F$ spans the trivial link of three components, a cobordism $M$ as in Definition \ref{def:state-solid}\ref{case:triple_cobordism} exists if and only if the component-wise boundary slopes of $F$ are identically zero. These slopes will generally depend both on the smoothings of the incident arcs of double points and on the layering of the solids that cap off the incident components of $X$. 
    This observation \ref{obsv:slopes} is similar to \cite[Proposition 3.2(1)]{trisections_seifertsolids}. 
\end{observation}

As a warm-up, we consider only those state solids $M_X$ that are \emph{all-up}.

\subsection{Viability of states and state solids: the all-up case}\label{sec:up_solid}

Let $X$ be a state from a broken surface diagram $E$ of a knotted surface $K\subset S^4$, and suppose $M_X$ is an all-up state solid. Let $T$ be a component of $C_t$, and consider how $M_X$ intersects the 3-sphere $\partial T$. As described in part \ref{case:Menasco_Ct} of Definition \ref{def:Menasco_embedding}, $K\cap\partial T$ is a trivial link $L$ of three components in this 3-sphere, embedded following a six-crossing classical diagram $P$ on the 2-sphere $S^3\cap \partial T$, using six 3-dimensional Menasco balls which arise as the ends of the incident arcs of double points.  The resolutions of those incident arcs of double points determine a Kauffman state of $P$, which, by part \ref{case:Ct} of Definition \ref{def:state}, is the boundary of $X\cap\partial T$. Moreover, by part \ref{case:Ct} of Definition \ref{def:state-solid}, $M_X\cap T$ is a cobordism with boundary from a state surface $F$ (necessarily all-up, because $M_X$ is all-up) for this Kauffman state to the empty set; as per Observation \ref{obsv:slopes}, all three component-wise boundary slopes of $F$ must be zero. 

Thus, we must determine which all-up state surfaces of the diagram $P$ shown left in Figure \ref{fig:trivial3_saddle} have all component-wise boundary slopes equal to zero. 

In order for the state surface associated to a layered state to have component-wise boundary slopes identically zero, it is necessary, but not sufficient, for the sum of those slopes to be zero; near the triple point $t\in T$, this is equivalent to the condition that the closed surface $\nu t \cap (K\cup X)$ has normal Euler number zero. When the underlying diagram has zero writhe, as $P$ does, the latter happens precisely when the numbers of $A$- and $B$-smoothings are equal.  Thus, there are $\binom{6}{3}=20$ candidate states of $P$ for us to consider. Note that the two checkerboard states of $P$ are not among these, because they have two $A$- and four $B$-smoothings or vice-versa. Checking these candidate states, we make the following observations. 

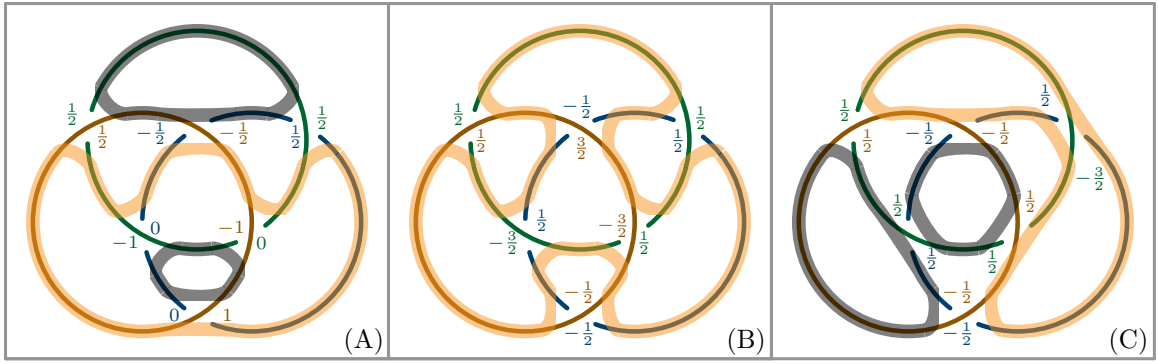
\begin{figure}
    \centering
    \labellist\small
    \pinlabel{(A)} at 950 50
    \pinlabel{(B)} at 1960 50
    \pinlabel{(C)} at 2970 50
    \tiny
    \pinlabel{{\color{mygreen} $\frac{1}{2}$}} at 180 650
    \pinlabel{{\color{mygreen} $\frac{1}{2}$}} at 840 640
    \pinlabel{{\color{mygreen} $0$}} at 680 310
    \pinlabel{{\color{mygreen} $-1$}} at 320 310

    \pinlabel{{\color{myblue} $-\frac{1}{2}$}} at 390 600
    \pinlabel{{\color{myblue} $\frac{1}{2}$}} at 770 590
    \pinlabel{{\color{myblue} $0$}} at 450 120
    \pinlabel{{\color{myblue} $0$}} at 400 350
    \pinlabel{{\color{mybrown} $\frac{1}{2}$}} at 260 590
    \pinlabel{{\color{mybrown} $-1$}} at 600 350
    \pinlabel{{\color{mybrown} $-\frac{1}{2}$}} at 610 600
    \pinlabel{{\color{mybrown} $1$}} at 590 120
    \pinlabel{{\color{mygreen} $\frac{1}{2}$}} at 1200 650
    \pinlabel{{\color{mygreen} $\frac{1}{2}$}} at 1840 640
    \pinlabel{{\color{mygreen} $\frac{1}{2}$}} at 1690 310
    \pinlabel{{\color{mygreen} $-\frac{3}{2}$}} at 1320 310
    \pinlabel{{\color{myblue} $-\frac{1}{2}$}} at 1515 670  
    \pinlabel{{\color{myblue} $\frac{1}{2}$}} at 1780 580
    \pinlabel{{\color{myblue} $-\frac{1}{2}$}} at 1520 70
    \pinlabel{{\color{myblue} $\frac{1}{2}$}} at 1420 370
    \pinlabel{{\color{mybrown} $\frac{1}{2}$}} at 1260 580
    \pinlabel{{\color{mybrown} $-\frac{3}{2}$}} at 1610 355
    \pinlabel{{\color{mybrown} $\frac{3}{2}$}} at 1525 560
    \pinlabel{{\color{mybrown} $-\frac{1}{2}$}} at 1520 180
    \pinlabel{{\color{mygreen} $\frac{1}{2}$}} at 2220 660
    \pinlabel{{\color{mygreen} $-\frac{3}{2}$}} at 2870 480
    \pinlabel{{\color{mygreen} $\frac{1}{2}$}} at 2610 260
    \pinlabel{{\color{mygreen} $\frac{1}{2}$}} at 2360 410
    \pinlabel{{\color{myblue} $\frac{1}{2}$}} at 2750 700  
    \pinlabel{{\color{myblue} $-\frac{1}{2}$}} at 2420 600
    \pinlabel{{\color{myblue} $-\frac{1}{2}$}} at 2520 70
    \pinlabel{{\color{myblue} $\frac{1}{2}$}} at 2450 270
    \pinlabel{{\color{mybrown} $\frac{1}{2}$}} at 2280 580
    \pinlabel{{\color{mybrown} $\frac{1}{2}$}} at 2710 420
    \pinlabel{{\color{mybrown} $-\frac{1}{2}$}} at 2620 600
    \pinlabel{{\color{mybrown} $-\frac{1}{2}$}} at 2520 180
    \endlabellist
    \includegraphics[width=\textwidth]{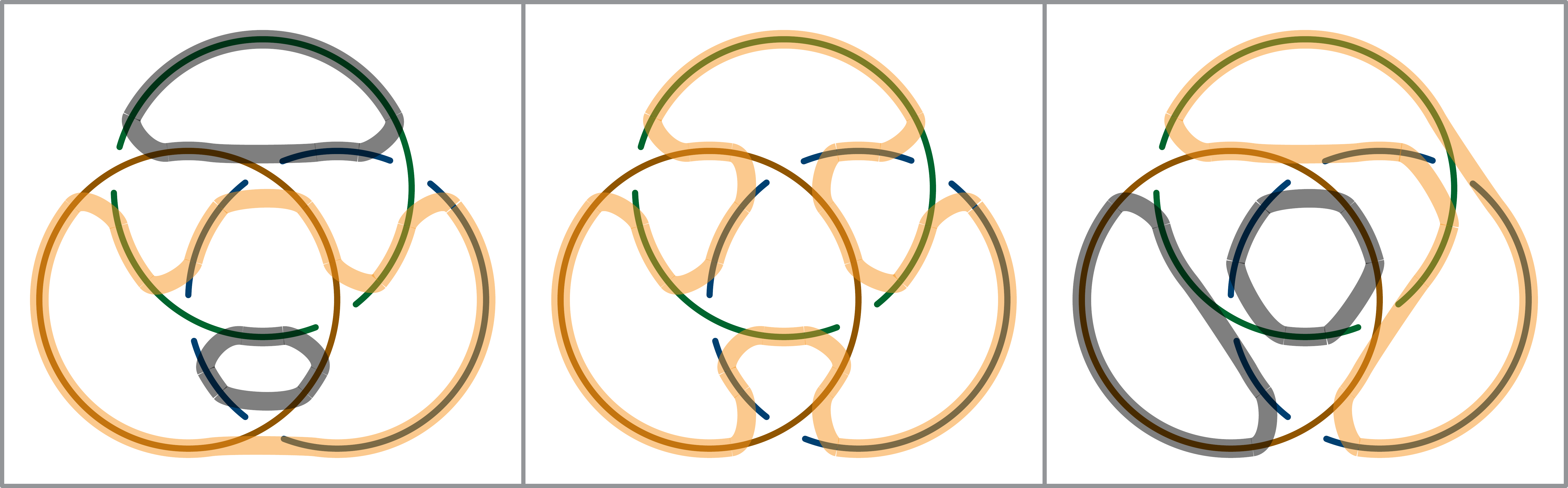}
    \caption{Three all-up layered states whose surfaces have net slopes of zero.}
    \label{fig:ThreeGoodStatesUp}
\end{figure}

\begin{enumerate}

\item Four of these states yield 2-sided surfaces; each is the Seifert state of $P$ under one of its four semi-orientations. Figure \ref{fig:ThreeGoodStatesUp} (A) shows such a state. (In this figure and in Figures \ref{fig:TwoBadStatesUp} and \ref{fig:ThreeGoodStatesDown}, the three link components from $P$ are drawn in three different colors, and the local contributions of each crossing to the boundary slope along its overpass and underpass are written in the appropriate color. Figures \ref{fig:CrossingBands1} and \ref{fig:CrossingBands2} may help clarify why these are the correct local contributions.) In each of these cases, we find that the all-up layering of each of these states has component-wise boundary slopes identically zero. \footnote{In the next subsection, we will see that the same will be true for any layering of these state surfaces. Each of these states can also be constructed from either checkerboard state of $P$ by reversing all three smoothings incident to an unshaded region.}

\item Four of the states of $P$, including the one in Figure \ref{fig:ThreeGoodStatesUp} (B), can be constructed as follows: choose a checkerboard shading of $P$, construct the associated state, choose one of the four shaded regions, and reverse all three smoothings around the chosen region. (Each such state can be constructed from either checkerboard shading of $P$, which is why there are four such states, rather than eight.) The all-up layering of each of these states has component-wise boundary slopes identically zero. In fact, we will see in the next subsection that {\it any} layering of these states gives us a positive answer.

\item Four of the states of $P$, including the one in Figure \ref{fig:ThreeGoodStatesUp} (C) can be constructed from a checkerboard state of $P$ by reversing the smoothing at one of the two crossings between the middle and top circles of the link. Figure \ref{fig:ThreeGoodStatesUp} (C) is one such state, the state obtained by reversing the smoothings of the two starred crossings is another, and reversing all smoothings in each of those two states gives the last two states presently under consideration.  As above, the all-up layering of each of these states has component-wise boundary slopes identically zero).

\item Four further states of $P$ are obtained in the same way, but with the role of the bottom circle assuming the role taken here by the top circle. Figure \ref{fig:TwoBadStatesUp} (A) shows one such state.  None of these states are viable because they each include a circle of slope 2 and another circle of slope -2.

\begin{figure}
    \centering
        \labellist\small
    \pinlabel{(A)} at 950 50
    \pinlabel{(B)} at 1960 50
    \tiny
    \pinlabel{{\color{mygreen} $\frac{1}{2}$}} at 180 530
    \pinlabel{{\color{mygreen} $\frac{1}{2}$}} at 850 620
    \pinlabel{{\color{mygreen} $\frac{1}{2}$}} at 590 260
    \pinlabel{{\color{mygreen} $\frac{1}{2}$}} at 340 410
    \pinlabel{{\color{myblue} $\frac{1}{2}$}} at 760 580  
    \pinlabel{{\color{myblue} $-\frac{1}{2}$}} at 400 600
    \pinlabel{{\color{myblue} $-\frac{1}{2}$}} at 500 70
    \pinlabel{{\color{myblue} $\frac{1}{2}$}} at 430 260
    \pinlabel{{\color{mybrown} $-\frac{3}{2}$}} at 300 700
    \pinlabel{{\color{mybrown} $\frac{1}{2}$}} at 680 420
    \pinlabel{{\color{mybrown} $-\frac{1}{2}$}} at 600 600
    \pinlabel{{\color{mybrown} $-\frac{1}{2}$}} at 500 180
    \pinlabel{{\color{mygreen} $\frac{1}{2}$}} at 1200 540
    \pinlabel{{\color{mygreen} $\frac{1}{2}$}} at 1840 640
    \pinlabel{{\color{mygreen} $\frac{1}{2}$}} at 1690 310
    \pinlabel{{\color{mygreen} $\frac{1}{2}$}} at 1340 420
    \pinlabel{{\color{myblue} $-\frac{1}{2}$}} at 1400 600  
    \pinlabel{{\color{myblue} $\frac{1}{2}$}} at 1780 580
    \pinlabel{{\color{myblue} $-\frac{1}{2}$}} at 1420 140
    \pinlabel{{\color{myblue} $\frac{1}{2}$}} at 1450 270

    \pinlabel{{\color{mybrown} $-\frac{3}{2}$}} at 1320 710
    \pinlabel{{\color{mybrown} $-\frac{3}{2}$}} at 1610 365
    \pinlabel{{\color{mybrown} $\frac{3}{2}$}} at 1620 140
    \pinlabel{{\color{mybrown} $-\frac{1}{2}$}} at 1610 610
    \endlabellist
    \includegraphics[width=.667\textwidth]
    {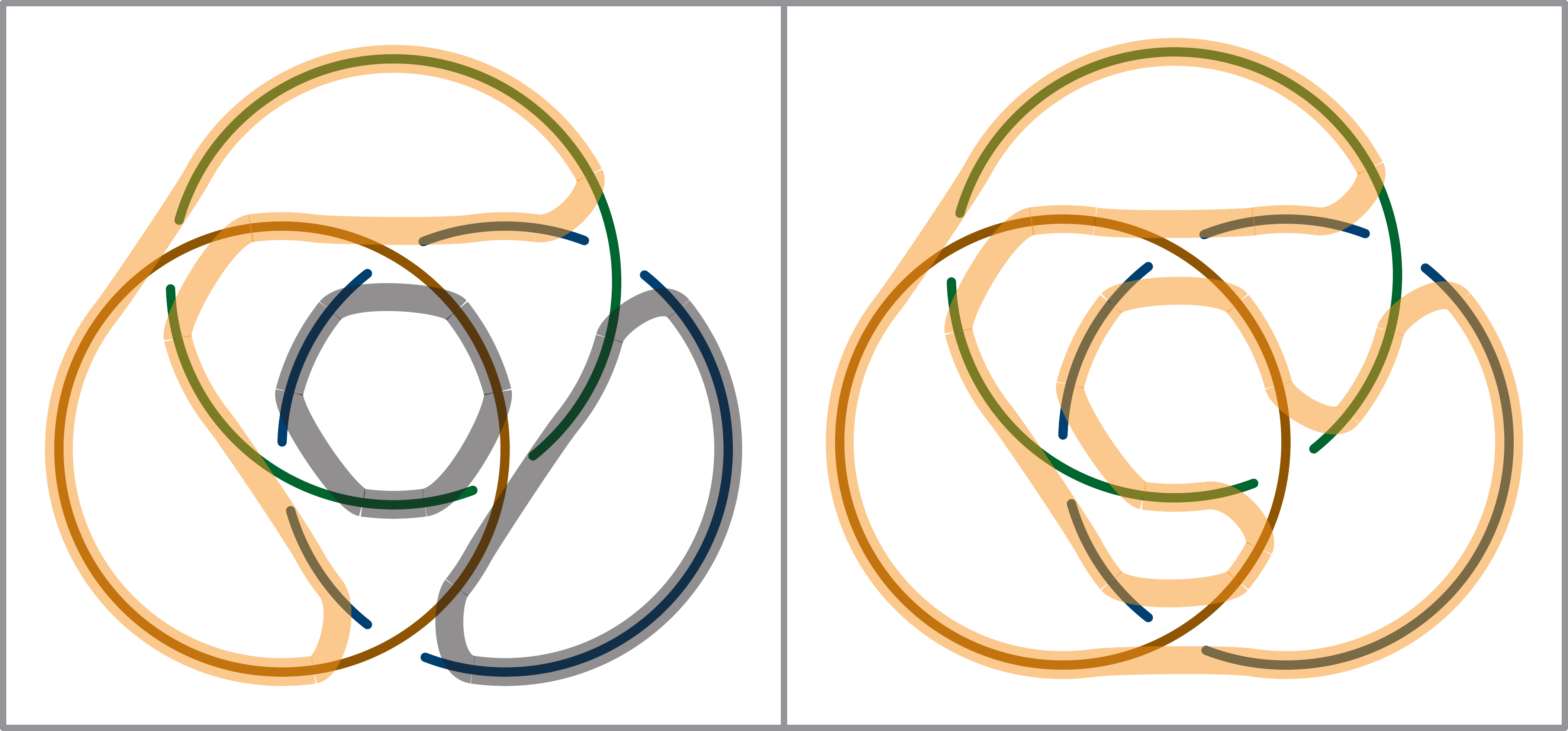}
    \caption{Two all-up layered states whose surfaces have net slopes, but not component-wise slopes, of zero.}
    \label{fig:TwoBadStatesUp}
\end{figure}
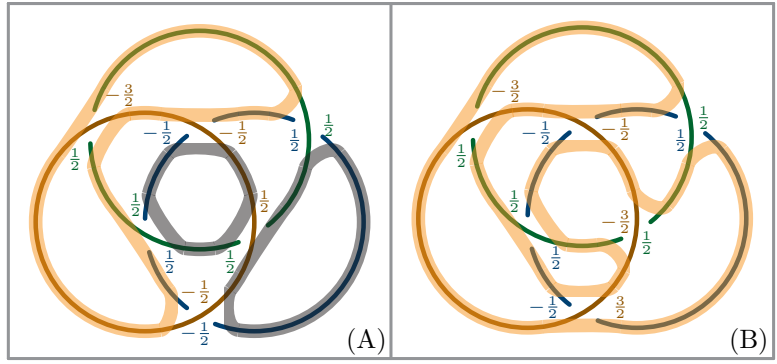

\item The remaining four candidate states can be constructed from either checkerboard state by reversing two non-adjacent smoothings incident to the middle circle and reversing the smoothing at one of the crossings between the top and bottom circles. Figure \ref{fig:TwoBadStatesUp} (B) shows one such state, which is non-viable because the slopes along the top and middle circles are -2 and 2, respectively. As in the previous case, none of these states are viable. 

\end{enumerate}

Thus, we have proven:

\begin{theorem}\label{thm:up_solid}
    A state of a broken surface diagram extends to an all-up state solid if and only if, up to diagrammatic symmetry, every triple point is smoothed in type (A) or (B) or (C) from Figure \ref{fig:ThreeGoodStatesUp}.    
\end{theorem}

Note that of the 64 possible smoothings near a triple point, four are of type (A), four are of type (B), and four are of type (C).

\begin{remark}
    If a broken surface diagram is oriented, then all smoothings of triple points that are consistent with the orientation are of type (A), up to diagrammatic symmetry.

    As a consequence of Theorem \ref{thm:up_solid}, a broken surface diagram of an oriented surface $S$ always admits a state solid that is a Seifert solid for $S$ (take the state that uses smoothings consistent with the orientation of the surface). This algorithm is presented differently than the Seifert algorithm of Carter--Saito  \cite{cartersaito_solids}, as triple points are handled differently.
    
    In Carter--Saito's work, a closed surface in $S^3$ is obtained by first attaching tubes to $S\subset S^4$ to modify the surface so its projection has no triple points or closed curves of double points. They then smooth arcs of intersection in the resulting projection to get an embedded surface in $S^3$ that bounds a solid, and reverse engineer the previous smoothing and tube attachments to obtain a Seifert solid for $S$. 
\end{remark}

\subsection{Viability of states and state solids: the general case}\label{sec:layered_solid}

\begin{figure}
    \centering
        \labellist\small
    \pinlabel{(A)} at 950 50
    \pinlabel{(B)} at 1960 50
    \pinlabel{(C)} at 2970 50
    \tiny
    \pinlabel{{\color{mygreen} $\frac{1}{2}$}} at 180 650
    \pinlabel{{\color{mygreen} $\frac{1}{2}$}} at 840 640
    \pinlabel{{\color{mygreen} $-1$}} at 690 310
    \pinlabel{{\color{mygreen} $0$}} at 320 310
    \pinlabel{{\color{myblue} $-\frac{1}{2}$}} at 390 600
    \pinlabel{{\color{myblue} $\frac{1}{2}$}} at 770 590
    \pinlabel{{\color{myblue} $1$}} at 450 120
    \pinlabel{{\color{myblue} $-1$}} at 400 350
    \pinlabel{{\color{mybrown} $\frac{1}{2}$}} at 260 590
    \pinlabel{{\color{mybrown} $0$}} at 600 350
    \pinlabel{{\color{mybrown} $-\frac{1}{2}$}} at 610 600
    \pinlabel{{\color{mybrown} $0$}} at 590 120
    \pinlabel{{\color{mygreen} $\frac{1}{2}$}} at 1200 650
    \pinlabel{{\color{mygreen} $\frac{1}{2}$}} at 1840 640
    \pinlabel{{\color{mygreen} $-\frac{3}{2}$}} at 1710 310
    \pinlabel{{\color{mygreen} $\frac{1}{2}$}} at 1340 310
    \pinlabel{{\color{myblue} $\frac{3}{2}$}} at 1525 670  
    \pinlabel{{\color{myblue} $\frac{1}{2}$}} at 1780 580
    \pinlabel{{\color{myblue} $-\frac{1}{2}$}} at 1520 70
    \pinlabel{{\color{myblue} $\frac{3}{2}$}} at 1420 370
    \pinlabel{{\color{mybrown} $\frac{1}{2}$}} at 1260 580
    \pinlabel{{\color{mybrown} $\frac{1}{2}$}} at 1630 355
    \pinlabel{{\color{mybrown} $-\frac{1}{2}$}} at 1515 560
    \pinlabel{{\color{mybrown} $-\frac{1}{2}$}} at 1520 180
    \pinlabel{{\color{mygreen} $-\frac{3}{2}$}} at 2200 500
    \pinlabel{{\color{mygreen} $\frac{1}{2}$}} at 2850 640
    \pinlabel{{\color{mygreen} $\frac{1}{2}$}} at 2700 310
    \pinlabel{{\color{mygreen} $\frac{1}{2}$}} at 2350 420
    \pinlabel{{\color{myblue} $-\frac{1}{2}$}} at 2400 600  
    \pinlabel{{\color{myblue} $\frac{1}{2}$}} at 2780 580
    \pinlabel{{\color{myblue} $-\frac{1}{2}$}} at 2420 133
    \pinlabel{{\color{myblue} $\frac{1}{2}$}} at 2460 260

    \pinlabel{{\color{mybrown} $\frac{1}{2}$}} at 2330 710
    \pinlabel{{\color{mybrown} $-\frac{3}{2}$}} at 2610 365
    \pinlabel{{\color{mybrown} $\frac{3}{2}$}} at 2620 133
    \pinlabel{{\color{mybrown} $-\frac{1}{2}$}} at 2610 600
    \endlabellist
    \includegraphics[width=\textwidth]{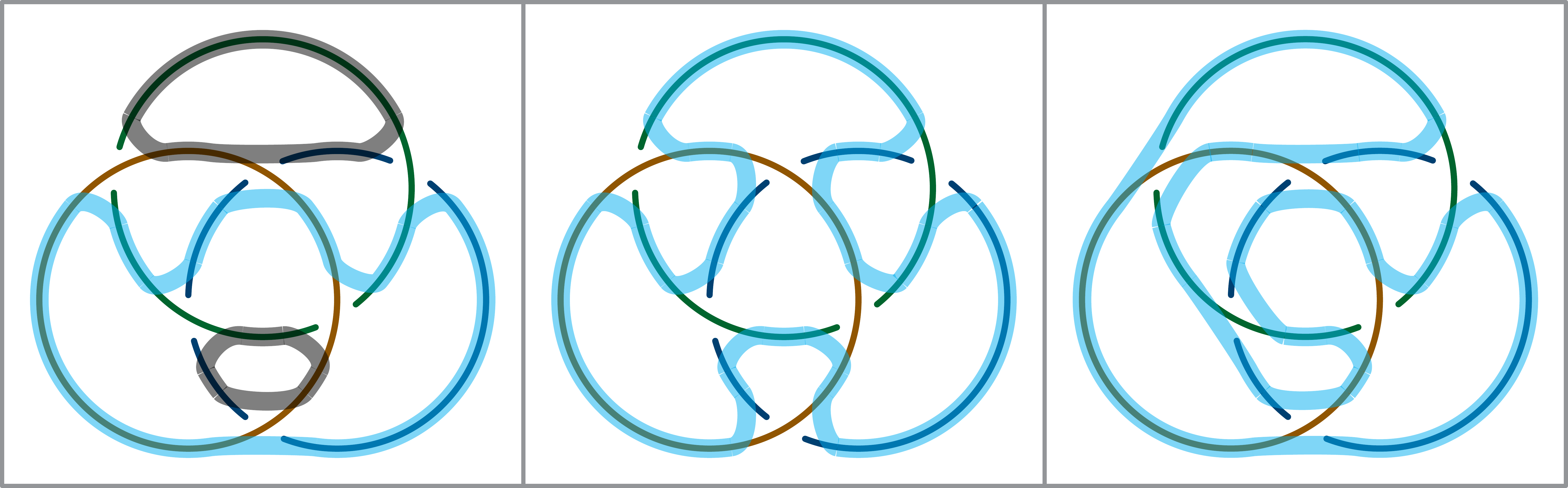}
    \caption{Three all-down layered states whose surfaces have net slopes of zero.}
    \label{fig:ThreeGoodStatesDown}
\end{figure}
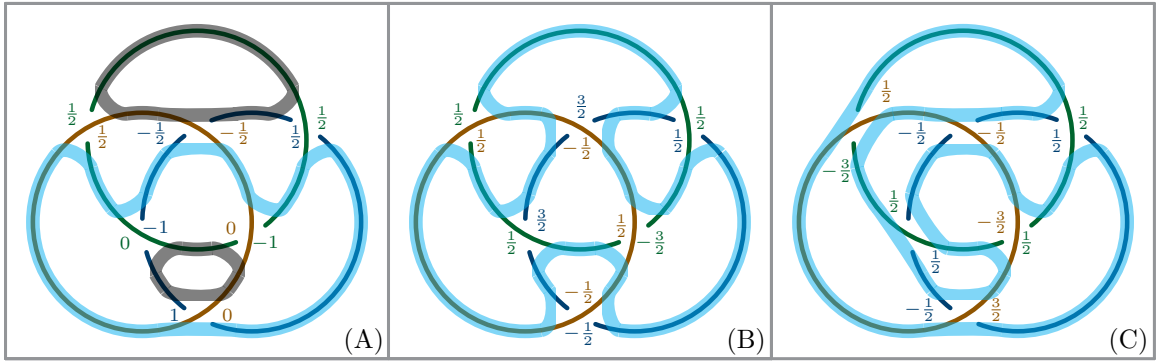

\begin{figure}
    \centering
    \labellist\small
    \pinlabel{(A)} at 950 50
    \pinlabel{(B)} at 1960 50
    \pinlabel{(C)} at 2970 50
    \endlabellist
    \includegraphics[width=\textwidth]{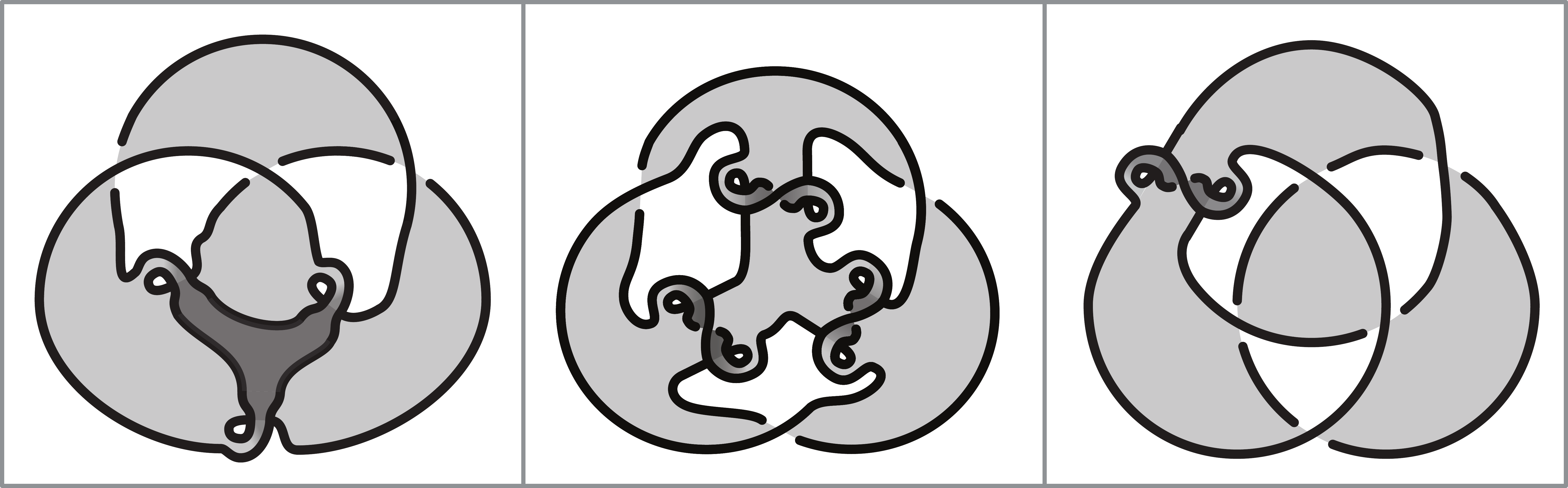}
    \caption{State surfaces (after small isotopies) that correspond to the layered states in Figure \ref{fig:ThreeGoodStatesDown}}
    \label{fig:ThreeStateSurfaces}
\end{figure}

\begin{figure}
    \centering
    \labellist\small
    \pinlabel{(A)} at 950 50
    \pinlabel{(B)} at 1960 50
    \pinlabel{(C)} at 2970 50
    \endlabellist
    \includegraphics[width=\textwidth]{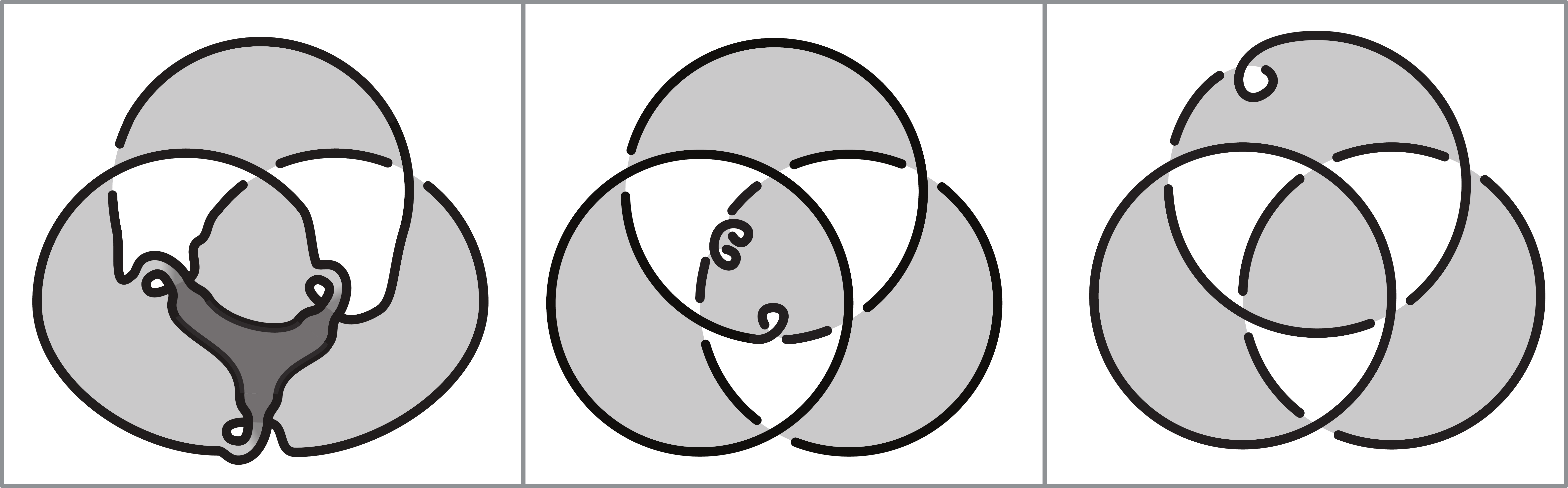}
    \caption{State surfaces (after small isotopies) that correspond to the layered states in Figures \ref{fig:ThreeGoodStatesDown}-\ref{fig:ThreeStateSurfaces}}
    \label{fig:ThreeStatesCrosscaps}
\end{figure}

Before stating the general theorem about which layerings of which smoothings near a triple point are viable in the construction of a state solid, we provide more explicit descriptions of the local state surfaces under consideration near a triple point. Figure \ref{fig:ThreeGoodStatesDown} shows three all-down layered states, analogous to the three all-up layered states from Figure \ref{fig:ThreeGoodStatesDown}, and Figures \ref{fig:ThreeStateSurfaces}-\ref{fig:ThreeStatesCrosscaps} each show the associated state surfaces.  The surfaces in Figures \ref{fig:ThreeStateSurfaces}-\ref{fig:ThreeStatesCrosscaps} are drawn in particular styles that may be easier to understand after consulting Figures \ref{fig:CrossingBands1}-\ref{fig:CrossingBands2}.

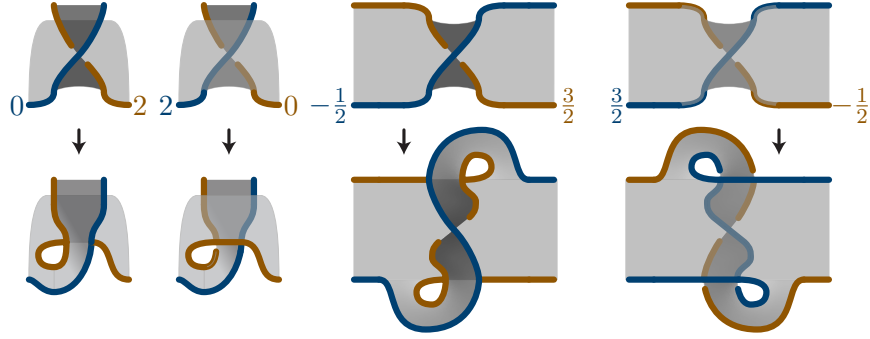
\begin{figure}
    \centering
    \labellist
    \pinlabel{{\color{myblue} $0$}} at -25 660
    \pinlabel{{\color{mybrown} $2$}} at 330 660
    \pinlabel{{\color{myblue} $2$}} at 405 660
    \pinlabel{{\color{mybrown} $0$}} at 765 660
    \pinlabel{{\color{myblue} $-\frac{1}{2}$}} at 875 660
    \pinlabel{{\color{mybrown} $\frac{3}{2}$}} at 1565 660
    \pinlabel{{\color{myblue} $\frac{3}{2}$}} at 1705 660
    \pinlabel{{\color{mybrown} $-\frac{1}{2}$}} at 2380 660
    \endlabellist
    \includegraphics[width=.7\textwidth]{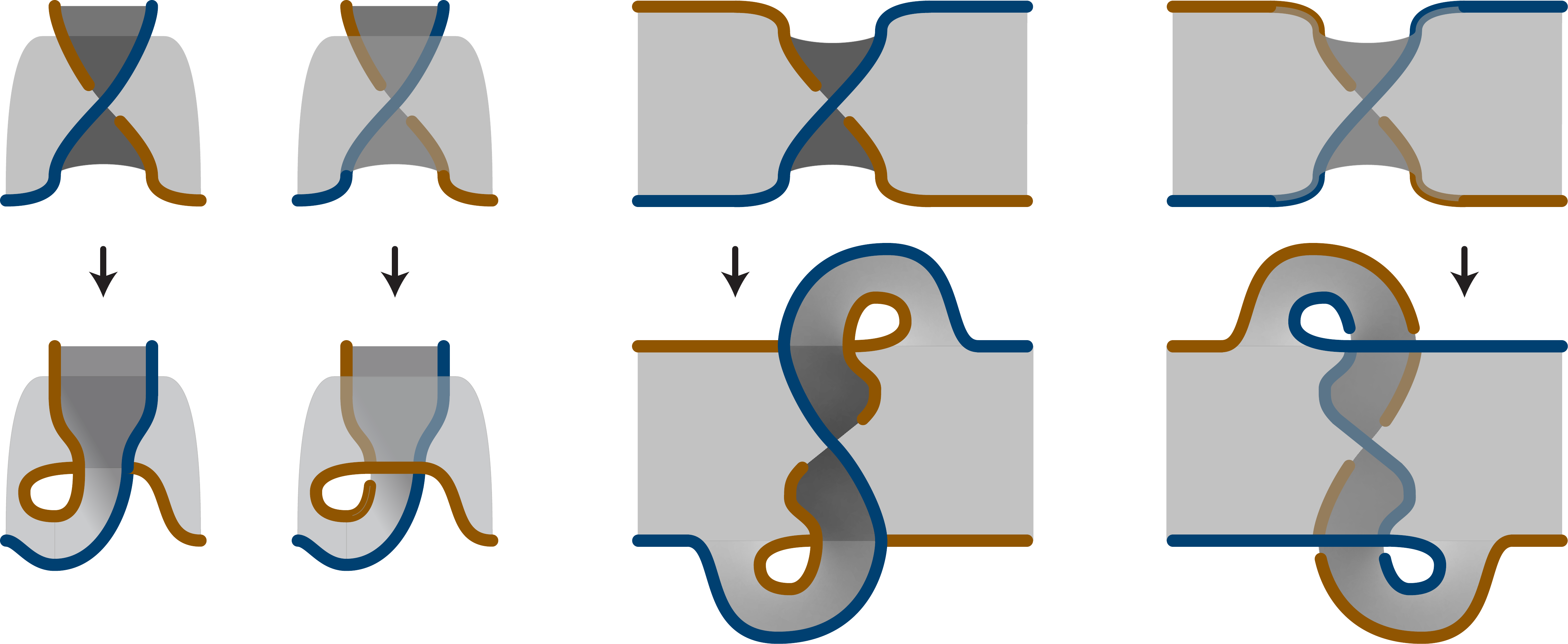}
    \caption{Isotopies of certain crossing bands}
    \label{fig:CrossingBands1}
\end{figure}

\begin{figure}
    \centering
    \labellist
    \pinlabel{{\color{myblue} $-\frac{1}{2}$}} at -60 160
    \pinlabel{{\color{mybrown} $\frac{3}{2}$}} at -30 430
    \pinlabel{{\color{myblue} $-\frac{1}{2}$}} at 3540 430
    \pinlabel{{\color{mybrown} $\frac{3}{2}$}} at 3510 160
    \endlabellist
    \includegraphics[width=\textwidth]{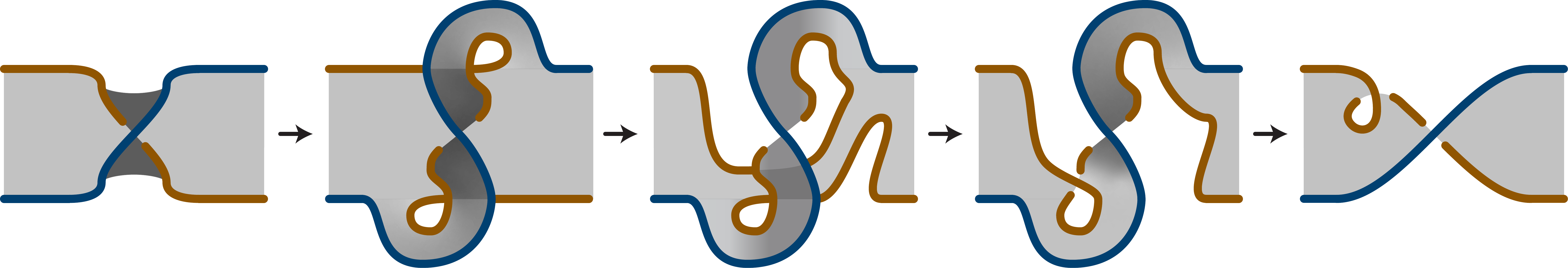}
    \caption{Visualization of the third isotopy from Figure \ref{fig:CrossingBands1}}
    \label{fig:CrossingBands2}
\end{figure}

An analysis similar to that used to prove Theorem \ref{thm:up_solid} yields:

\begin{theorem}\label{thm:state_solid}
    A state $X$ of a broken surface diagram $E$ extends to:
    \begin{enumerate}[label=(\arabic*)]
    \item \label{item:both} Both an all-up state solid and an all-down state solid if and only if, $E$ has no triple points.
    \item An all-down (resp. all-up) state solid if and only if, up to diagrammatic symmetry,  every triple point is smoothed in type (A), (B), or (C) from Figure \ref{fig:ThreeGoodStatesDown} (resp. Figure \ref{fig:ThreeGoodStatesUp}).    
    \item \label{item:updown} An up-down state solid if and only if, up to diagrammatic symmetry, every triple point is smoothed in type (A) or (B) or (C) from Figure \ref{fig:ThreeGoodStatesUp} or \ref{fig:ThreeGoodStatesDown}, and no component of $X$ contains both an orange (light) circle from Figure \ref{fig:ThreeGoodStatesUp} and a blue (light) circle from Figure \ref{fig:ThreeGoodStatesDown}.
    \item\label{item:allstates} A state solid only if, up to diagrammatic symmetry, every triple point is smoothed in one of the ways shown in Figure \ref{fig:ThreeGoodStatesUp}, \ref{fig:ThreeGoodStatesDown}, or \ref{fig:Relayered}. In Figure \ref{fig:Relayered}, we also allow the smoothing obtained by viewing the picture from the opposite side of the projection sphere.\footnote{Figure \ref{fig:ThreeGoodStatesDown} may also be seen as Figure \ref{fig:ThreeGoodStatesUp} from the opposite side of the projection sphere.}
    \end{enumerate}
\end{theorem}

One may check that of the 64 states of $P$, 20 are viable in the sense that they yield surfaces with net slopes of zero, and of these, eight qualify for part (1) of Theorem \ref{thm:state_solid}, twelve qualify for part (2), 16 qualify for part (3), and all 20 qualify for part (4).

\begin{remark}
    While parts \ref{item:both}-\ref{item:updown} in Theorem provide necessary and sufficient conditions, part \ref{item:allstates} \ref{thm:state_solid} provides a necessary, but not sufficient condition. The issue is that if some component $X_0$ of a state $X$ intersects northern hemispheres of some components of $\partial C_t$ and southern hemispheres of others, then the interior of the solid that caps off $X_0$ will always intersect $S^3$, and sometimes this intersection will be forced to include points in $K$, which is impermissible. 
    
    A sufficient condition would thus be along the following lines. For each component $X_0$ of $X$, there is a (possibly empty) system of circles $\gamma\subset X_0\cap S^3$ such that 
    (i) no component of $X_0\setminus \gamma$ intersects northern hemispheres of some components of $\partial C_t$ and southern hemispheres of others; (ii) each component $\gamma_0$ of $\gamma$ is incident to two distinct components of $X_0\cut \gamma$, one intersecting northern hemispheres of $\partial C_t$  and the other intersecting southern hemispheres; and (iii) for (at least) one of the components $Y$ of $\widehat{S^3}\cut X_0$, there is a properly embedded surface $F$ in $Y\setminus \partial C$ with $\partial F=\gamma$.  
    
    In that case, for each component $Y_0$ of $Y\cut F$, delete all southern (resp. northern) hemispheres of $\partial C_t$ in $Y_0$ unless $\partial Y_0$ intersects such southern (resp. northern) hemisphere, and then, fixing $\partial Y_0\subset X_0\cup F$, push the interior of $Y_0$ into whichever 4-ball of $S^4\setminus (S^3\cup\mathring{C})$ is incident to $\partial Y_0$ (there will be only one such choice unless $Y_0=Y$, in which case either 4-ball is fine). Doing this for each component of $Y\cut F$ yields a spanning solid $M_X$.
\end{remark}

Recall from Remark \ref{rem:state_plumb_3D} that state surfaces in $S^3$ decompose naturally under Murasugi sum.  We now note, with caveats, that state solids behave analogously.

\begin{remark}\label{rem:state_solid_deplumb}
    Let $X$ be a state of a broken surface diagram $E$ of a knotted surface $K\subset S^4$, let $M_X$ be an associated state solid, and let $X_0$ be a component of $X$.  Suppose that $X_0$ is a 2-sphere, and that the interior of the ball $U$ in $M_X$ that caps it off lies entirely on one side of $S^3$.  Let $V$ be the 3-ball obtained by reflecting $U$ across $S^3$, write $Q=U\cup V$, write $B^4_1$ and $B^4_2$ for the components of $S^4\cut Q$, and write $M_i=M_X\cap B^4_i$, $i=1,2$.  Then $M=M_1\cup_UM_2$.  Further, this decomposition is a Murasugi sum if
    \begin{equation}
        \label{E:sum_condition}\tag{$*$}
        \text{wherever }X_0\text{ intersects the neighborhood }T\text{ of a triple point of }E,~X_0\cap T\subset S^3,
    \end{equation} so that $\partial X$ looks like one of the gray circles in Figure \ref{fig:ThreeGoodStatesUp} or Figure \ref{fig:ThreeGoodStatesDown}. The key point specified by condition \eqref{E:sum_condition} is that we need $\partial U\subset S^3$ in order for $U$ and $V$ to share a boundary, and thus for $U\cup V$ to be a 3-sphere.
\end{remark}

One might wonder if, given a knotted surface $L$ with normal Euler number 0, whether every broken surface diagram of $L$ has a state that yields a state solid. The answer is no. For example, take a connected sum of a positive projective plane and a negative projective plane, and construct a broken surface diagram that consists of two arcs, each with two cusps of matching signs as its endpoints. However, this question becomes more interesting when modified in either of the two following ways.

\begin{restatable}{question}{statewithsolidconnectsum}\label{Q:state_has_solid}
    Let $L$ be a knotted surface, and let $E$ be a broken surface diagram of $L$ such that, if you decompose $E$ under diagrammatic connect sum, each resulting component has normal Euler number 0.  Does $E$ have a state that yields a state solid?
\end{restatable}

In case the answer to Question \ref{Q:state_has_solid} turns out to be no, we offer the following weaker question.

\begin{restatable}{question}{qp}\label{Q:State_Solid_Always_Exists_2}
    Let $L$ be a knotted surface with normal Euler number 0. Does $L$ have a broken surface diagram that has a state that yields a state solid?
\end{restatable}

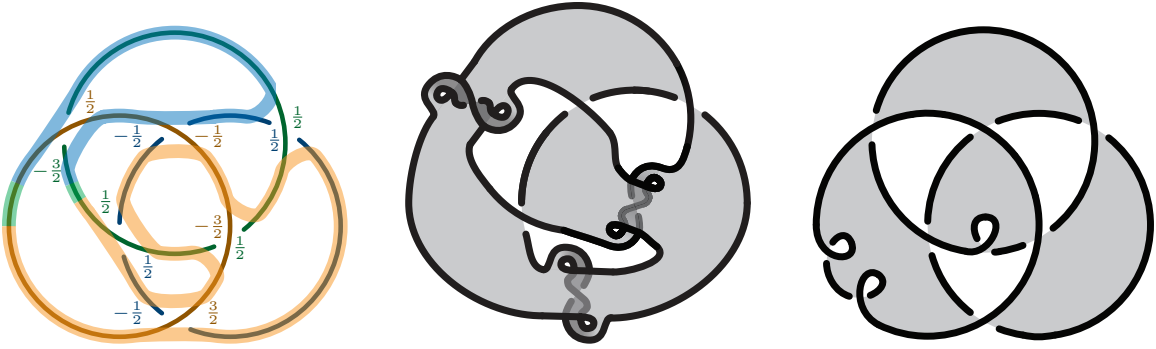
\begin{figure}
    \centering
    \labellist\tiny
    \pinlabel{{\color{mygreen} $-\frac{3}{2}$}} at 120 450
    \pinlabel{{\color{mygreen} $\frac{1}{2}$}} at 770 590
    \pinlabel{{\color{mygreen} $\frac{1}{2}$}} at 620 250
    \pinlabel{{\color{mygreen} $\frac{1}{2}$}} at 270 370
    \pinlabel{{\color{myblue} $-\frac{1}{2}$}} at 330 540  
    \pinlabel{{\color{myblue} $\frac{1}{2}$}} at 710 530
    \pinlabel{{\color{myblue} $-\frac{1}{2}$}} at 330 80
    \pinlabel{{\color{myblue} $\frac{1}{2}$}} at 380 200

    \pinlabel{{\color{mybrown} $\frac{1}{2}$}} at 230 630
    \pinlabel{{\color{mybrown} $-\frac{3}{2}$}} at 540 305
    \pinlabel{{\color{mybrown} $\frac{3}{2}$}} at 550 80
    \pinlabel{{\color{mybrown} $-\frac{1}{2}$}} at 540 540
    \endlabellist
    \includegraphics[width=0.3\textwidth]{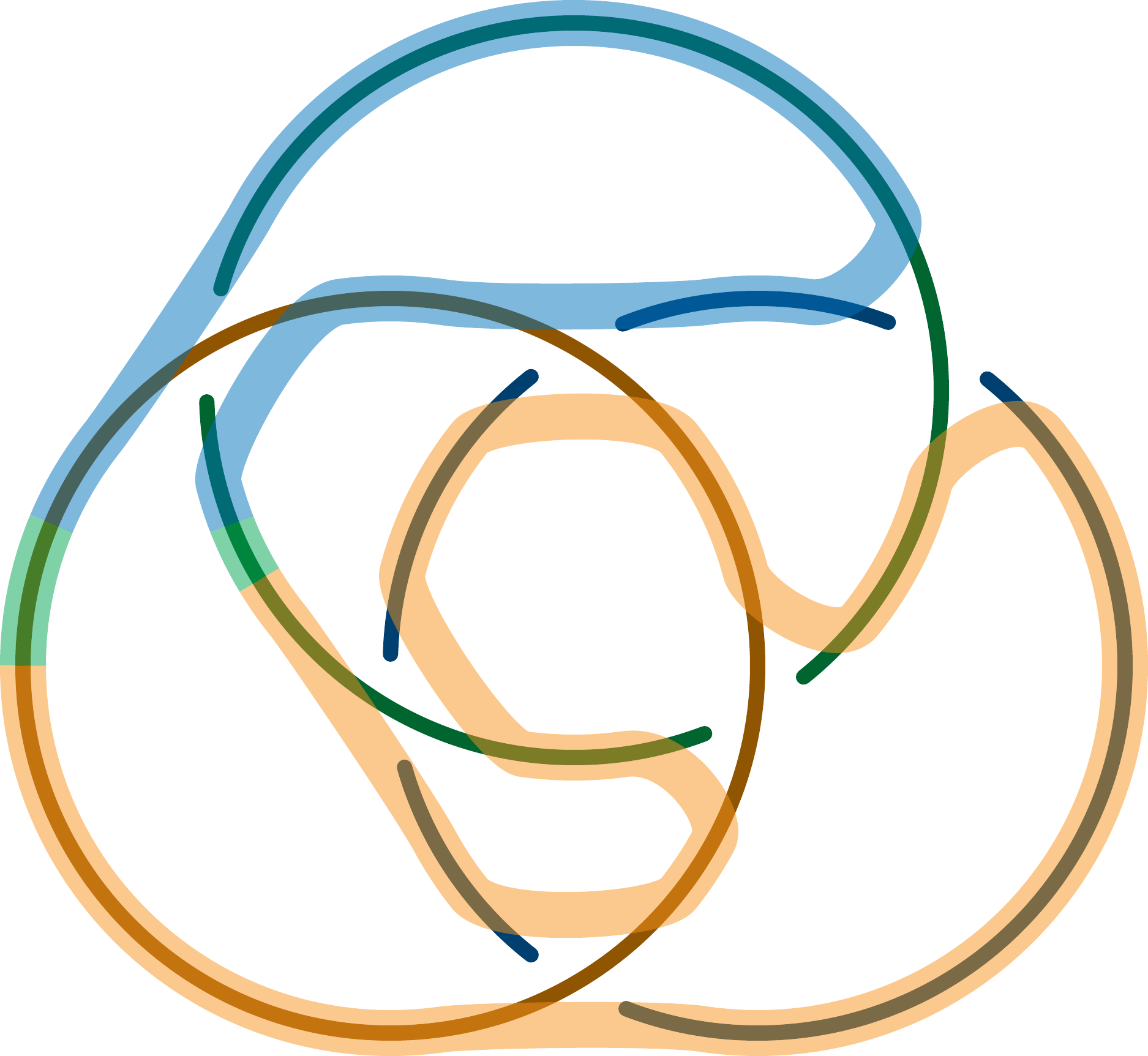}\hfill
    \includegraphics[width=0.3\textwidth]{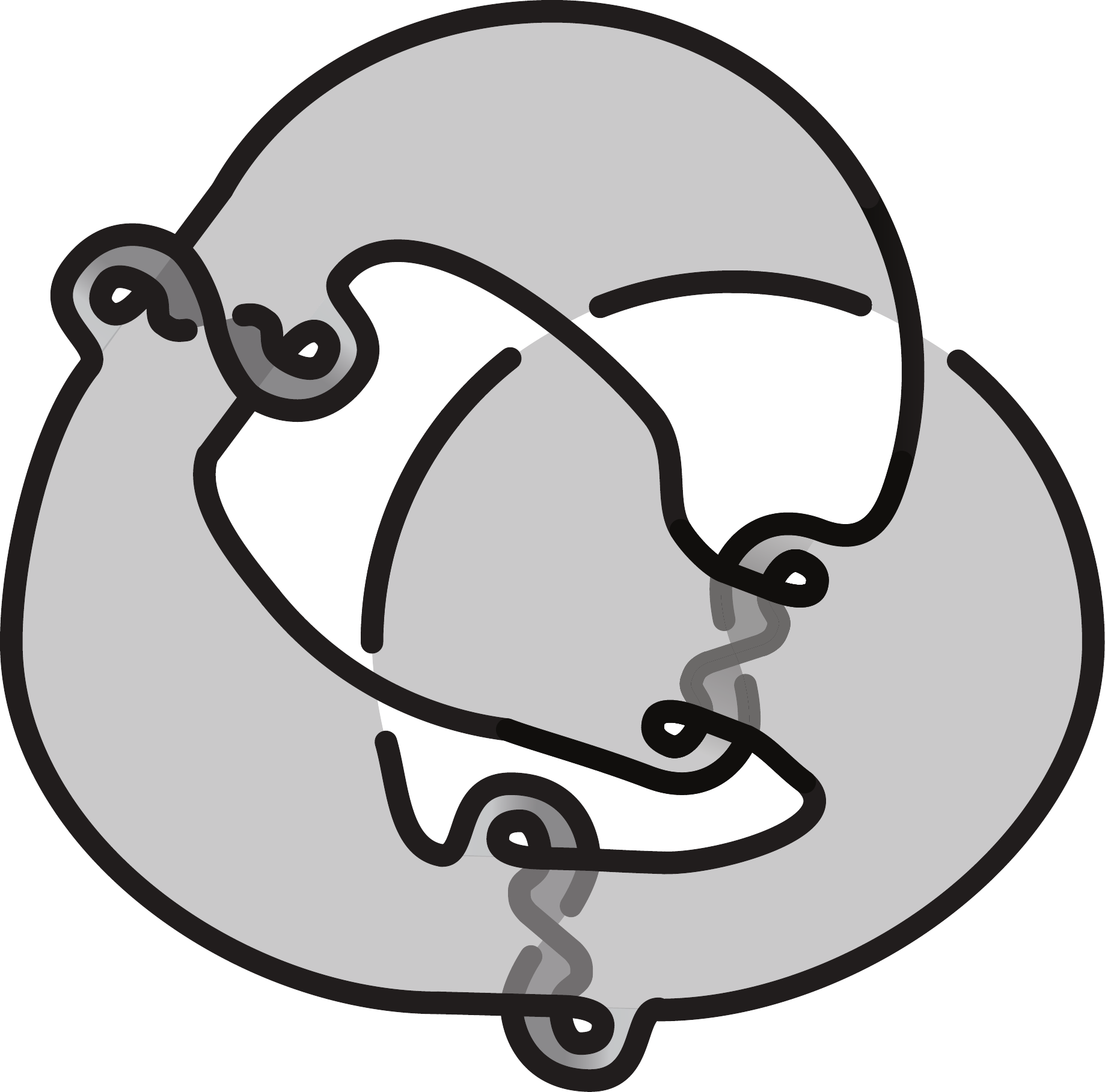}\hfill
    \includegraphics[width=0.3\textwidth]{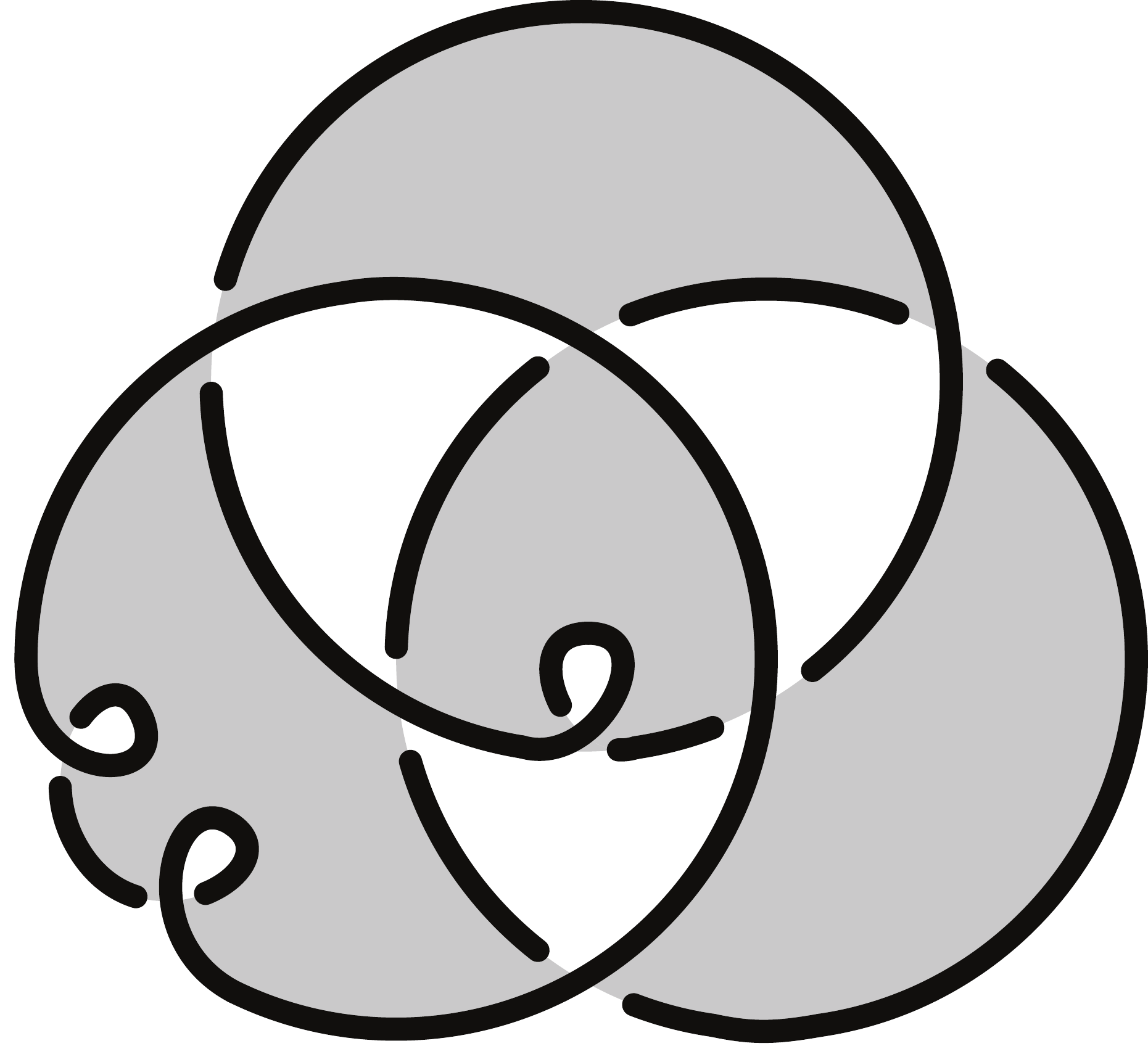}
    \caption{Relayering state (B) from Figure \ref{fig:TwoBadStatesUp} to obtain a state surface with component-wise slopes identically zero.}
    \label{fig:Relayered}
\end{figure}

\section{Pseudo-ribbon diagrams and their state solids}\label{sec:ribbon}

In this section, we focus on state solids obtained from pseudo-ribbon diagrams. 

\begin{definition}\label{def:pseudo_ribbon}
    A \emph{pseudo-ribbon diagram} is a broken surface diagram $E$ with no cusps or triple points, i.e., its crossing points consist of circles of double points. We call a knotted surface {\it pseudo-ribbon} if it has such a diagram.  
\end{definition}

Pseudo-ribbon surfaces are sometimes also called {\it simply knotted} in the literature (see e.g., \cite[\S2.2]{carter_kamada_saito}). Recall that by a result of Yajima \cite{yajima1964}, pseudo-ribbon 2-knots are, in fact, ribbon. 

\begin{proposition}\label{prop:cb_shading}
    Suppose that $E\subset S^3$ is a diagram of a knotted surface $K\subset S^4$ such that $E$ has no triple points or cusps. Then the components of $S^3\cut E$ are checkerboard colorable: one can assign each component a ``light'' or ``dark'' shading so that components of the same shade abut only at crossing circles.
\end{proposition}

\begin{proof}
    Consider any crossing in $E$, and choose either smoothing of it.  The resulting diagram $E'$ is pseudo-ribbon, and $E'$ is checkerboard colorable if and only if $E$ is.  The proposition thus follows by induction on the number of crossing circles.
\end{proof}

\subsection{Murasugi sums of spun trivial M\"obius bands}\label{sec:trivial Mob}

In this section, we prove Theorem \ref{thm:ribbonstate}, which states that every pseudo-ribbon diagram has a state solid which is a particularly nice Murasugi sum of copies of the spun trivial M\"obius band $Y_1$. We prepare with the following observation, which is in some sense analogous to the classical fact that boundary summing a trivial M\"obius band onto a spanning surface does not change the boundary of the surface, up to isotopy. 

\begin{observation}
\label{obs:Y1_1_stripe}
    If $M=M_0*Y_1$ is a 1-stripe plumbing of a spanning solid $M_0$ with the trivial spun M\"obius band $Y_1$ in which $Y_1$ is attached along a thickening of the spin $d$ of the strip indicated in Figure \ref{fig:Y1_one_stripe_sphere}, then $\partial M$ and $\partial M_0$ are smoothly equivalent surfaces in $S^4$. Note that the core of the plumbing region in $Y_1$ is a disk (and thus the core of the plumbing region in $M_0$ is an arc).

    \begin{figure}[ht]
        \begin{center}
            \labellist\small \hair 4pt
            \pinlabel{{\color{red} $d$}} at 22 142
            \endlabellist
            \includegraphics[height=1.5in]{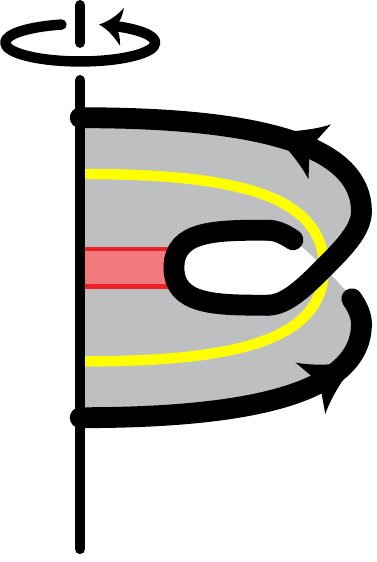}
            \caption{A spanning solid $M_0$ has boundary equivalent to that of the solid $M=M_0*Y_1$ whenever this Murasugi sum is a 1-stripe plumbing in which the plumbing region in $Y_1$ is the spin of the indicated stripe.}
            \label{fig:Y1_one_stripe_sphere}
        \end{center}
    \end{figure}
\end{observation}

We will make use of Observation \ref{obs:Y1_1_stripe} several times in \textsection\ref{sec:trivial Mob}. For convenience, if $Y_1$ is Murasugi summed to another solid as in Observation \ref{obs:Y1_1_stripe}, we may say that $Y_1$ is ``attached using $d$.''

\begin{proposition}\label{prop:2_knot_one_side}
    Suppose that $E\subset S^3$ is a pseudo-ribbon diagram of a knotted surface $K\subset S^4$ and $E$ has a state $X$ that is a 2-sphere. If all crossings of $E$ lie on the same side of $X$, then the associated spanning solid $M_X$ is a 1-stripe Murasugi sum of spun trivial M\"obius bands $Y_1$ of the form in Observation \ref{obs:Y1_1_stripe}. It follows that  $K$ is the trivial 2-knot.
\end{proposition}

\begin{proof}
    Let $B$ denote the 3-ball of $S^3\cut X$ that contains all the crossing annuli for the state $X$, and write $Z=S^3\cut B$. For each crossing annulus $A$, one of the components $B_A$ of $B\cut A$ is a 3-ball.  
    We say that $A'$ lies ``inside" $A$ if $A'\subset B_A$.  Note that for any two distinct crossing annuli $A$ and $A'$, it is possible that $A$ lies inside $A'$, or vice versa, or neither, but not both.
    
    If necessary, isotope $M_X$ through 4-space so that for each annulus $A$, the solid torus $Z\cup B_A$ is unknotted in $S^3$. (Note that this may change $E$ and $X$, which is fine because we aim only to characterize the equivalence class of $M_X$.) 
    One way to effect this isotopy is find an outermost annulus $A$ for which $Z\cup B_A$ is knotted in $S^3$, isotope $B_A$ (and all crossing annuli within it) through 4-space to unknot $Z\cup B_A$, and then repeat,  proceeding inward.  Note that this isotopy preserves the nesting structure among the crossing annuli. The 3-ball used to cap off the state solid is isotopic rel.\ boundary to $Z$, so now take it to coincide with $Z$.

    Now consider an innermost crossing annulus $A$, and let $C_A$ denote the associated crossing band. While fixing $Z$ in $S^3$, push $C_A$ into the 4-ball $B^4_+$ above $S^3$, and push the rest of $M_X$ into the 4-ball $B^4_-$ below $S^3$. Cutting $M_X$ along $S^3$ now decomposes $M_X$ as a 1-stripe plumbing of $Z\cup C_A\cong Y_1$ with $M_X\cut C_A$ of the form in Observation \ref{obs:Y1_1_stripe}.  

    Now apply the preceding argument to an innermost crossing annulus within $M_X\cut C_A$ and repeat inductively. This procedure terminates when $M_X\cut (C_{A_1}\cup \cdots C_{A_n})$ contains no more crossing annuli and is thus isotopic to $Z$.  

    Observation \ref{obs:Y1_1_stripe} implies that $K$ is the trivial 2-knot. 
\end{proof}

\begin{proposition}\label{prop:connected_state_2_knot}
    If $E\subset S^3$ is a pseudo-ribbon diagram of a knotted surface $K\subset S^4$ and $E$ has a state $X$ that is a 2-sphere, then the associated spanning solid $M_X$ (with any layering of the 3-ball that caps off $X$) is a Murasugi sum of two solids, each of which is a 1-stripe Murasugi sum of spun trivial M\"obius bands $Y_1$, in which each band $Y_1$ is attached along $d$, as in Observation \ref{obs:Y1_1_stripe}.
\end{proposition}

\begin{proof}
    Let $B$ be a 3-ball in $S^3$ bounded by $X$, and let $U\subset M_X$ denote the 3-ball that caps off $X$. Then $M_X$ is the union of $U$ and several crossing bands, some of which lie in $B\times I$ and some of which lie in $\left(S^3\setminus\mathring{B}\right)\times I$. Isotope $M_X$ (and $K$) so that $M_X\cap S^3=U$, while all crossing bands that were initially inside $B$ now lie in the 4-ball $B^4_+$ above $S^3$ and all other crossing bands now lie in the 4-ball $B^4_-$ below $S^3$.  Now cutting along $S^3$ decomposes $M_X$ into a Murasugi sum of two solids $M_\pm=M_X\cap B^4_\pm$. Each summand $M_\pm$ is isotopic to a state solid obtained from $M_X$ by keeping only those crossing bands projecting to one side of $X$ and deleting those projecting to the other side of $X$. Proposition \ref{prop:2_knot_one_side} thus implies that each $M_\pm$ is a 1-stripe Murasugi sum of spun trivial M\"obius bands, in which each band $Y_1$ is attached using $d$ as in Observation \ref{obs:Y1_1_stripe}. 
\end{proof}

\begin{theorem}\label{thm:ribbonstate}
    Every pseudo-ribbon diagram $E$ of a 2-knot $K$ has a state solid that is a Murasugi sum of two spanning solids for the trivial 2-knot, each of which is a 1-stripe Murasugi sum of spun trivial M\"obius bands $Y_1$, in which each copy of $Y_1$ is attached using the region $d$ as in Observation \ref{obs:Y1_1_stripe}.

    In particular, every ribbon 2-knot has a spanning solid that is a Murasugi sum of some number of copies of $Y_1$, the spun trivial M\"obius band.
\end{theorem}

\begin{proof}
    By Proposition \ref{prop:connected_state_2_knot} and the fact that $\chi(K)=2$, it will suffice to prove that $E$ has a connected state.  We will prove this by induction on the number of crossings in $E$.

    Let $f:S^2\to S^4$ denote an embedding of $K$ obtained by perturbing $E$ near its crossing circles.  Choose any crossing circle $\gamma\subset E$.  Then $f^{-1}(\gamma)$ cuts $S^2$ into two disks and an annulus.  One smoothing of $\gamma$ glues the two disks together;  choose and perform the opposite smoothing.  This yields a pseudo-ribbon diagram of another 2-knot with fewer circles of intersection. The proof now follows by induction.
\end{proof}

In order for a pseudo-ribbon diagram $E$ of a knotted surface $K$ to yield the type of state solid described in Theorem \ref{thm:ribbonstate}, it is necessary that $\chi(K)=2$.  This Euler characteristic condition is not, however, sufficient, as the diagram of $K$ may fail to have a connected state. Figure \ref{fig:torus_sphere_split} shows two such examples; in both cases, the diagrams have only disconnected states.

\begin{figure}[h]
    \centering
    \;\hfill\includegraphics[width=0.25\textwidth]{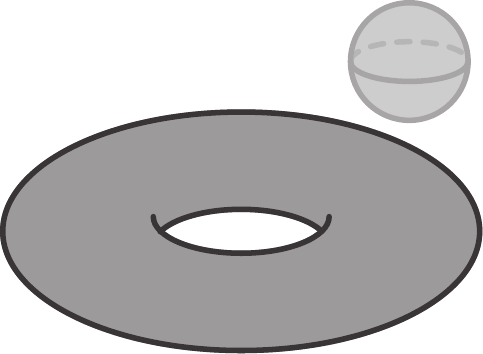}\hfill \includegraphics[width=0.25\textwidth]{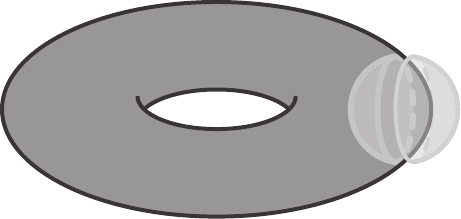}\hfill\;
    \caption{Two broken surface diagrams with no connected states}
    \label{fig:torus_sphere_split}
\end{figure}

\begin{remark}\label{remark:splitexample}
    Both examples in Figure \ref{fig:torus_sphere_connected} are split links. One might speculate that no split link in $S^4$ has the type of spanning solid described in Theorem \ref{thm:ribbonstate}. This, however, is false. For example, take the split union $K$ of an unknotted sphere and an unknotted torus. The surface $K$ has a crossingless diagram $D$ in which the torus component is unknotted in $S^3$. Take a properly embedded annulus $A$ in $S^3\cut D$ whose boundary consists of a circle on the sphere and an essential circle on the torus. Note that this annulus exists because the torus is unknotted in $S^3$. As shown in Figure \ref{fig:torus_sphere_connected}, isotope one component of $D$ through $\nu A$ to get a broken surface diagram of $K$ with two circles of self-intersection.  This diagram has a connected state, shown right in the figure, which yields the type of spanning solid described in Theorem \ref{thm:ribbonstate}.
\end{remark}

\begin{figure}[h]
    \centering
    \raisebox{0.05\textwidth}{\includegraphics[height=0.12\textwidth]{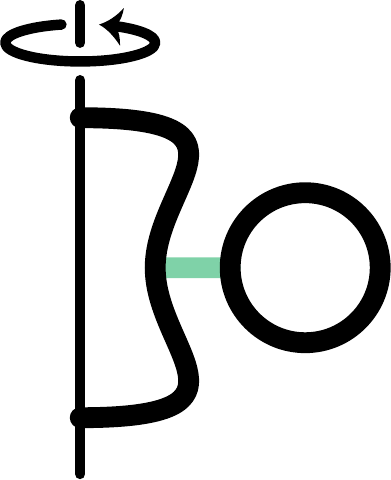}}
    \raisebox{0.1\textwidth}{$=$}
    \includegraphics[height=0.20\textwidth]{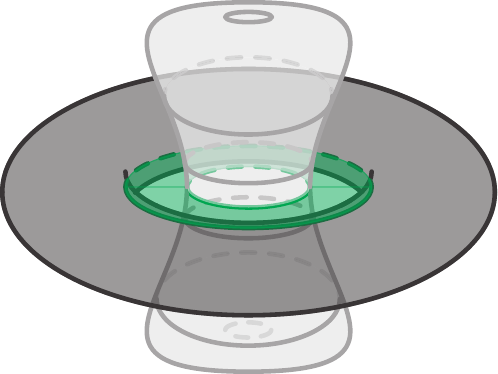} 
    \raisebox{0.1\textwidth}{$\to$}
    \includegraphics[height=0.20\textwidth]{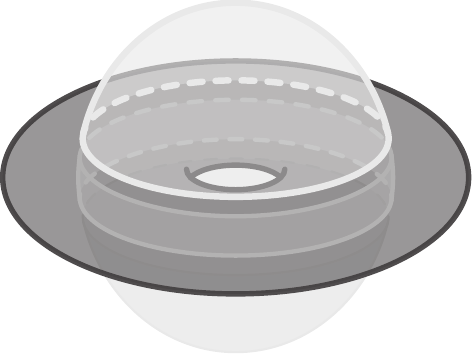}
    \raisebox{0.1\textwidth}{$=$}
    \raisebox{0.05\textwidth}{\includegraphics[height=0.12\textwidth]{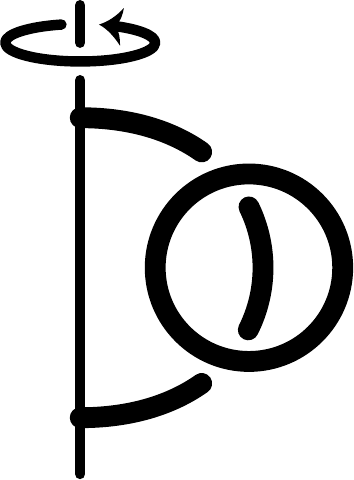}}
    \raisebox{0.1\textwidth}{$\to$}
    \raisebox{0.05\textwidth}{\includegraphics[height=0.12\textwidth]{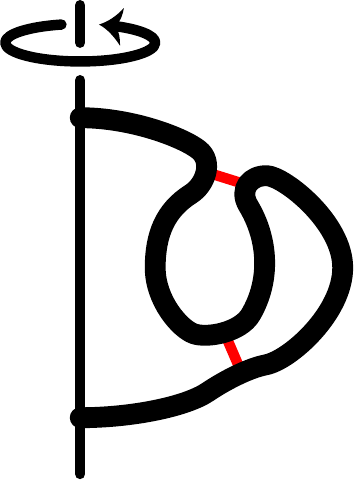}}
    \caption{Constructing a broken surface diagram (right) of the split {knotted surface} from Figure \ref{fig:torus_sphere_split} that has a connected state}
    \label{fig:torus_sphere_connected}
\end{figure}

\begin{remark}
    One might then suspect that every pseudo-ribbon surface $K$ with $\chi(K)=2$ bounds a spanning solid that is a Murasugi sum of copies of $Y_1$. As an attempt at proving this to be true, consider the following start to an argument. 

    Let $D\subset S^3$ be a pseudo-ribbon diagram of a  knotted surface $K\subset S^4$ with $\chi(K)=2$. Choose a state $X$ of $D$ for which $|X|$ is minimal.  Because $|X|$ is minimal, every crossing annulus for $X$ has the property that its two boundary circles either lie on the same closed component of $X$ or are each separating on their respective component of $X$ (or else we could change the smoothing of $X$ at that annulus and decrease $|X|$).

    Now we aim to find an annulus like the one used in Remark \ref{remark:splitexample} --- if we can find an annulus $A$ that is properly embedded in $S^3\cut X$ such that $A$ is disjoint from all crossing annuli for $X$ (including also the fact that $\partial A$ is disjoint from the boundaries of these crossing annuli), and such that the two circles of $\partial A$ are on distinct closed components of $X$ with at least circle being one nonseparating, then we can isotope one component of $D$ through $\nu A$ to create a new diagram $D'$ (with two more crossing circles than $D$) which has a state $X'$ with $|X'|=|X|-1$. If we can repeat this process indefinitely, doing so will eventually yield a pseudo-ribbon diagram of $K$ with a connected state, and the associated state solid will be the type we seek to construct. It turns out, however, that such an annulus need not exist.
\end{remark}

\begin{example}\label{ex:no_spun_mobius}
    Let $D$ be the disconnected pseudo-ribbon diagram constructed in Figure \ref{fig:torus_1_1}---the idea is to place a 2-crossing diagram of the Hopf link in the $x$-positive half of the $xz$-plane and to spin the diagram about the $z$-axis while simultaneously turning one circle once about the center of the other. One component of $D$ is a crossingless diagram of the trivial 2-knot. The other is a 2-crossing diagram of a link of two unknotted tori; the projection of each torus is unknotted in $S^3$, and both of their crossing circles are $(1,1)$ curves on both projected tori.

    This diagram $D$ has two 3-component states and two 2-component states; each of the latter comes from taking a connected state of the 2-crossing diagram of the Hopf link and roll-spinning it around the $z$-axis.  
    Each of those 2-component states consists of a sphere, an unknotted torus $T$, and a crossing annulus whose boundary circles $\gamma_1,\gamma_2$ both lie on $T$ as $(1,1)$ curves.  Importantly, every circle on $T$ either separates $T$, intersects (both) $\gamma_i$, or bounds no disk in $S^3\cut T$.  Therefore, the type of annulus $A$ that we seek does not exist in this setting.
    
    \begin{figure}[h]
        \centering
        \raisebox{.25in}{\includegraphics[height=1.1in]{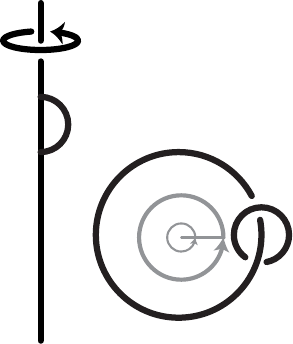}}
        \raisebox{.75in}{$~=~$}
        \includegraphics[height=1.6in]{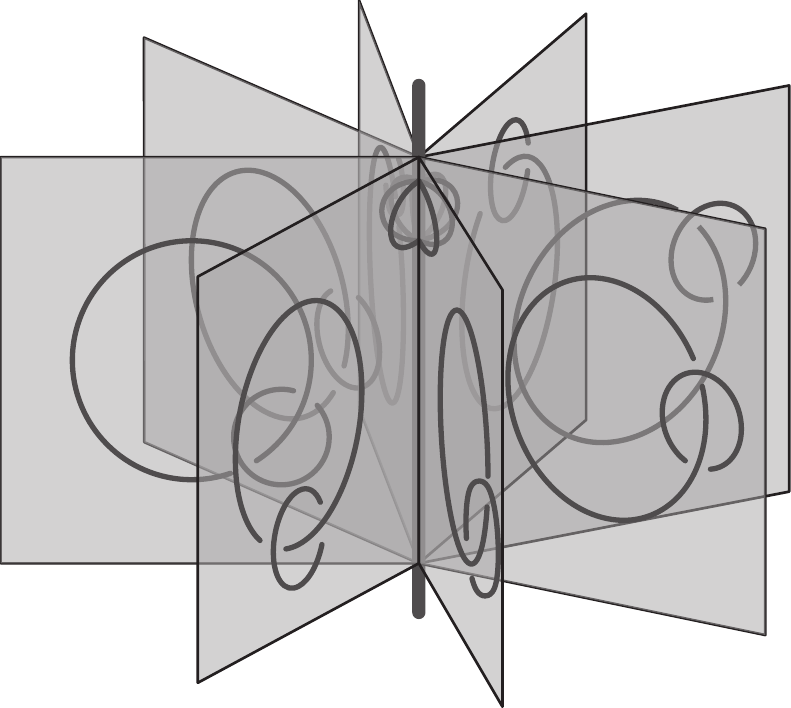}
        \raisebox{.75in}{$~=~$}
        \includegraphics[height=1.6in]{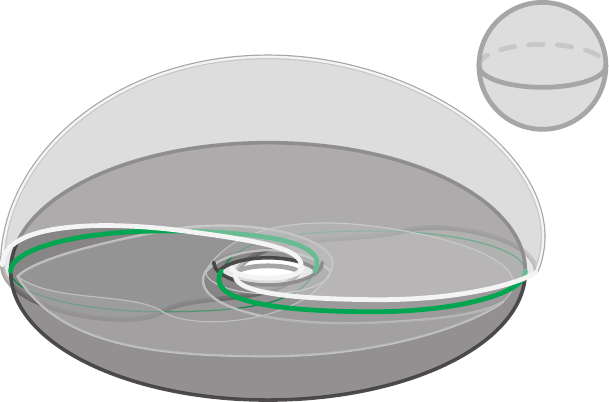}
        \caption{Does this {knotted surface} have a spanning solid that is a Murasugi sum of spun trivial M\"obius bands?}
        \label{fig:torus_1_1}
    \end{figure}

\end{example}

We have not proven, of course, that the knotted surface in Example \ref{ex:no_spun_mobius} has no broken surface diagram with a connected state, let alone that it has no spanning solid which is a (unspecified or particular type of) Murasugi sum of spun trivial M\"obius bands.  We ask the following questions.

\begin{restatable}{question}{connectedstateone}\label{q:mobius_1}
    Let $K$ be the knotted surface described in Example \ref{ex:no_spun_mobius}.  
    \begin{enumerate}[label=(\alph*)]
        \item Does $K$ have a pseudo-ribbon diagram with a connected state?
        \item Does $K$ have a spanning solid that is a Murasugi sum of two solids, each of which is a 1-stripe Murasugi sum of spun trivial M\"obius bands $Y_1$, each attached along $d$ as in Observation \ref{obs:Y1_1_stripe}?
        \item Does $K$ have a spanning solid that is a Murasugi sum of spun trivial M\"obius bands?
    \end{enumerate}    
\end{restatable}

More generally:

\begin{restatable}{question}{connectedstatetwo}\label{q:mobius_2}
    Which knotted surfaces $K$ with $\chi(K)=2$ and normal Euler number 0 have  
    \begin{enumerate}[label=(\alph*)]
        \item a pseudo-ribbon diagram with a connected state?
        \item a spanning solid that is a Murasugi sum of two solids, each of which is a 1-stripe Murasugi sum of spun trivial M\"obius bands $Y_1$, each attached using $d$ as in Observation \ref{obs:Y1_1_stripe}?
        \item a spanning solid that is a Murasugi sum of spun trivial M\"obius bands?
    \end{enumerate}
\end{restatable}

We highlight the following question in particular.

\begin{restatable}{question}{spunMobius}\label{q:mobius_3}
    Is there any knotted surface $K$ (pseudo-ribbon or not) with $\chi(K)=2$ and normal Euler number 0 for which no spanning solid of $K$ is a Murasugi sum of trivial M\"obius bands?
\end{restatable}

To compare Question \ref{q:mobius_3} with the classical setting, recall that every link $L\subset S^3$ has a spanning solid that is a (particular type of) Murasugi sum of trivial M\"obius bands.

\section{Ribbon-alternating {knotted surface}s}\label{sec:alternating}

In this section, we will define and characterize {\it ribbon-alternating {knotted surface}s}. 
We have already encountered these links, including in the construction of certain arborescent {knotted surface}s (where the links came from spinning alternating links) and in Example \ref{ex:no_spun_mobius}. We will see that these examples are nearly exhaustive of the class of alternating links in the sense that every ribbon-alternating {knotted surface} is either a (non-deformed) spun alternating tangle or a link of tori, and possibly one Klein bottle, constructed in a manner analogous to that shown in Figure \ref{fig:torus_1_1}. See Figure \ref{fig:Hex_Torus} for typical examples.

We note that Satoh \cite{satoh02} gave a definition of alternating diagrams for knotted surfaces that applies to arbitrary broken surface diagrams, including those with triple points. He showed that, in this sense, admitting an alternating diagram is equivalent to being pseudo-ribbon. Our definition is somewhat more restrictive.

\begin{figure}[ht]
    \begin{center}
        \includegraphics[height=1.5in]{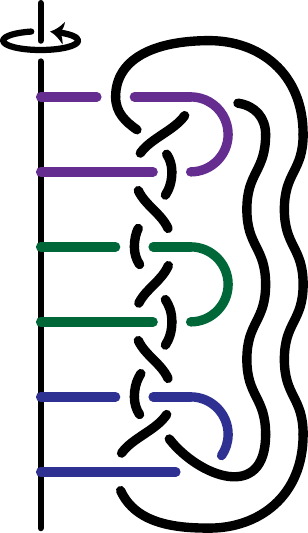}
        \hspace{.25in}
        \includegraphics[height=1.5in]{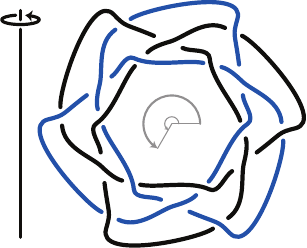}
        \hspace{.25in}
        \includegraphics[height=1.5in]{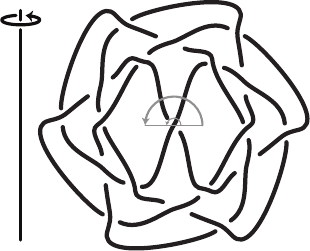}
        \caption{Constructing ribbon-alternating diagrams of knotted surfaces consisting of three linked spheres and a torus (left), two tori (center), and a Klein bottle (right).}
        \label{fig:Hex_Torus}
    \end{center}
\end{figure}

\subsection{Ribbon-alternating diagrams}\label{subsec:alternating_surface_links} Let $E$ be a pseudo-ribbon diagram in which every component has at least one crossing.  In analogy with the classical case, we define $E$ to be ribbon-alternating if, for every path $p:[0,1]\to E$ where $p(0)$ and $p(1)$ are both crossing points, $p(I)$ either runs between an overpass and an underpass or is homotopic rel. endpoints in $E$ into a crossing circle. We present the following equivalent definition because it is simpler and will be easier to work with.

\begin{definition}
    Let $E$ be a pseudo-ribbon diagram in which every component has at least one crossing.  We say that $E$ is \emph{ribbon-alternating} if each component of $E$ cut along its crossing circles either is a disk or is an annulus incident to one overpass and one underpass. 
\end{definition}

\begin{remark}\label{rem:ribbon}
    Each of our solids $Y_n$ in Definition \ref{def:standard_objects} is a checkerboard solid obtained from a ribbon-alternating diagram $E_n$. More generally, given any alternating tangle diagram in a 3-ball, the pseudo-ribbon diagram obtained by spinning it is a ribbon-alternating diagram.
\end{remark}

\begin{example}\label{ex:klein}
Figure \ref{fig:Klein_bottle_spun} shows three alternating diagrams of trivial tangles---the tangle shown left has two endpoints, and those shown center and right have none. Spinning the diagram shown left gives a 1-crossing ribbon-alternating diagram of the trivial 2-knot, spinning the diagram shown center while turning $240^\circ$ gives a 1-crossing ribbon-alternating diagram of the unknotted torus, and spinning the diagram shown right while turning $180^\circ$ gives a 1-crossing ribbon-alternating diagram $E$ of the unknotted Klein bottle with Euler number zero. The crossings arising left and center are easily removable and are examples of what we will call {\it nugatory} crossings in \textsection\ref{sec:nugatory_ribbon}, but the crossing $C$ in the diagram $E$ of the Klein bottle is not removable---$C$ cuts $E$ into a single annulus, and the overpass and underpass at $C$ are both M\"obius bands.  Figure \ref{fig:Klein} shows the resulting broken surface diagram of the Klein bottle.
\end{example}

\begin{figure}[ht]
    \begin{center}
        \includegraphics[height=1in]{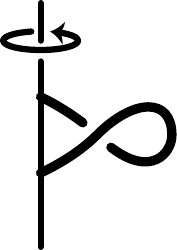}
        \hspace{.75in}
        \includegraphics[height=1in]{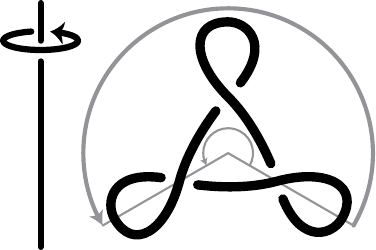}
        \hspace{.75in}
        \includegraphics[height=1in]{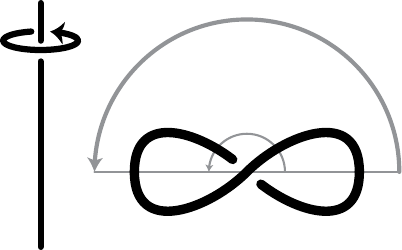}
        \caption{Spinning these tangle diagrams while turning as shown gives ribbon-alternating broken surface diagrams of the trivial 2-knot (left), unknotted torus (center), and Euler number zero unknotted Klein bottle (right).} 
        \label{fig:Klein_bottle_spun}
    \end{center}
\end{figure}

\begin{figure}[ht]
    \begin{center}
        \includegraphics[width=.5\textwidth]{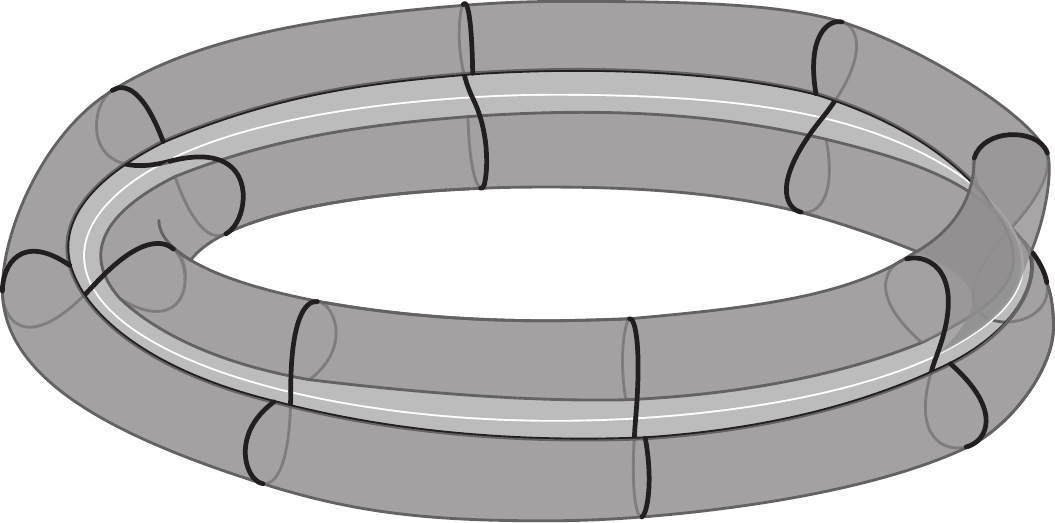}
        \caption{A ribbon-alternating diagram $E$ of the Euler number zero unknotted Klein bottle.}
        \label{fig:Klein}
    \end{center}
\end{figure}

We will see that these trivial examples are typical, in the sense that every ribbon-alternating knotted surface arises essentially in the same way as one of the surfaces in Example \ref{ex:klein}.  In particular:

\begin{proposition}\label{prop:spheres_tori}
    Each component of a ribbon-alternating knotted surface $K$ must be a sphere, torus, or Klein bottle which the crossing circles cut into disks and annuli.
\end{proposition}

\begin{proof}
    Since $K$ is a union of disks and annuli, each component of $K$ must have non-negative Euler characteristic. Moreover, since $K$ admits a pseudo-ribbon diagram, every component of $K$ must have Euler number zero and hence not be a projective plane.
    \end{proof}

\begin{remark}
    Every ribbon-alternating diagram is checkerboard colorable, due to Proposition \ref{prop:cb_shading}.
\end{remark}

\begin{proposition}\label{prop:solid_torus}
        Suppose $E\subset S^3$ is a connected ribbon-alternating diagram of a knotted surface $K\subset S^4$.  Then:
        \begin{enumerate}[label=(\alph*)]
            \item Every component of $S^3\cut E$ is a 3-ball, solid torus, or knot exterior.
            \item Some component of $S^3\cut E$ is a 3-ball if and only if some component of $K$ is a 2-sphere.
            \item If some component of $K$ is a 2-sphere, then the union of the 3-ball components of $S^3\cut E$ is either a solid torus or all of $S^3$.          
        \end{enumerate}
\end{proposition}

In fact, in part (c), the union of the 3-ball components of $S^3\cut E$ can be all of $S^3$ only if $K$ is a trivial link of spheres. See Proposition \ref{prop:split}.

\begin{proof}[Proof of Proposition \ref{prop:solid_torus}]
Consider an arbitrary component of $S^3\cut E$.  Its boundary is a closed, connected surface embedded in $S^3$ formed by gluing together some of the disks and/or annuli of $K\cut (\text{crossing circles})$ along crossing circles. Therefore, the boundary is a sphere or a torus. This implies (a).

For (b), note that a component $Z$ of $S^3\cut E$ 
is a 3-ball if and only if 
its boundary includes a disk component of $K\cut (\text{crossing circles})$. Such a component exists exactly when $K$ includes a 2-sphere component. 

Lastly, note that each 3-ball of $S^3\cut E$ is incident to exactly two disk components of $K\cut (\text{crossing circles})$.  This implies (c). 
\end{proof}

In light of Propositions \ref{prop:spheres_tori} and \ref{prop:solid_torus}, it is sensible to consider ribbon-alternating knotted surfaces including 2-spheres, then those comprised only of tori, and finally those with Klein bottles. Before doing so, we establish several results that apply to any ribbon-alternating broken surface diagram $E$ of a knotted surface $K$.  The first tells us how many regions of $S^3\cut E$ can be incident to each crossing circle. The second is a procedure for changing $E$ into another diagram $E'$ of $K$, such that each component of $S^3\cut E'$ is either a 3-ball or a solid torus, rather than a nontrivial knot complement.

\begin{proposition}\label{prop:3_or_4}
    Suppose $\gamma$ is a crossing in a connected, ribbon-alternating diagram $E\subset S^3$ of a knotted surface $K\subset S^4$. If the overpass and underpass at $\gamma$ are M\"obius bands, then $\gamma$ is incident to two distinct components of $S^3\cut E$. 
    If the overpass and underpass at $\gamma$ are annuli, then $\gamma$ is incident to three or four distinct components of $S^3\cut E$.
\end{proposition}

\begin{proof}
    Proposition \ref{prop:cb_shading} asserts that $S^3\cut E$ is checkerboard colorable, implying that there is are components $T_1$ and $T_2$ of $S^3\cut E$ (shaded light and dark, respectively) that are incident to $\gamma$.   If the overpass and underpass at $\gamma$ are M\"obius bands, then these are the only components of $S^3\cut E$ incident to $\gamma$.
    
    Assume instead that the overpass and underpass at $\gamma$ are annuli. 
    If $T_1$ is the only lightly shaded region of $S^3\cut E$ incident to $\gamma$, then $\partial T_1$ contains two copies of $\gamma$, so there is a properly embedded annulus $A_1\subset T_1$ whose boundary consists of these two copies of $\gamma$. The natural map $T_1\to S^3$ takes $A_1$ to an embedded torus $S_1\subset S^3$ with $S_1\cap E=\gamma$.  Likewise, if $T_2$ is the only darkly shaded region of $S^3\cut E$ incident to $\gamma$, then we obtain an embedded torus $S_2\subset S^3$ with $S_2\cap E=\gamma$. Yet, $\gamma$ is essential on both tori, and it is impossible to embed two tori in $S^3$ such that they intersect along a single essential circle. Therefore, if there is only one lightly shaded region adjacent to $\gamma$, then there must be two distinct darkly shaded regions adjacent to $\gamma$ (and vice versa). Thus, $\gamma$ is incident to at least three components of $S^3\cut E$.
\end{proof}

\begin{procedure}\label{proc:unknot}
Let $E\subset S^3$ be a connected ribbon-alternating diagram of a knotted surface $K\subset S^4$.  If some component $R$ of $S^3\cut E$ is a (nontrivial) knot complement, then $E$ lies entirely in the knotted solid torus $T=S^3\cut R$. We will unknot the solid torus $T$ and obtain a new ribbon-alternating diagram of $K$. To achieve this, view $T$ as the tubular neighborhood of a knot $J$, and consider a crossing change of $J$ along a framed arc between $p,q\in J$ (with the interior of the arc in the complement of $J$). Let $D_p, D_q$ be meridia of $J$ in $T$ centered at $p,q$. Perform ambient isotopy in $S^4$ to move a neighborhood of $D_p$ above the sphere of projection, then isotope it above the 3-sphere as if performing the crossing change, and then push the neighborhood of $D_p$ back to the sphere of projection. We obtain a new projection $E'$ of $K$ contained in the solid torus $T'$ obtained from $T$ by the crossing change. The projection $E'$ agrees with $E$ away from the crossing change and locally within each relevant sheet of $T$, so $S^3\cut T'$ is still a component of $S^3\cut E'$. The other components of $S^3\cut E'$ agree with those of $S^3\cut E$. Repeat iteratively to unknot $T'$ to an unknotted solid torus $T''$, so that $S^3\cut T''$ is a solid torus. Thus, we obtain a ribbon-alternating diagram $E''$ of $K$ with one fewer component of $S^3\cut E''$ being a nontrivial knot complement.

By repeating this procedure for every component of $S^3\cut E$ that is a nontrivial knot complement, we eventually obtain a pseudo-ribbon diagram $\tilde{E}$ for $K$ such that every component of $S^3\cut \tilde{E}$ is a ball or solid torus.
\end{procedure}

The following three propositions will be useful in forthcoming induction arguments.

\begin{proposition}\label{prop:proc_h1}
    Let $\gamma$ be a crossing in a ribbon-alternating diagram $E\subset S^3$, let $R$ be a solid torus component of $S^3\cut E$, let $E'$ be obtained from $E$ via Procedure \ref{proc:unknot}, and let $\gamma'$ and $R'$ respectively be the crossing of $E'$ and component of $S^3\cut E'$ coming from $\gamma$ and $E$.  Then $\gamma$ represents a generator of $H_1(R)$ if and only if $\gamma'$ represents a generator of $H_1(R')$.
\end{proposition}

\begin{proof}
 The step in Procedure \ref{proc:unknot} which changes $R$ to $R'$ takes the solid torus $T=S^3\cut R$ to a solid torus $T'=S^3\cut R'$. Take a curve $\lambda\subset\partial T=\partial R$ that represents a generator of $H_1(T)$, and let $\lambda'$ be the corresponding curve on $\partial T'=\partial R'$.  Letting $i(\cdot,\cdot)$ denote algebraic intersection number, observe that $\gamma$ represents a generator of $H_1(R)$ if and only if $i(\gamma,\lambda)=1$, and $\gamma'$ represents a generator of $H_1(R')$ if and only if $i(\gamma',\lambda')=1$.  Finally, note that the change $E\to E'$ gives a diffeomorphism $T\to T'$ that sends $\gamma\to \gamma'$ and $\lambda\to \lambda'$.  The result follows.
\end{proof}

\begin{proposition}\label{prop:smoothing}
    Suppose $E\subset S^3$ is a connected ribbon-alternating diagram of a knotted surface $K\subset S^4$. Let $\gamma$ be a crossing circle of $E$ with the property that each incident component of $S^3\cut E$ is a solid torus, and consider the two smoothings $E\to E_1$ and $E\to E_2$ of $\gamma$.  Then both $E_i$ are ribbon-alternating. Moreover, some $E_i$ is disconnected (and in fact some component of $S^3\cut E_i$ is a thickened torus) if and only if $\gamma$ is incident to exactly three distinct components of $S^3\cut E$ (in which case the overpass and underpass at $\gamma$ are annuli).
\end{proposition}

\begin{proof}
The first assertion is a simple observation.  For the second, note that smoothing $\gamma$ can change the diffeomorphism type of at most two components of $S^3\cut E$, and that $E_i$ is connected if and only if each component of $S^3\cut E$ has connected boundary.
We now describe what happens in each case.

If the overpass and underpass at $\gamma$ are M\"obius bands, then smoothing $\gamma$ changes one of the incident solid tori of $S^3\cut E$ by attaching a thickened (in $S^3$) M\"obius band $(S^1\ttimes I)\ttimes I$ along the annulus $(S^1\ttimes S_0\ttimes I)$, which yields another solid torus. See Figure \ref{fig:Klein_bottle_spun}, right.
In this case, we have shown that both $E_i$ are connected.

Assume instead from now on that the overpass and underpass at $\gamma$ are annuli. By Proposition \ref{prop:3_or_4}, at least one smoothing of $\gamma$ affects two distinct solid tori, in fact changing them by attaching a thickened annulus $(S^1\times I)\times I$ along a pair of thickened circles $(S^1\times S^0)\times I$, one on the boundary of each solid torus.  Because each attaching circle is nonseparating on its torus, the new component of $S^3\cut E_i$ has connected boundary, so $E_i$ is connected.  

If the other smoothing of $\gamma$ also affects two solid tori, it also yields a connected diagram. Assume instead that this
other smoothing of $\gamma$ affects only one solid torus $T$. Then it does so by attaching a thickened annulus along a parallel pair of essential curves on $\partial T$.  This changes $T$ to a thickened torus and thus yields a disconnected diagram $E_i$.

In summary, the first smoothing of $\gamma$ always yields a connected diagram, and the second one yields a disconnected diagram if and only if $\gamma$ is incident to exactly three solid tori.
\end{proof}

\begin{proposition}\label{prop:glue_gen}
For $i=1,2$, let $T_i\subset S^3$ be solid tori. Let $t_i$ be a generator of $H_1(T_i)$, let $\gamma_i\subset\partial T_i$ be circles with $[\gamma_i]=a_it_i$ for some positive integers $a_i$ with $a_1\geq 2$, and let $\gamma\subset \partial T_1\setminus\gamma_1$ be parallel to $\gamma_1$.  Construct a solid $M$ by gluing a thickened annulus $(S^1\times I)\times I$ to $T_1\sqcup T_2$ along the thickened circles $(S^1\times S^0)\times I=(\gamma_1\sqcup\gamma_2)\times I$, with everything embedded in $S^3$.  
Then $M$ is a knot complement, and $[\gamma]$ does not represent a generator of $H_1(M)$.
\end{proposition}

\begin{proof}
    The thickened annulus and $T_i$ are embedded in $S^3$, and $M$ has torus boundary; therefore, $M$ is a knot complement. Consider an annulus $A=(S^1\times\{\text{point}\})\times I$ which is a co-core of the thickened annulus $(S^1\times I)\times I$.  It is properly embedded in $M$, hence is boundary-parallel in $M$. Yet, it is not $\partial$-parallel through $T_1$ because we assumed that $a_1\geq 2$, i.e., that $\gamma$ is not primitive in $T_1$.   
    It follows that $M$ deformation retracts to $T_1$.  Since $\gamma$ was not primitive in $T_1$, it is not primitive in $M$ either.    
\end{proof}

Turning our focus now to ribbon-alternating knotted surfaces that include 2-spheres, our next result improves upon Proposition \ref{prop:solid_torus} to ensure that, except in the most trivial cases, the 3-ball complementary regions of a ribbon-alternating broken surface diagram always have a solid torus as their union.

\begin{proposition}\label{prop:split}
    Suppose $E\subset S^3$ is a connected ribbon-alternating diagram of a knotted surface $K\subset S^4$. If the union of the 3-ball components of $S^3\cut E$ is all of $S^3$, then $E$ is an $n$-crossing diagram of the trivial link $K$ of $n+1$ 2-spheres. 
\end{proposition}

Thus, if $K$ is a nontrivial link containing a 2-sphere and $E$ is a connected ribbon-alternating diagram of $K$, then the union of the 3-ball components of $S^3\cut E$ is a solid torus.

\begin{proof}[Proof of Proposition \ref{prop:split}]
By assumption, each component of $S^3\cut E$ is a 3-ball.  Note that it follows that each crossing circle of $E$ bounds a disk in each incident region of $S^3\cut E$.

First, we claim that no component of $K$ is a torus or Klein bottle.  Suppose instead that $K$ has a torus or Klein bottle component $T$, and let $p(T)$ be the projection of $T$ in $E$. If $p(T)$ has a self-intersection, then at least one of the smoothings of $E$ along this  circle yields a new ribbon-alternating diagram of a link with a torus or Klein bottle component, such that the new diagram cuts $S^3$ into 3-balls (note here that each crossing circle is incident to at least three distinct components of $S^3\cut E$). We may thus assume that $p(T)$ has no self-intersections. It follows that $p(T)$ is an embedded torus in $S^3$, and therefore that there is a compressing disk $D$ on one side of $p(T)$.  

Let $R$ be the component of $S^3\cut p(T)$ that contains $D$.  Since $E$ cuts $R$ into 3-balls, the crossing circles of $E$ on $p(T)$ must be parallel to $\partial D$. Yet, these crossing circles do not bound disks on the other side of $p(T)$. This contradicts our earlier observation that each crossing circle of $E$ bounds a disk in each incident region of $S^3\cut E$. We thus conclude that $K$ has no torus or Klein bottle components.

Now let $S$ be a 2-sphere component of $K$, and suppose that $p(S)\subset E$ has some number of self-intersections. Then $S^3\cut p(S)$ has a component $X$ whose boundary has positive genus. Each essential curve in $\partial X$ that is nullhomologous in $X$ intersects a self-intersection circle of $p(S)$. Therefore, $X$ cannot be cut into $3$-balls in $S^3\cut E$ without crossings between self-intersection circles, i.e., triple points, which are not present in $E$. We conclude that $p(S)$ is an embedded 2-sphere.

Suppose there are two 2-sphere components $S,S'$ of $K$ such that $p(S)$ and $p(S')$ intersect in more than one circle. Then again, some component of $S^3\cut (p(S)\cup p(S)')$ has boundary of positive genus and cannot be cut into 3-balls by cutting along surfaces that do not cross the intersection circles of $p(S),p(S')$, yielding a similar contradiction. Thus, any two components of $K$ have projections in $E$ that intersect in at most one circle.

Let $G$ be a graph with vertices corresponding to the components of $K$, with an edge two vertices $v_i,v_j$ if there is an intersection between the corresponding projections in $E$. 
Suppose $G$ has a cycle $(v_1,v_2,\ldots, v_n, v_1)$. Let $S_1,\ldots, S_n$ be the components in $K$ corresponding to $v_1,\ldots, v_n$. Then finally, a component of $S^3\cut (p(S_1)\cup \cdots p(S_n))$ has boundary of positive genus and cannot be cut into 3-balls without introducing triple points. 
We conclude $G$ is a tree.

Take $K$ to lie close to the 3-sphere of projection. 
Let $v_0$ be a leaf in $G$. The corresponding component $S_0$ of $K$ has exactly one intersection circle in its projection. Let $v_1$ be the unique vertex adjacent to $v_0$, and let $S_1$ be the component of $K$ corresponding to $v_1$. If $S_0$ lies below (with respect to the projection direction) $S_1$, then isotope $S_0$ downward in $S^4$ to lie below all of the other components. If $S_0$ lies above $S_1$, then isotope $S_0$ upward to lie above the other components of $K$. In either case, we conclude that $S_0$ a split, unknotted component. Now induct on $G\setminus\{v_0\}$ to conclude that every component of $K$ is an unknotted 2-sphere split from the other components, and hence $K$ is an unlink of 2-spheres.
\end{proof}

We are now ready to characterize ribbon-alternating knotted surfaces that include 2-spheres.

\begin{proposition}\label{prop:knot_a}
Suppose $K\subset S^4$ is a nontrivial knotted surface that has a connected ribbon-alternating diagram, and let $E$ be a connected ribbon-alternating diagram of $K$ such that each component of $S^3\cut E$ is a 3-ball or solid torus (one may obtain $E$ from any ribbon-alternating diagram of $K$ via Procedure \ref{proc:unknot}). 
If some component of $K$ is a 2-sphere, then 

\begin{enumerate}[label=(\alph*)]
    \item\label{prop:spheres_a} the union of the 3-ball components of $S^3\cut E$ is an unknotted solid torus $T$, say with core $\tau$,
    \item\label{prop:spheres_b} the crossing circles on $\partial T$ are meridia on $\partial T$, and
    \item\label{prop:spheres_c} taking $\gamma$ to be the core of any remaining component $R$ of $S^3\cut E$, $\gamma$ is parallel to the crossing circles on $\partial R$.
\end{enumerate}

Therefore, $E$ is the result of spinning an alternating classical tangle (without deformation).
\end{proposition}

\begin{proof}
    Since $K$ is nontrivial with a ribbon-alternating diagram, Proposition \ref{prop:split} and part (c) of Proposition \ref{prop:solid_torus} imply that the union of the 3-ball components of $S^3\cut E$ is a solid torus $T$, say with core $\tau$. This proves all of \ref{prop:spheres_a} except the claim that $T$ is unknotted, and \ref{prop:spheres_b} follows immediately.

    For \ref{prop:spheres_c}, consider a crossing circle $\gamma$ on $\partial T$, and let $T'$ be an incident solid torus component of $S^3\cut E$.  The fact that $\gamma$ represents a primitive element of $H_1(S^3\cut T)$ implies that it also represents a primitive element of $H_1(T')$; it follows that $\gamma$ is parallel to the core of $T'$.  
    
    To see that this is true of the crossing circles on the boundary of all solid torus components of $S^3\cut E$, not just those incident to $T$, we argue by induction on the number of crossings of $E$ that are not incident to $T$---call this number $n$. We have already proven the base case.  For the induction step, consider a crossing circle $\gamma$ not on $\partial T$, and smooth $\gamma$ in a way that yields a connected ribbon-alternating diagram $E'$---this is possible by Proposition \ref{prop:smoothing}, and it involves merging two of the distinct solid tori $S_1$ and $S_2$ incident to $\gamma$ into one solid torus $S$. By induction, all crossing circles in $E'$ are parallel to the cores of all incident solid tori. In particular, the crossing circles on $\partial S$ are parallel to the core of $S$.  The contrapositive of Proposition \ref{prop:glue_gen} implies that $\gamma$ is also parallel to the core circles of its incident solid tori. 
    
    Next, we complete the proof of \ref{prop:spheres_a} by showing that $T$ must be unknotted. If $n=0$, choose an incompressible Seifert surface $F$ for $\tau$ and isotope it so that $F\cap\partial T$ is a longitude and $|(F\cap E)\cut T|$ is minimal.  Then, for each component $T'$ of $S^3\cut E$ other than those in $T$, each component of $F\cap T'$ is incompressible, orientable, and non-$\partial$-parallel in $T'$, and therefore is a disk.  It follows (since $n=0$) that $F$ is a disk and thus that $\tau$ is unknotted. For the induction step, consider any crossing circle $\gamma$ not on $\partial T$, and smooth $\gamma$ in a way that yields a simpler connected ribbon-alternating diagram $E'$---again, this is possible by Proposition \ref{prop:smoothing}.  It follows by induction that $T$ is unknotted.
    
    It remains only to prove the last assertion in the proposition. Fiber $S^3\setminus \tau$ with disks such that each crossing circle of $E$ is transverse to each disk.  Choose one of these disks $D^2$, and consider the classical tangle diagram $D=D^2\cap E$.  We can now see that the broken surface diagram $E$ is obtained by spinning $D$ while performing a planar isotopy of $D$ that returns $D$ to itself setwise (preserving orientations). Since $K$ contains 2-spheres, $D$ has endpoints. The isotopy of $D$ must fix these endpoints, so the spinning of $D$ proceeds without deformation.
    \end{proof}

Next, we characterize ribbon-alternating links comprised only of tori.

\begin{proposition}\label{prop:knot_b}
Suppose $K\subset S^4$ is a knotted surface comprised only of tori, suppose that $K$ has a nontrivial, connected, ribbon-alternating diagram, and let $E$ be an ribbon-alternating diagram of $K$ such that each component of $S^3\cut E$ is a solid torus (one may obtain $E$ from the first diagram via Procedure \ref{proc:unknot}). 
Then, with either one or two exceptions,  each component $T$ of $S^3\cut E$ has a core circle $\tau$ parallel to the crossing circles on $\partial T$; each exceptional core circle is unknotted.  Therefore, there is an alternating link diagram $D$ such that $E$ is obtained by spinning $D$ while performing a planar isotopy of $D$ that returns $D$ to $D$ setwise (preserving orientations). 
\end{proposition}

\begin{remark}
It follows that if $L$ is the link described by $D$, then $K$ is obtained from $L$ by spinning, turned spinning (as in \cite{boyle_torus}), or symmetry-spinning via a periodic symmetry of $L$. Note that since the orientations of $D$ are fixed by the rotation, the rotation can only fix points on $D$ if the rotation is a 360$^\circ$ rotation, yielding a turned spin. Otherwise, the rotation has no fixed points, and hence is a periodic symmetry of $D$. For more discussion of knotted surfaces obtained by symmetry-spinning, defined more generally, see \cite{kanenobu1987,teragaito1989,teragaito1990}. 
\end{remark}

\begin{proof}[Proof of Proposition \ref{prop:knot_b}]
    We will prove the last assertion at the end.  For the rest, we argue by induction on the number of crossing circles in $E$.  
    
    Suppose $E$ has a single crossing circle $\gamma_0$. (The overpass and underpass at $\gamma_0$ are annuli.) Then $\gamma_0$ is incident to each component of $S^3\cut E$, and there are exactly three components---otherwise, by Propositions \ref{prop:3_or_4} and \ref{prop:smoothing}, either smoothing of $\gamma_0$ would yield a connected surface that cuts $S^3$ into three components, which is impossible.  Therefore, by Proposition \ref{prop:smoothing}, one of the smoothings of $\gamma_0$ yields a disconnected diagram $E'$ that cuts $S^3$ into two solid tori (unchanged by the smoothing) and a thickened torus. Viewing $E'$ as a state of $E$, we may record the smoothed crossing $\gamma_0$ by inserting a crossing annulus $A$ in the thickened torus $R$ of $S^3\cut E'$; $A$ cuts $R$ into a single solid torus---corresponding to a component of $S^3\cut E$---whose core circle is parallel to the core circle of $A$ (corresponding to $\gamma_0$). Finally, note that the core circles of the solid tori of $S^3\cut E'$ are unknotted circles which are also cores of solid tori of $S^3\cut E$, and at most one of these circles is parallel to the crossing of $E$. 

    For the induction step, 
    choose a crossing circle $\gamma$ of $E$, and write $S_1,S_2,S_3,S_4$ for the incident solid tori of $S^3\cut E$, such that (taking indices modulo 4) each $S_i$ abuts $S_{i+1}$ along $E\cut(\text{crossing circles})$ and $S_1\neq S_3$. (Here, we use Proposition \ref{prop:3_or_4}.)  Smooth $\gamma$ in the way that merges $S_1$ and $S_3$ into a new region $R$. The resulting diagram $E'$ is connected and ribbon-alternating. 
    
    If $S_1$ or $S_3$ was exceptional in $E$, then by Proposition \ref{prop:glue_gen} the crossing circles on $\partial R$ do not represent generators of $H_1(R)$, i.e., $R$ is exceptional (in $E'$); therefore, by induction, the core circle of $R$ is unknotted---as therefore is the core circle of whichever of $S_1$ or $S_3$ is exceptional in $E$---and either one of the components of $S^3\cut E'$ other than $R$ is exceptional (and unknotted), or none are.  These components match the components of $S^3\cut E$, so the result follows. 

    If neither $S_1$ nor $S_3$ was exceptional in $E$, then neither is $R$ in $E'$, and the result similarly follows by induction.

    It remains only to prove the last assertion in the proposition. Choose a core circle $\tau$ of a solid torus $T$ of $S^3\cut E$ such that $\tau$ is not parallel to the crossing circles on $\partial T$. We know that $\tau$ is unknotted.  Fiber $S^3\setminus \tau$ with disks such that each crossing circle of $E$ is transverse to each disk.  Choose one of these disks $D^2$, and consider the classical link diagram $D=D^2\cap E$.  We can now see that the broken surface diagram $E$ is obtained by spinning $D$ while performing a planar isotopy of $D$ that returns $D$ to $D$ setwise (preserving orientations).     
\end{proof}

Finally, we characterize nonorientable ribbon-alternating knotted surfaces.

\begin{proposition}\label{prop:knot_c}
    Let $E$ be a connected ribbon-alternating diagram of a nonorientable knotted surface $K\subset S^4$.  If some component of $K$ is a Klein bottle $K_0$, then the rest of the components of $K$ are tori, and there is a unique self-crossing $\gamma_0$ of $K_0$ such that both smoothings of $E$ at $\gamma_0$ yield diagrams of orientable knotted surfaces.  Further, with one exception,  each component $T$ of $S^3\cut E$ has a core circle $\tau$ parallel to every crossing circle on $\partial T$ besides $\gamma_0$. 
    Hence, there is a classical alternating link diagram $D$ such that $E$ is obtained  by spinning $D$ while rotating by some odd multiple of $180^\circ$ about a crossing of $D$, where the rotation fixes $D$ setwise.
\end{proposition}

\begin{proof}
        We will prove the last assertion at the end.  
        For each crossing $\gamma_i$ of $E$ at which the overpass and underpass are M\"obius bands, smooth $\gamma_i$ arbitrarily. Denote the  resulting diagram by $E'$. Insert a M\"obius band $A_i$ to record where each smoothed crossing $\gamma_i$ was. Double each $A_i$ to get a properly embedded annulus $A_i'\subset S^3\cut E'$.  Then add a new crossing to $E'$ along each annulus $A_i'$ to get a connected ribbon-alternating diagram $E''$ of an orientable knotted surface $K'\subset S^4$.  Observe that $E''$ cuts $S^3$ into solid tori---one coming from each M\"obius band $A_i$, one coming from each component of $S^3\cut E$.  
        Proposition \ref{prop:knot_b} implies that $E''$ is the result of spinning an alternating 2-crossing diagram $D$ of the unknot while rotating $D$, while returning $D$ to itself setwise, preserving orientations.  This rotation must also return each $A_i$ to itself, setwise, and must therefore be through an odd multiple of $180^\circ$, with each $A_i$ containing the axis of rotation.  Therefore, there is just one $A_i$, and $E$ is the result of spinning an alternating diagram $D'$ while rotating $D'$ by some odd multiple of $180^\circ$ around a crossing. See Figure \ref{fig:Whitehead_Mobius_annulus}.
\end{proof}

\begin{figure}[h]
    \begin{center}
        \includegraphics[height=.9in]{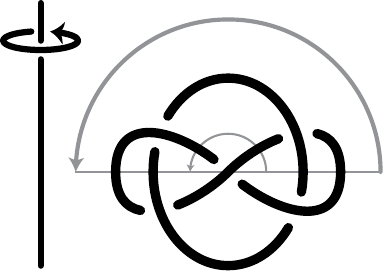}\raisebox{.4in}{$\to$}
        \includegraphics[height=.9in]{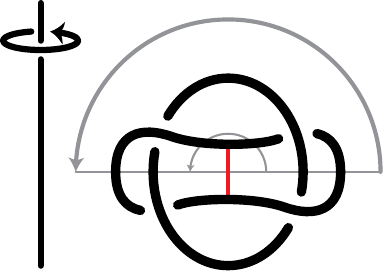}\raisebox{.4in}{$\to$}
        \includegraphics[height=.9in]{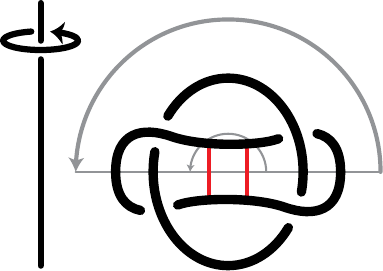}\raisebox{.4in}{$\to$}
        \includegraphics[height=.9in]{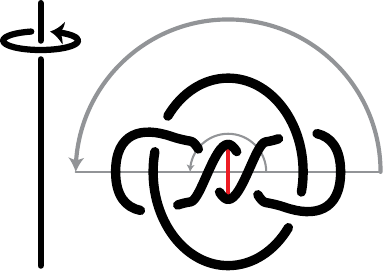}
    \end{center}
    \caption{Replacing a crossing where the overpass and underpass are M\"obius bands with two crossings where the overpass and underpass are annuli}
    \label{fig:Whitehead_Mobius_annulus}
\end{figure}

For convenience, we collect the main preceding results in the following theorem.

\begin{theorem}\label{thm:alternating}
    Let $K\subset S^4$ be a knotted surface admitting a connected, ribbon-alternating diagram $E$. 
    \begin{enumerate}[label=(\alph*)]
        \item\label{caseA} If some component of $K$ is a 2-sphere, then $E$ is the result of spinning an alternating classical tangle diagram (without deformation), and in particular $K$ is a union of spheres and tori. 
        \item\label{caseB} If $K$ is orientable and contains no 2-spheres, then it is a union of tori, and 
        there is an alternating link diagram $D\subset D^2$ such that $E$ is obtained by spinning $D$ while turning $D$ about some point $x\in D^2\setminus D$ of rotational symmetry 
        (thus preserving orientations). 
        \item\label{caseC} If $K$ is nonorientable, then $K$ has exactly one Klein bottle component, the other components of $K$ are tori, and there is an alternating link diagram $D$ such that $E$ is obtained by spinning $D$ while rotating $\pm 180^\circ$ about a crossing of $D$, where the rotation fixes $D$ setwise.
    \end{enumerate}
\end{theorem}

\begin{proof}
The theorem simply collects part of what we have proven in Propositions \ref{prop:knot_a}, \ref{prop:knot_b}, and \ref{prop:knot_c}. For the last assertion, note that adding multiples of $720^\circ$ to the rotation of $D$ does not change the underlying knotted surface. 
\end{proof}

\begin{theorem}\label{thm:non_alternating}
If $K$ is the spin of a 
classical nonalternating knot $L$, then $K$ is not ribbon-alternating. 
\end{theorem}

\begin{proof}
    Suppose otherwise. Then by Theorem \ref{thm:alternating}, there is a classical alternating knot $L'$ such that $K$ is also the spin of $L'$.  Note that $\pi_1(S^3\setminus K)$, $\pi_1(S^3\setminus L)$, and $\pi_1(S^3\setminus L')$ are all isomorphic.  Therefore, the prime factors of $L$ and $L'$ are equivalent up to mirroring---this follows from \cite[Theorem 1]{feustel-witten} and \cite[Theorem 1]{gordon-luecke}. Yet, the prime factors of $L'$ are all alternating, since $L'$ is alternating, and so the prime factors of $L$ are all alternating. This implies contrary to assumption that $L$, and thus $K$, is ribbon-alternating. 
\end{proof}

\begin{restatable}{question}{spunNAlink}\label{q:non_alternating}
If $K$ is a spin of a prime classical nonalternating link $L$, must $K$ be non ribbon-alternating? What if $L$ is not prime? 
\end{restatable}

Example \ref{ex:35_nugatory} will show that turned spins do not always respect the non (ribbon) alternating condition.

\subsection{Nugatory crossings in pseudo-ribbon diagrams}\label{sec:nugatory_ribbon}

Inspired by Tait's classical conjectures, we will ask several questions about ribbon-alternating {knotted surface}s. Before we can do so, we must adapt the notion of {\it nugatory} crossings to dimension four.  We begin with the definitions of nugatory crossings for classical link diagrams and tangle diagrams.

\begin{definition}\label{def:nugatory_classical}
A crossing $c$ in a classical link diagram $E\subset S^2$ is \emph{nugatory} if there is a circle $\gamma\subset S^2$ with $\gamma\cap E=\{c\}$ (transversally). A crossing $c$ in a classical tangle diagram $E\subset D^2$ is \emph{nugatory} if there is a circle or properly embedded arc $\gamma\subset S^2$ with $\gamma\cap E=\{c\}$. In both cases, we say the circle or arc $\gamma$ {\it detects} the nugatory crossing $c$.
\end{definition}

A classical nugatory crossing can be removed by isotopy, as shown in Figures \ref{fig:nugatory_detour} and \ref{fig:nugatory_flype} and discussed further below.  For now, we note just that removing a nugatory crossing (in a tangle) detected by an arc generally requires sliding the endpoints of the arc along the boundary of the 3-ball that contains the tangle.

\begin{definition}\label{def:nugatory}
Let $E\subset S^3$ be a pseudo-ribbon diagram of a knotted surface $K\subset S^4$. Call a crossing circle $\gamma$ in $E$ \emph{nugatory} if there is an embedded sphere or torus $F\subset S^3$ that satisfies $F\cap E=\gamma$, and we say that this sphere or torus \emph{detects} the nugatory crossing.
\end{definition}

\begin{remark}
    The overpass and underpass at any nugatory crossing are annuli, not M\"obius bands.
\end{remark}

If a sphere or torus $F$ detects a nugatory crossing $\gamma$ in a pseudo-ribbon diagram $E\subset S^3$ of a knotted surface $K\subset S^4$, and if all crossing circles of $E$ besides $\gamma$ lie on the same side of $S^3\cut F$, then $K$ has an isotopy which removes $\gamma$. To see this, imagine smoothing $\gamma$ so that the resulting diagram $E'$ is disjoint from $F$. Then, one of the components of $S^3\cut E'$ is a crossingless sphere or torus.  A 3-ball or solid torus it bounds in $S^4\cut E$ (together with half of a standard solid crossing band in $\nu \gamma$) guide an isotopy which removes $\gamma$. When $F$ is a 2-sphere, this isotopy pushes one disk of $K\cut \gamma$ past another such disk, and when $F$ is a solid torus, the lift of $\gamma$ in $K$ cuts an annulus off of $K$, and this isotopy removes this annulus by pushing it past $\gamma$.

Not all nugatory crossings, however, can be removed in this way. Indeed, both components of $S^3\setminus F$ could contain crossings of $E$.  Classical link diagrams encounter an analogous inconvenience in that some (classical) nugatory crossings cannot be removed merely by a Reidemeister-1 move---see the left image in Figure \ref{fig:nugatory_detour}. Still, there are at least two good ways to rectify this classical inconvenience (in which the original diagram $E$ is always a diagrammatic connect sum). 

\begin{figure}[ht]
    \begin{center}
        \labellist\tiny
        \pinlabel{$1$} at 457 105
        \pinlabel{$2$} at 460 25
        \pinlabel{$3$} at 340 105
        \pinlabel{$4$} at 337 25
        \pinlabel{$2$} at 638 100
        \pinlabel{$1$} at 630 17
        \pinlabel{$3$} at 515 105
        \pinlabel{$4$} at 515 25
        \endlabellist
    \includegraphics[height=1.2in]{figures/Nugatory_Crossing_1A}\quad\raisebox{.6in}{$\to$}\quad\includegraphics[height=1.2in]{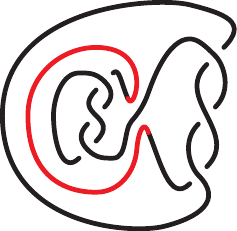}
        \hfill
        \includegraphics[height=1.2in]{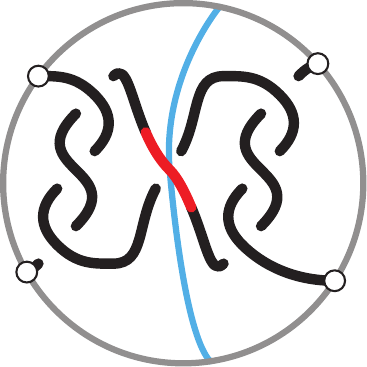}\quad\raisebox{.6in}{$\to$}\quad\includegraphics[height=1.2in]{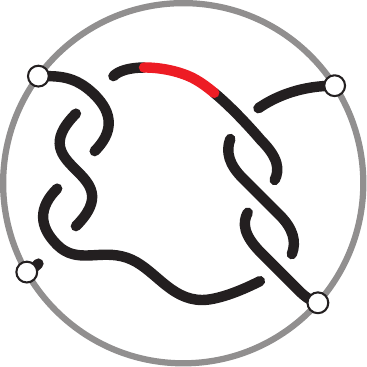}
        \caption{Left: removing a (classical) nugatory crossing detected by a circle via a detour move. Right: removing a nugatory crossing detected by an arc may require moving the endpoints of the tangle.}
        \label{fig:nugatory_detour}
    \end{center}
\end{figure}

One option is to remove the nugatory crossing via a detour move, as shown in Figure \ref{fig:nugatory_detour}.  If $E$ is ribbon-alternating, the resulting diagram will not be ribbon-alternating, but one can fix this by reflecting one of the diagrammatic connect summands (this corresponds to rotating the associated connect summand of the knot).  

\begin{figure}[ht]
    \begin{center}
        \includegraphics[height=1.2in]{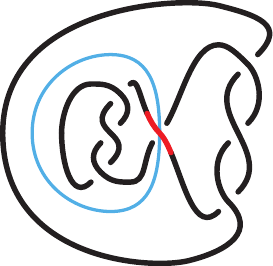}~\raisebox{.6in}{$\longrightarrow$}~\includegraphics[height=1.2in]{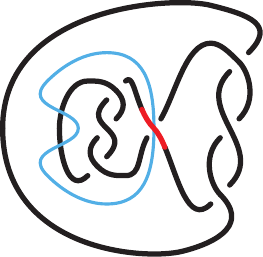}~\raisebox{.6in}{$\underset{\text{flype}}{\longrightarrow}$}~\includegraphics[height=1.2in]{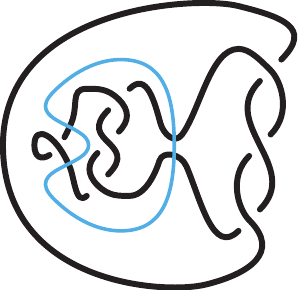}~\raisebox{.6in}{$\underset{\text{R1}}{\longrightarrow}$}~\includegraphics[height=1.2in]{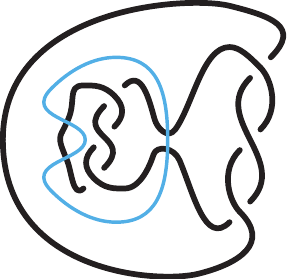}
        \caption{Removing a nugatory crossing via a flype and an R1 move.}
        \label{fig:nugatory_flype}
    \end{center}
\end{figure}

A second option is to observe that any classical nugatory crossing can be removed by performing a {\it flype} and then an R1 move. See Figure \ref{fig:nugatory_flype}.  
We now adapt flype moves to dimension four; we will find that they help us remove any nugatory crossing detected by a torus, while preserving the ribbon-alternating condition (if applicable).  A similar type of move will allow us to remove any nugatory crossing detected by a sphere, again while preserving the ribbon-alternating condition. 

\begin{definition}\label{def:flype_T2}
    Suppose that $E\subset S^3$ is a broken surface diagram of a knotted surface $K\subset S^4$, and $T\subset S^3$ is a torus which intersects $E$ in a crossing circle $\gamma$ and two other circles $\gamma_1,\gamma_2$, such that $\gamma\cup \gamma_1\cup \gamma_2$ cuts $T$ into three annuli. Let $A$ be the annulus bounded by $\gamma_1\cup\gamma_2$. 
    Let $\nu T=T\times [-1,1]$ be a regular neighborhood of $T$ in $S^3$. As illustrated in Figure \ref{fig:nugatory_torus_parts}, write $A_+$ and $A_-$ for the overpass and underpass of $E$ at $\gamma$, and write $\gamma^\pm_\pm=\partial A_\pm\cap (T\times\{\pm1\})$ and $\gamma_i^\pm=\gamma_i\cap (T\times\{\pm1\}) $, also as shown.
    Note that $\gamma^+_+\cup \gamma^+_-\cup \gamma_1^+\cup \gamma_2^+$ cuts $T\times\{1\}$ into four annuli, one co-bounded by $\gamma^+_+\cup\gamma^+_-$, another $A\times\{1\}$ co-bounded by $\gamma_1^+\cup \gamma_2^+$, and two more. If necessary, switch the subscripts on $\gamma_1$ and $\gamma_2$ so that the boundaries of these last two annuli are $\gamma_+^+\cup \gamma_1^+$ and $\gamma_-^+\cup \gamma_2^+$, as in Figure \ref{fig:nugatory_torus_parts}, top.
    
    Change $E$ to another diagram $E'$ of $K$ as follows. Replace the two annuli $\gamma_i\times[-1,1]$ with two annuli $A'_-,A'_+\subset A\times [-1,1]$ with $\partial A'_-=\gamma_1^+\cup \gamma_{2}^-$ and $\partial A'_+=\gamma_1^-\cup \gamma_{2}^+$, taking indices modulo 2, with $A'_+$ passing over $A'_-$.  Also replace the annuli $A_+$ and $A_-$ with two product 
    annuli in $T\times[-1,1]$, one bounded by $\gamma_+^+\cup\gamma_-^-$ and the other by $\gamma_-^+\cup\gamma_+^-$. Finally, choose a component $Y$ of $(S^3\cut T)\times[-1,1]$, a properly embedded annulus $A_R$ in $Y$, and an involution $f:Y\to Y$ that fixes $E\cap\partial Y$ setwise and fixes $A_R$ pointwise (that is, $f$ is a reflection of $Y$ across $A_R$ that fixes $E\cap\partial Y$ setwise), and replace $E\cap Y$ with $f(E\cap Y)$ (preserving crossing information). Keep $E\setminus (Y\cup T\times[-1,1])$ unchanged.  We say that the resulting diagram $E'$ is obtained from $E$ via a {\emph{flype}}.
\end{definition}

\begin{figure}[ht]
    \begin{center}
        \labellist\small
        \pinlabel{$A_-$} at 410 455
        \pinlabel{$A_+$} at 410 510
        \pinlabel{$A_+$} at 610 507
        \pinlabel{$A_-$} at 614 452
        \pinlabel{$\gamma_+^-$} at 685 460
        \pinlabel{$\gamma_+^+$} at 460 500
        \pinlabel{$\gamma_-^+$} at 460 460        
        \pinlabel{$\gamma_1^-$} at 110 500
        \pinlabel{$\gamma_1^+$} at 1040 500
        \pinlabel{$\gamma_{2}^-$} at 110 460
        \pinlabel{$\gamma_{2}^+$} at 1040 460
        \pinlabel{{\color{ForestGreen} $A$} } at 1000 480
        \pinlabel{{\color{ForestGreen} $A$} } at 30 480
        \pinlabel{$A_-$} at 410 455
        \pinlabel{$A_+$} at 410 510
        \pinlabel{$A_+$} at 610 507
        \pinlabel{$A_-$} at 614 452
        \pinlabel{$\gamma_-^-$} at 685 500
        \pinlabel{$\gamma_+^-$} at 685 460
        \pinlabel{$\gamma_+^+$} at 460 500
        \pinlabel{$\gamma_-^+$} at 460 460
        \pinlabel{$\gamma_+^-$} at 685 100
        \pinlabel{$\gamma_+^+$} at 460 140
        \pinlabel{$\gamma_-^+$} at 460 100 
        \pinlabel{$\gamma_-^-$} at 685 140
        \pinlabel{$\gamma_1^-$} at 110 140
        \pinlabel{$\gamma_1^+$} at 1040 140
        \pinlabel{$\gamma_{2}^-$} at 110 100
        \pinlabel{$\gamma_{2}^+$} at 1040 100
        \pinlabel{ $A'_-$ } at 990 150
        \pinlabel{ $A'_-$ } at 40 150
        \pinlabel{ $A'_+$ } at 990 90
        \pinlabel{ $A'_+$ } at 40 90
        \pinlabel{$A_-$} at 410 455
        \pinlabel{$A_+$} at 410 510
        \pinlabel{$A_+$} at 610 507
        \pinlabel{$A_-$} at 614 452
        \pinlabel{$\gamma_+^-$} at 685 460
        \pinlabel{$\gamma_+^+$} at 460 500
        \pinlabel{$\gamma_-^+$} at 460 460
        \Huge
        \pinlabel{{\color{Fuchsia}     \scalebox{-1}[-1]{$T_1$}}} at 220 124
        \pinlabel{{\color{Fuchsia}\scalebox{1}[-1]{$T_1$}}} at 800 124
        \pinlabel{{\color{Fuchsia}     \scalebox{-1}[1]{$T_1$}}} at 220 472
        \pinlabel{{\color{Fuchsia}\scalebox{1}[1]{$T_1$}}} at 800 472
        \endlabellist
        \includegraphics[width=130mm]{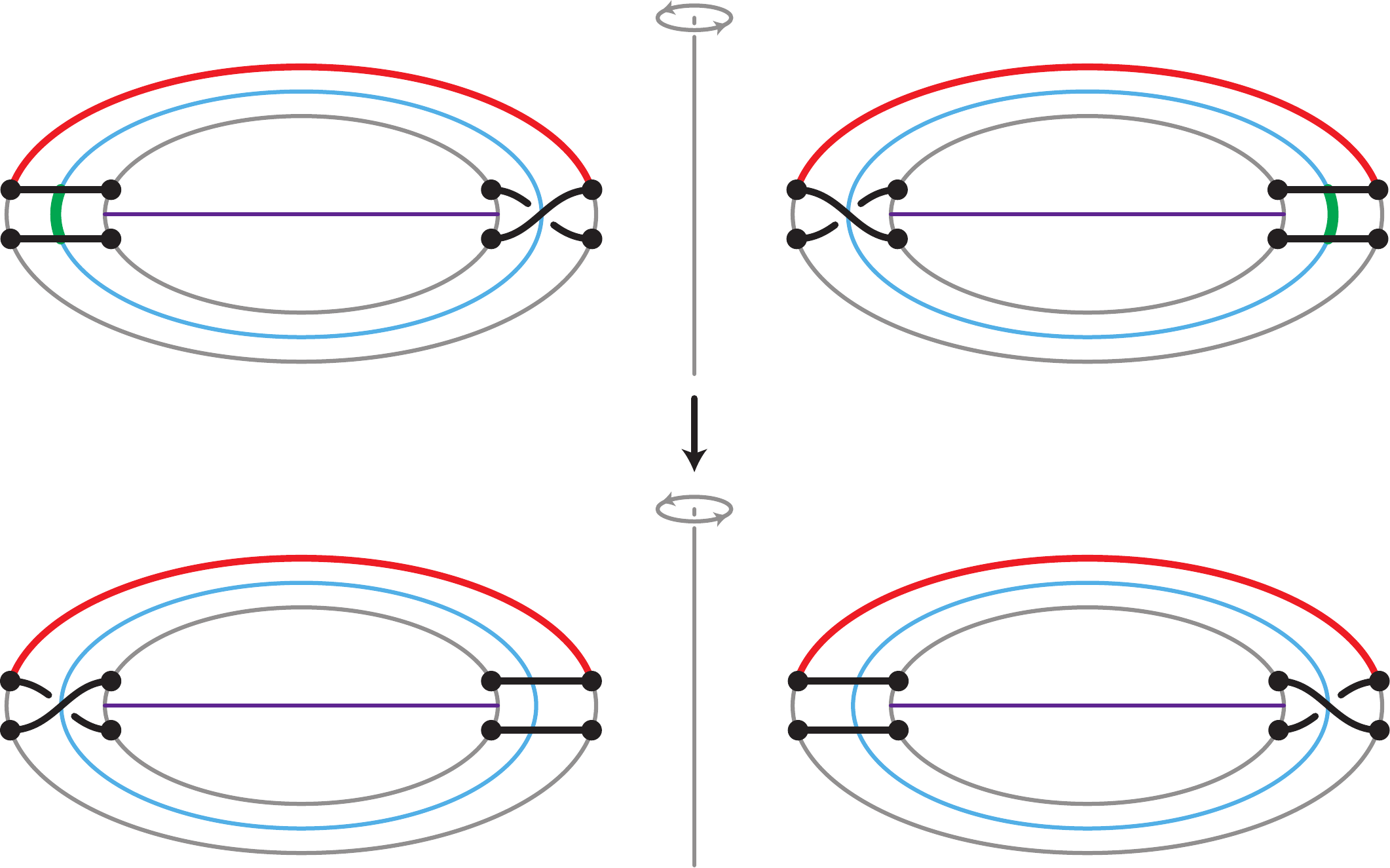}
        \caption{The process of flyping a broken surface diagram along a torus, illustrated as the spin of an unknotted circle. The torus $T$ is drawn in light blue, and the annulus $A$ is drawn in green. }
        \label{fig:nugatory_torus_parts}
    \end{center}
\end{figure}

\begin{figure}[ht]
    \begin{center}
        \includegraphics[height=1.2in]{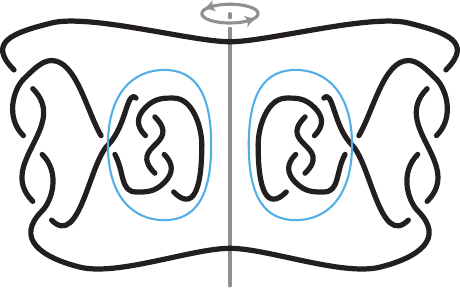}\raisebox{.6in}{~$\to$~}
        \includegraphics[height=1.2in]{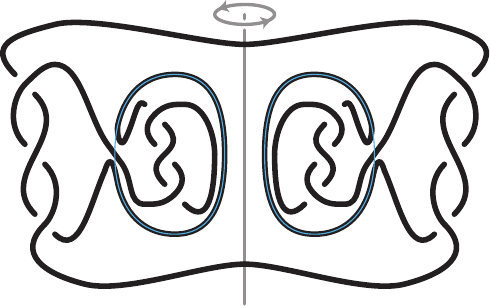}
        \caption{Removing a nugatory crossing detected by a torus via a detour move}
        \label{fig:nugatory_torus_detour}
    \end{center}
\end{figure}

\begin{remark}
If $E\subset S^3$ is an ribbon-alternating diagram of a knotted surface $K\subset S^4$ and $T\subset S^3$ is a torus that detects a nugatory crossing $\gamma$ in $E$, then there are two convenient ways to remove $\gamma$ from $E$.  One option is to perform a detour move, perturbing $T$ so that it contains the entire overpass at $\gamma$ and then replacing this overpass with the rest of $T$. See Figure \ref{fig:nugatory_torus_detour}.  This change in the diagram corresponds to an isotopy of $K$, because $T$ bounds a solid torus through which these annuli are parallel, and it takes $E$ to a pseudo-ribbon diagram $E'$ of $K$ with one fewer crossing than $E$.  Unfortunately, $E'$ is not necessarily ribbon-alternating.  The easiest way to fix this is to perform a flype on $E$ before removing the nugatory crossing by a small detour move.  See Figure \ref{fig:nugatory_torus_flype}.   
\end{remark}     

\begin{figure}
    \begin{center}
        \includegraphics[height=1.1in]{figures/Nugatory_Torus_1A.pdf}\raisebox{.6in}{~$\to$~}
        \includegraphics[height=1.1in]{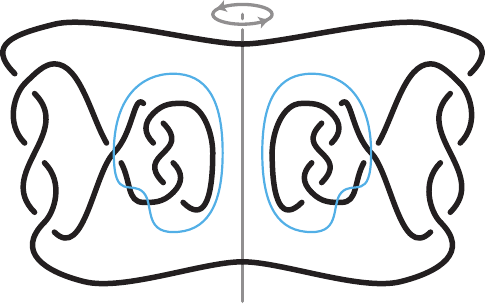}\raisebox{.6in}{~$\to$~}
        \includegraphics[height=1.1in]{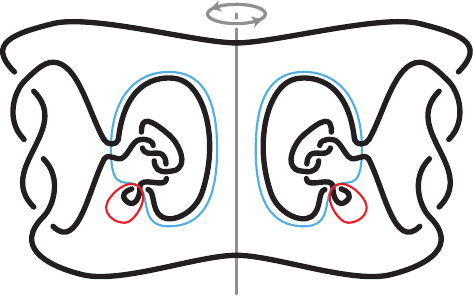}
        \caption{Removing a nugatory crossing detected by a torus via a flype move (followed by a small detour move, not pictured)}
        \label{fig:nugatory_torus_flype}
    \end{center}
\end{figure}

\begin{example}\label{ex:2/5_nugatory}
    The crossing circle in the ribbon-alternating broken surface diagram shown in the left image of Figure \ref{fi:2/5_nugatory} is nugatory and is detected by the torus shown red.  Removing this crossing circles as shown demonstrates that the diagram represents the unknotted torus in $S^4$. In fact, this proves that any pseudo-ribbon diagram with a single crossing circle represents an unknotted surface \cite{shima1998}.
\end{example}

\begin{figure}[h]
    \centering
    \includegraphics[height=1.2in]{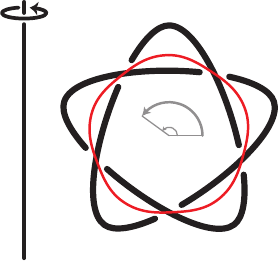}\raisebox{.6in}{$~\to$}
    \includegraphics[height=1.2in]{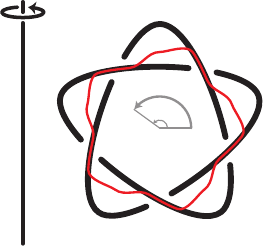}\raisebox{.6in}{$~\to$}
    \includegraphics[height=1.2in]{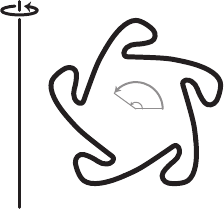}
    \caption{Removing a nugatory crossing in a ribbon-alternating diagram of the unknotted torus in $S^4$}
    \label{fi:2/5_nugatory}
\end{figure}

\begin{example}\label{ex:1/6_nugatory}
    All three crossing circles in the ribbon-alternating broken surface diagram shown in the left image of Figure \ref{fi:1/6_nugatory} are nugatory and are detected by the three tori shown red, green, and purple.  Removing these crossing circles as shown demonstrates that the diagram represents the unknotted torus in $S^4$.
\end{example}

\begin{figure}[ht]
    \centering
    \includegraphics[height=1.1in]{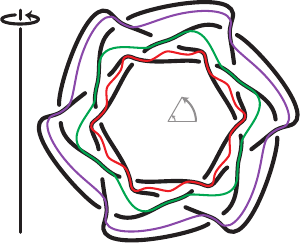}\raisebox{.5in}{$~\to$\!\!\!}
    \includegraphics[height=1.1in]{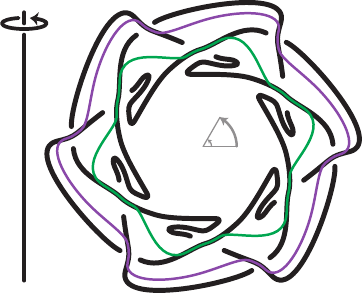}\raisebox{.5in}{$~\to$\!\!\!}
    \includegraphics[height=1.1in]{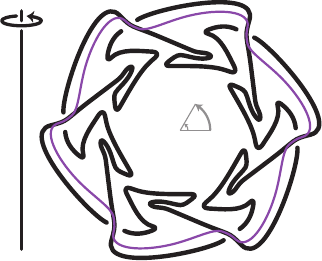}\raisebox{.5in}{$~\to$\!\!\!}
    \includegraphics[height=1.1in]{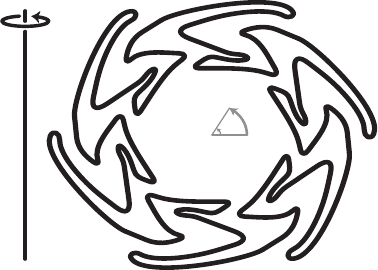}
    \caption{Removing three nugatory crossings in a ribbon-alternating diagram of the unknotted torus in $S^4$}
    \label{fi:1/6_nugatory}
\end{figure}

\begin{example}\label{ex:35_nugatory}
    Figure \ref{fi:35_nugatory} shows that one can construct the unknotted torus as a turned-spin of the $(3,5)$ torus knot. This shows in particular that turned spins need not respect the nonalternating condition.
\end{example}

\begin{figure}[ht]
    \centering
    \includegraphics[height=1.1in]{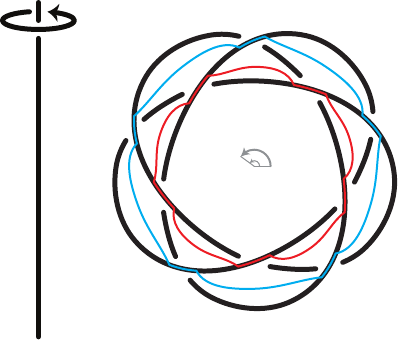}\raisebox{.5in}{$~\to$\!\!\!}
    \includegraphics[height=1.1in]{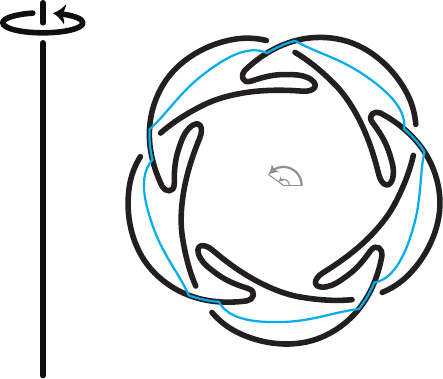}\raisebox{.5in}{$~\to$\!\!\!}
    \includegraphics[height=1.1in]{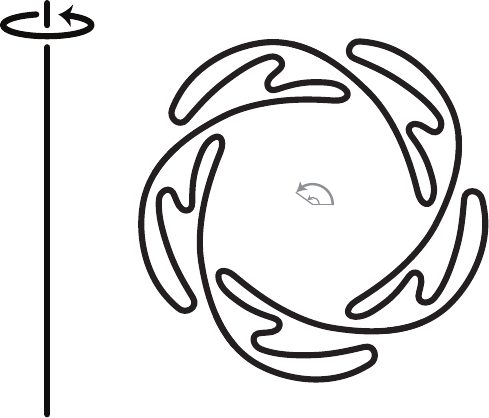}
    \caption{This turned spin of the classical $(3,5)$ torus knot is the unknotted torus in $S^4$. The former is nonalternating, but the latter is ribbon-alternating---see Figure \ref{fig:Klein_bottle_spun}, center, for example.}
    \label{fi:35_nugatory}
\end{figure}

Now suppose that a 2-sphere $F$, rather than a torus, detects a nugatory crossing $\gamma$ in a broken surface diagram $E\subset S^3$ of a knotted surface $K\subset S^4$.  If the two disks of $F\cut \gamma$ lie in the same component of $S^3\cut E$, then taking an arc between these disks and surgering along it yields a torus that also detects $\gamma$, and we can remove $\gamma$ in either of the manners just described.  Next, we describe how to remove any nugatory crossing detected by a sphere.  See Figure \ref{fig:Nugatory_Sphere}.

\begin{figure}
    \begin{center}
        \includegraphics[height=1.2in]{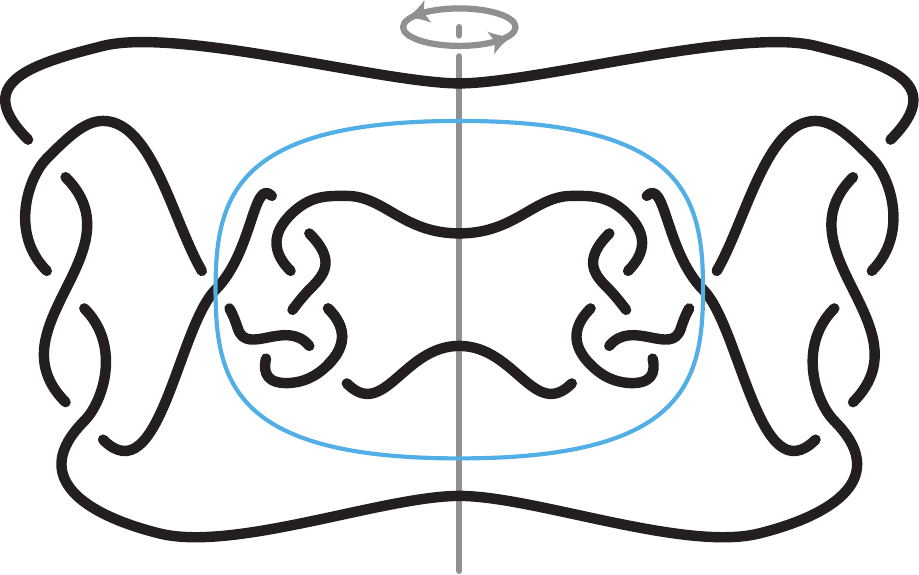}
        \raisebox{.6in}{$\longrightarrow$}
        \includegraphics[height=1.2in]{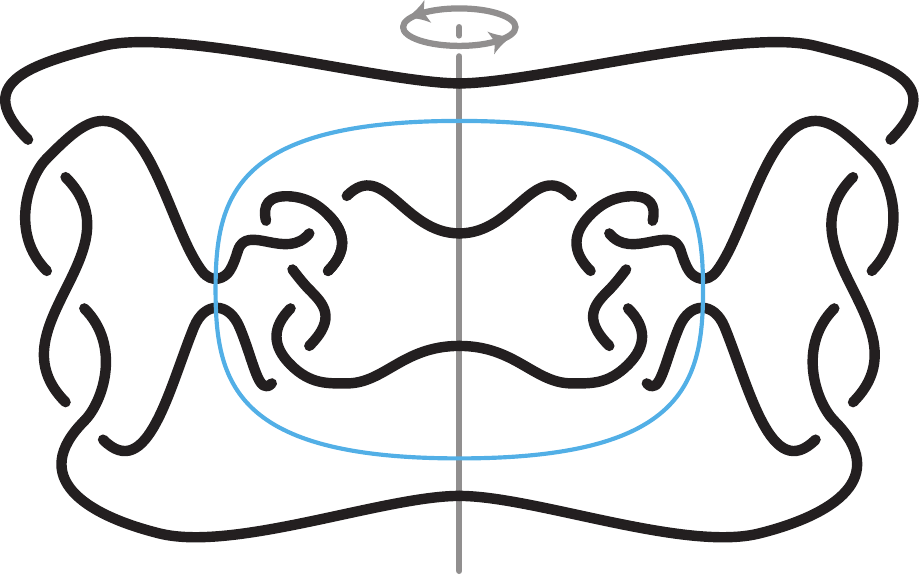}
        \caption{Left: Detecting a nugatory crossing with a sphere. Right: Removing the nugatory crossing as described in Remark \ref{rem:nugatory_sphere}.}
        \label{fig:Nugatory_Sphere}
    \end{center}
\end{figure}

\begin{remark}\label{rem:nugatory_sphere}
Suppose that $E\subset S^3$ is an ribbon-alternating diagram of a knotted surface $K\subset S^4$ and $Q\subset S^3$ is a sphere that detects a nugatory crossing $\gamma$ in $E$.  Let $Y$ be a properly embedded disk in either 3-ball of $S^3\cut Q$ with $\partial Y=\gamma$, let $S$ be a standard 3-sphere obtained as the suspension of $Q$ in $S^3\times [-1,1]$, and let $X$ be the 4-ball of $S^4\cut S$ which contains $Y$.  Let $X_+$ and $X_-$ be the 4-balls obtained by slightly enlarging and shrinking $X$, so that the Menasco neighborhood of $\gamma$ lies entirely inside $X_+$ but is disjoint from $X_-$.  Perform an ambient isotopy of $S^4$ which fixes $S^4\setminus\mathring{X}_+$ pointwise and rotates $X_-$ by $\pi$ about $Y\cap X_-$, ultimately returning $Y\cap X_-$ to itself setwise, and carrying the Menasco neighborhoods in $X_-$ to sets that intersect $S^3$ in such a way that they remain viable Menasco neighborhoods. Then perturb. It is possible to do all of this in a way that removes the crossing $\gamma$ while, essentially, reflecting the part of $E$ on the same side as $Y$ across $Y$.  See Figures \ref{fig:Nugatory_Sphere} and \ref{fig:nugatory_sphere_parts}.  
\end{remark} 

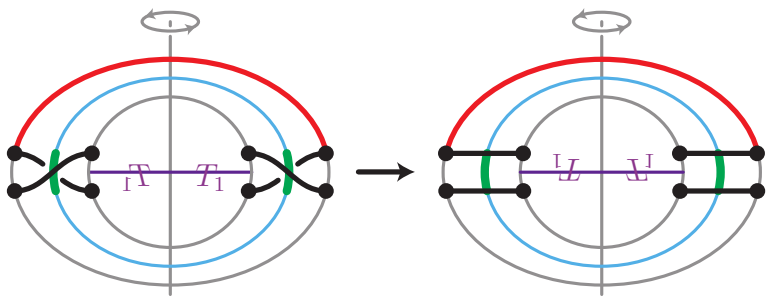
\begin{figure}
    \begin{center}
        \labellist 
        \large
        \pinlabel{{\color{Fuchsia}     \scalebox{-1}[-1]{$T_1$}}} at 260 60
        \pinlabel{{\color{Fuchsia}\scalebox{1}[-1]{$T_1$}}} at 295 60
        \pinlabel{{\color{Fuchsia}     \scalebox{-1}[1]{$T_1$}}} at 60 55
        \pinlabel{{\color{Fuchsia}\scalebox{1}[1]{$T_1$}}} at 95 55
        \endlabellist
        \includegraphics[height=1.5in]{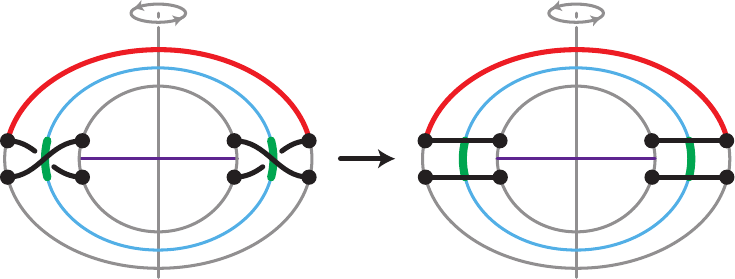}
        \caption{Removing a nugatory crossing detected by a sphere}
        \label{fig:nugatory_sphere_parts}
    \end{center}
\end{figure}

We end this section with some questions.

\begin{restatable}{question}{TaitOne}\label{Q:Tait_1}
    Let $E\subset S^3$ be a connected ribbon-alternating diagram of knotted surface $K\subset S^4$. If $E$ has no nugatory crossings, must $E$ have the minimal number of crossing circles among all pseudo-ribbon diagrams of $K$?
\end{restatable}

\begin{restatable}{question}{TaitFlypes}\label{q:flypes}
    Let $E_1,E_2\subset S^3$ be connected ribbon-alternating diagrams of the same knotted surface $K\subset S^4$, and assume that neither $E_i$ has any nugatory crossings. 
    \begin{enumerate}[label=(\arabic*)]
        \item Must $E_1$ and $E_2$ be related by flypes and the moves described in Procedure \ref{proc:unknot}?
        \item  If both $E_1$ and $E_2$ have the form described in Theorem \ref{thm:alternating}, must $E_1$ and $E_2$ be related by flypes alone?
    \end{enumerate}
\end{restatable}

\begin{restatable}{question}{nonsplit}\label{q:nonsplit}
    Let $E\subset S^3$ be an ribbon-alternating diagram of knotted surface $K\subset S^4$. If $E$ is connected and has no nugatory crossings, must $K$ be nonsplit?
\end{restatable}

\begin{remark} It is instructive to compare Question \ref{q:nonsplit} to the classical situation: a classical alternating link is nonsplit if and only if it has a connected diagram---no need to mention nugatory crossings---but a classical alternating tangle is nonsplit (meaning that every diagram is connected) if and only if it has a {\it reduced} alternating diagram. Split knotted surfaces, however, can have connected ribbon-alternating diagrams: for example, spin any connected alternating diagram of any split (e.g., trivial) tangle.
\end{remark}

\begin{restatable}{question}{connectsum}\label{Q:Connect_Sum}
    Let $E\subset S^3$ be a connected ribbon-alternating diagram of a knotted surface $K\subset S^4$. If $E$ has no nugatory crossings and does not decompose as a diagrammatic connect sum, does it follow that $K$ does not decompose as a connect sum (say, along any standard 3-sphere)?
\end{restatable}

\section{Essential spanning solids}\label{sec:essential}

In this last section, we define a notion of essential spanning solids in analogy with the classical case.

\begin{definition}\label{def:pi1essential}
Let $M\subset S^4$ be a spanning solid for a {knotted surface} $K\subset S^4$. 
Consider the inclusion map $j:\mathring{M}\hookrightarrow S^4\setminus K$, and the restriction $\iota:\partial(\nu M)\hookrightarrow S^4\setminus\mathring{\nu}M$ of $j$ to $\partial(\nu M)$.

We say that $M$ is \emph{$\pi_1$-injective} if every circle $\gamma$ with $[\gamma]\in \ker(\iota_\star)$ intersects $K$ unless $[\gamma]=0\in\pi_1(M)$.
Equivalently, $M$ is $\pi_1$-injective if and only if $\iota$ induces an injection $\iota_*$ of fundamental groups.

We say that $M$ is \emph{$\pi_1$-essential} if every circle $\gamma$ with $[\gamma]\in \ker(\iota_*)$ intersects $K$ at least twice unless $[\gamma]=0\in\pi_1(M)$.
Note that $M$ is $\pi_1$-essential if and only if it is $\pi_1$-injective and no circle $\gamma$ with $[\gamma]\in \ker(\iota_*)$ intersects $K$ exactly once. 
\end{definition}

The Whitney approximation theorem implies that any simple closed curve $\gamma$ with $[\gamma]\in\ker(j_*)$ is homotopic in $\partial(\nu M)$ to a smooth curve that bounds a smooth properly immersed disk in $S^4 \setminus M$.  

\begin{remark}
    Note that if $M$ is orientable then it is $\pi_1$-injective if and only if it is $\pi_1$-essential. This follows from the fact that $\partial(\nu M)\cong M\cup_{K}(-M)$, and so no curve in $\partial(\nu M)$ may intersect $K$ exactly once. On the other hand, when $M$ is non-orientable then $\nu(M)$ is a twisted $I$-bundle over $M$ and $K$ is nonseparating in $\partial(\nu M)$. Thus, when $M$ is non-orientable, these two notions may not agree.
\end{remark}

Ozawa proved that any Murasugi sum of $\pi_1$-essential spanning surface in $S^3$ is $\pi_1$-essential \cite[Lemma 3.4]{ozawa11}.  We now adapt this fact to spanning solids in $S^4$.

\begin{theorem}
\label{thm:essential} 
   Any Murasugi sum $M=M_0*M_1$ of $\pi_1$-essential spanning solids is $\pi_1$-essential.
\end{theorem}

Rather than prove this now, we prove it below, together with the more general  Theorem \ref{thm:essential Generalized}, since the proofs are almost identical. We note the following corollary of Theorem \ref{thm:essential} (and Corollary \ref{cor:spin_tangle_surface}).

\begin{corollary}
    If $G$ is a sunrise tree in which each vertex has weight at least 2, then the arborescent solid described by $G$ is $\pi_1$-essential.
\end{corollary}

\subsection{Generalized Murasugi sum}\label{S:Generalized}

\begin{definition}\label{def:generalized_murasugi_sum}
    Let $M\subset S^4$ be a spanning solid. We say that $M$ \emph{decomposes as a generalized Murasugi sum} if there is a standard 3-sphere $Q\subset S^4$ such that  $U=Q\cap M$ and $V=Q\cut U$ are each 3-dimensional 1-handlebodies, thus inducing a Heegaard splitting of $Q$, and (generically) $\partial U\cap \partial M$ is a 1-manifold. If $B^4_0$ and $B^4_1$ are the components of $S^4\cut Q$, we write $M_i=M\cap B^4_i$ and denote $M=M_0\hat{*}M_1$.

    Conversely, let $M_0\subset S^4_0$ and $M_1\subset S^4_1$ be spanning solids with $\partial M_i=K_i$. Suppose that there are 4-balls $B^4_i\subset S^4_i$ such that $U_0=\partial B^4_0\cap M_0$ and $U_1=\partial B^4_1\cap M_1$ are 3-dimensional 1-handlebodies. Assume also that there is a diffeomorphism $f:U_0\to U_1$ such that 
    $f(\partial{U}_0\cap\partial M_0)\cup(\partial{U}_1\cap \partial M_1)=\partial U_1$. Finally, assume 
    $f(\partial{U}_0\cap\partial M_0)\cap\partial{M}_1$ is a 1-manifold
    and $f$ extends to an orientation-reversing diffeomorphism $F:\partial B^4_0\to\partial B^4_1$.
    
    Then we can glue $M_0$ to $M_1$ via $f$ in the 4-sphere $B^4_0\cup_F B^4_1$. We write $U=f(U_1)=U_2$, $Q=f(\partial B^4_0)=\partial B^4_1$ and $V=Q\setminus\mathring{U}$. If $V$ is a handlebody, we call the resulting solid $M$ a \emph{generalized Murasugi sum} of $M_0$ and $M_1$ along $Q$, and denote it $M=M_0\hat{*}M_1$.  
\end{definition}

Note that Murasugi sum $M_0*M_1$ is the special case of generalized Murasugi sum $M_0\hat{*}M_1$ in which $U=Q\cap (M_0\hat{*} M_1)$ is a 3-ball. Thus, the following theorem statement is a generalization of Theorem \ref{thm:essential}.

\begin{theorem}
\label{thm:essential Generalized} 
    Any generalized Murasugi sum $M=M_0\hat{*}M_1$ of $\pi_1$-essential spanning solids is $\pi_1$-essential.
\end{theorem}

Since Theorem \ref{thm:essential Generalized} implies Theorem \ref{thm:essential}, we will only prove Theorem \ref{thm:essential Generalized}.

\begin{proof}[Proof of Theorem \ref{thm:essential Generalized}]
    Let $U=M_0\cap M_1\subset M$ be the 3-manifold
    along which $M_0$ and $M_1$ have been plumbed, and let $V$ be a handlebody in $S^4\cut M$ with $\partial V=\partial U$, such that $U\cup V$ is a standard 3-sphere. Note that if we have a non-generalized Murasugi sum, then $U$ and $V$ are 3-balls. Decompose $S^4$ into two 4-balls $B^4_0\cup B^4_1$ along the standard 3-sphere $U\cup V$, so $B^4_0\cap B^4_1=U\cup V$ and $M\cap B^4_i=M_i$.
    
    Write $K=\partial M$ and $K_i=\partial M_i$. Note that $S^4\cut M=(B^4_0\cut M_0)\cup_V(B^4_1\cut M_1)$, and that $\partial(S^4\cut M)$ contains two copies of $U$. We use $U_i$ to denote the copy of $U$ in $\partial B^4_i\cut M_i$. See Figure \ref{Fi:1D_Murasugi_Sum}.
    
    \begin{figure}[ht]
        \centering
        \labellist \small \hair 4pt
        \pinlabel {$=$} at 145 223
        \pinlabel {$*$} at 310 223
        \pinlabel {$B^4_0$} at 65 250
        \pinlabel {$B^4_1$} at 65 150
        \pinlabel {$M_0$} at 220 200
        \pinlabel {$M_1$} at 400 150
        \pinlabel {{\color{ForestGreen} $U$}} at 65 180
        \pinlabel {{\color{CornflowerBlue} $V$}} at 115 223
        \pinlabel {{\color{ForestGreen} $U_0$}} at 65 68
        \pinlabel {{\color{ForestGreen} $U_1$}} at 65 43
        \endlabellist
        \includegraphics[width=.8\textwidth]{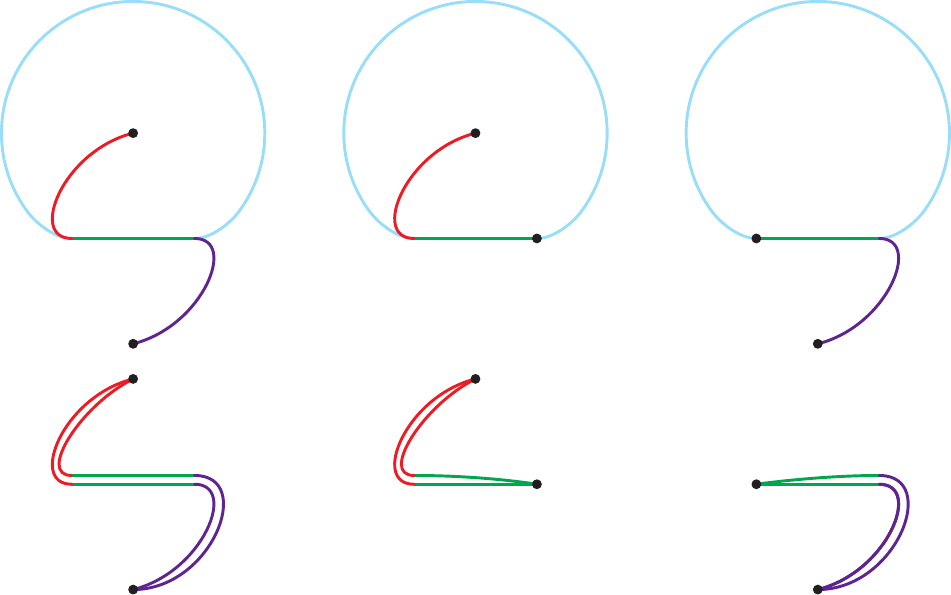}
        \caption{A schematic picture of a generalized Murasugi sum in $S^4$, labeled for the proof of Theorem  \ref{thm:essential Generalized}}
        \label{Fi:1D_Murasugi_Sum}
    \end{figure}
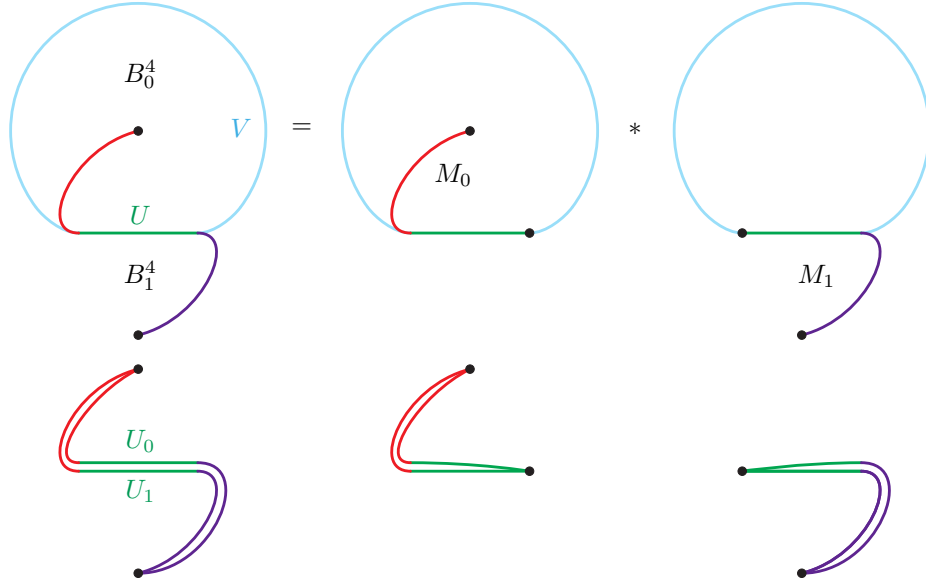
    
    Assume for contradiction that $M$ is not $\pi_1$-essential. Then, by definition there is a circle $\gamma\subset \partial(\nu M)$ with $|\gamma\pitchfork K|\leq 1$ that bounds a smooth properly immersed disk $D\subset S^4\cut M$ but bounds no disk in $\partial(\nu M)$. 
    If $\gamma$ intersects $K$ once, we assume that it does so in a single point which does not lie on $\partial V$. 
    
    Among all possible choices of such $\gamma$ and $D$, assume that these have been chosen so that the pair $(|D\pitchfork V|, |\partial D\pitchfork\partial U|)$ is minimal, lexicographically. Let $f:D^2\to D$ denote the immersion; we assume generically that any self-intersections of $D$ are disjoint from $V$.

\begin{figure}[ht]
    \centering
    \includegraphics[width=90mm]{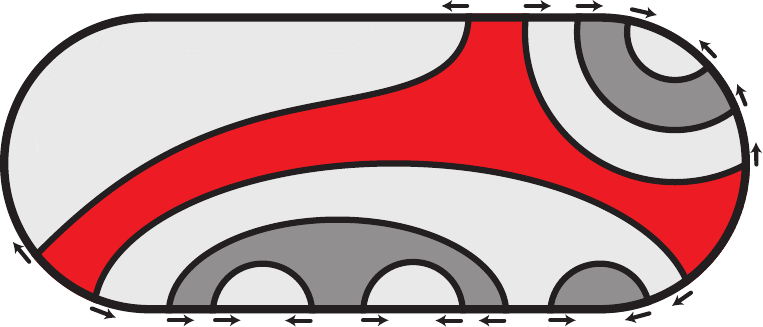}
    \caption{An example of the disk $D$, its arcs of intersection with the handlebody $V$, and markers at the endpoints of these arcs. Here, just one component of $D\cut V$, colored red, has fewer than two markers.}
    \label{fig:OM}
\end{figure}

    Let $\ell\in\mathbb{N}$ be the number of arc components of $V\cap D$. Then $V$ cuts $D$ (more precisely $f^{-1}(V)$ cuts $D^2$) into $\ell+1$ disks.  Decorate each of the $2\ell$ endpoints of $V\cap D$ on $\gamma$ with an arrow (``marker") that points in the direction where $\gamma$ runs along $U$.  
    If $\gamma$ intersects $K$, add a marker at this point on $\gamma$ as well.  There are either $2\ell$ or $2\ell+1$ markers, each of which lies on the boundary of exactly one of the $\ell+1$ disks of $D\cut  V$, so there is a disk component $D_0$ of $D\cut V$ with fewer than two markers. See Figure \ref{fig:OM}.  Assume without loss of generality that $D_0\subset B^4_0$. 
    
    \begin{figure}[!ht]
        \centering
        \labellist
        \pinlabel{$B^4_0$} at -50 700
        \pinlabel{$B^4_1$} at -50 0
        \endlabellist
        \includegraphics[width=140mm]{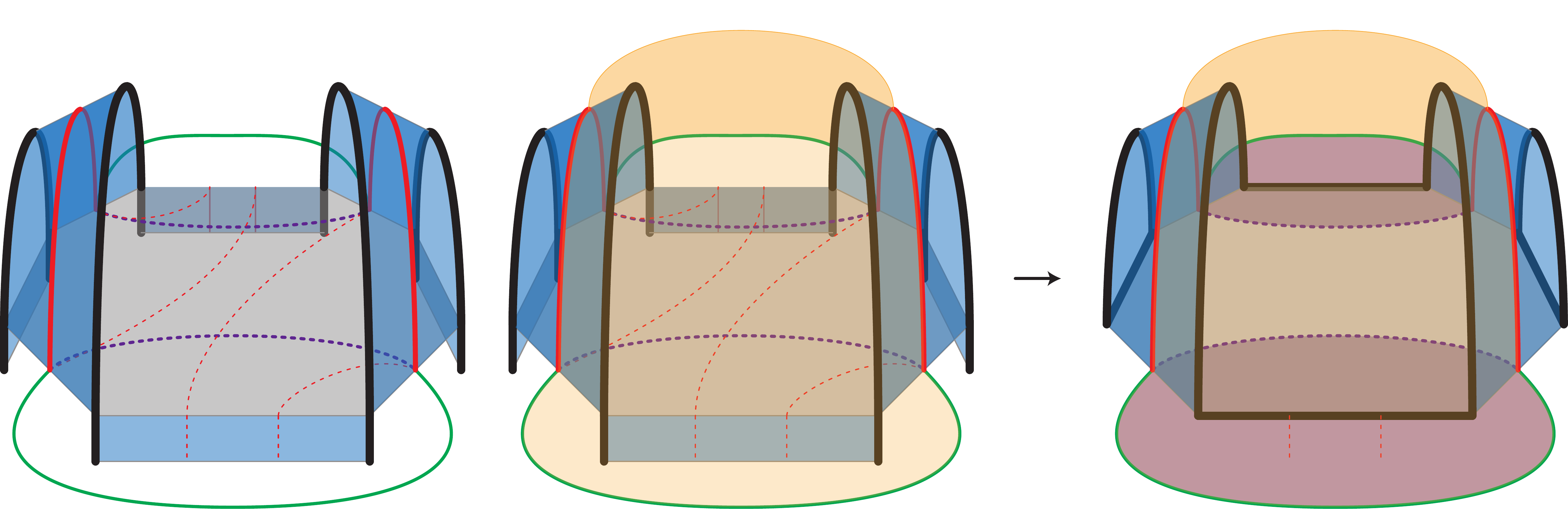}
        \caption{In this picture for the situation down a dimension, $U$ is an octagon, $U\cup V$ is a horizontal plane with $V$ not pictured, and $D_0$, shown in orange with red (highlighted parabolas) and green (horizontal arcs) boundary, intersects $V$ in two arcs, each of which is parallel through a disk in $B^4_1$ to n arc in $U$}
        \label{fig:D_0}
    \end{figure}
    
    Each arc component of $D_0\cap V$ is parallel through a homotopy rel.\ boundary 
    in $B^4_1$ (ignoring $M\cap\mathring{B}^4_1$) to an arc in $U_1$. (This is where we use the hypothesis that $V$ is a handlebody; also see Remark \ref{rem:more_generalized_murasugi_sum}.) Note that during the homotopy, each arc traces out an immersed disk. The union of all of these disks in $B^4_1$ with $D_0$ is a properly immersed disk $E_0$ in $S^4\cut M_0$. (See Figure \ref{fig:D_0} for a schematic.) Because $\partial D_0$ contains at most one marker, the circle $\gamma_0=\partial E_0$ either is disjoint from $K_0$ or intersects $K_0$ in a single point. Because $M_0$ is $\pi_1$-essential, the latter is impossible, and 
    $\gamma_0$ must be homotopically trivial in $M_0$.

\begin{figure}[!ht]
    \centering
        \labellist
        \pinlabel{$B^4_0$} at 330 150
        \pinlabel{$B^4_1$} at 330 25
        \endlabellist    \includegraphics[width=140mm]{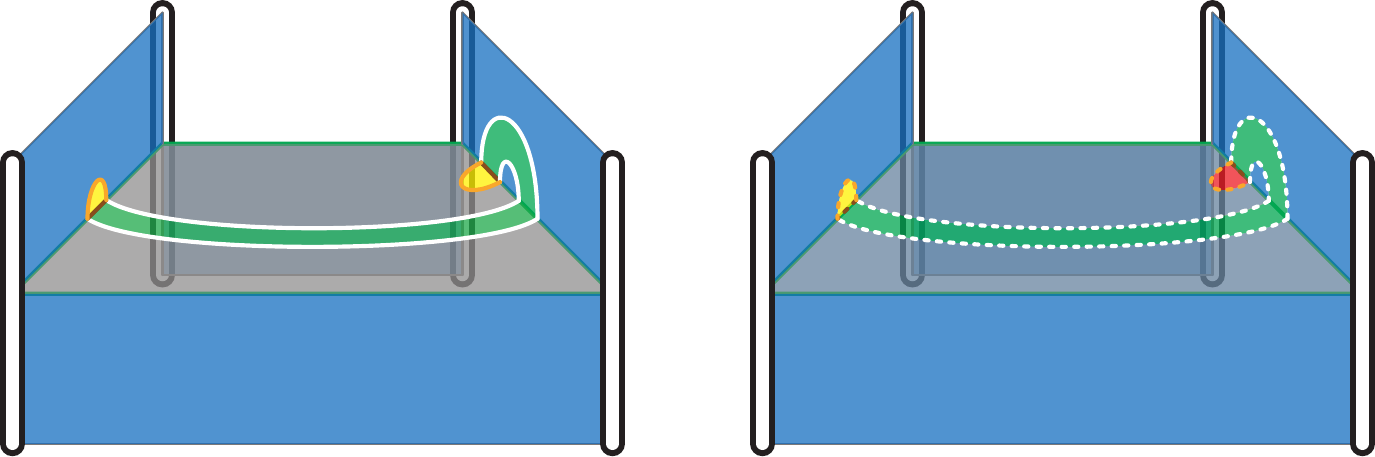}
    \caption{Schematics of the possibilities for the disk $F\subset M_0$, shown green, yellow, and red; 
    An outermost disk $F_0$ of $F\cut \partial U$, shown yellow as on the left/right of the left panel and the left of the right panel (immediately removable by isotopy) or red as in the right of the right panel (takes more work). The arcs $\alpha$ and $\beta$ lie on the boundary of $F_0$.}
    \label{fig:Trivial_Disk_F}
\end{figure}

    Consider an immersed disk $F\subset M_0$ with $\partial F=\gamma_0$.  Choose $F$ so as to minimize $|F\pitchfork \partial U|$.  Consider an outermost disk $F_0$ of $F\cut\partial U$. See Figure \ref{fig:Trivial_Disk_F}.  The boundary of $F_0$ is a union of arcs $\alpha, \beta$, where $\alpha$ is a component of $\gamma_0\cut \partial U$ and $\beta\subset \partial U\cut \gamma_0$.  We cannot have $\alpha\subset \partial D$ due to the minimality of $|D\pitchfork V|$ and $|\partial D\pitchfork\partial U|$, so $\alpha$ 
    and $F_0$ must lie in $U$.  More precisely, since $\alpha\subset\gamma_0\subset M_0$, $\alpha$ lies in $U_1$ and therefore, by minimality, 
    is homotopic rel.\ boundary to an arc component $\alpha'$ of $D\cap V$. The circle $\alpha'\cup\beta$ bounds a disk $F'_0$ in $V$.

\begin{figure}[!ht]
    \centering
    \labellist
        \pinlabel{$B^4_1$} at 330 150
        \pinlabel{$B^4_0$} at 330 25
        \endlabellist 
    \includegraphics[width=140mm]{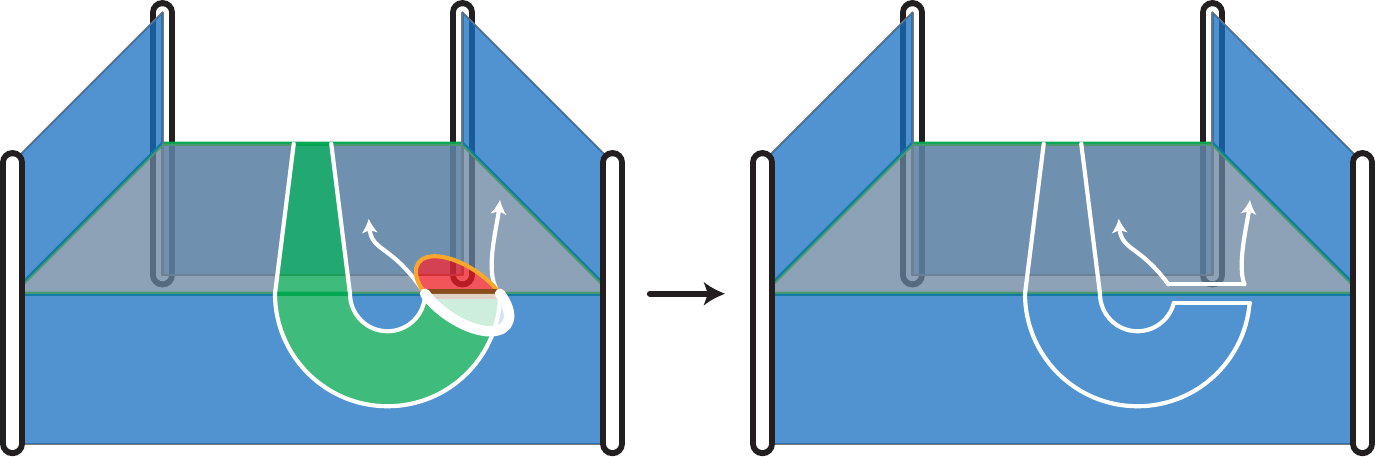}
    \caption{A schematic for surgering $D$ along $F'_0$---the color scheme is the same is in Figure \ref{fig:Trivial_Disk_F}; additionally, the arc $\alpha'$ (thick), the disk $F'_0$ (co-bounded by $\alpha'$ and the brown arc $\beta$), $\partial D$, $\partial D_1$, and $\partial D_2$ are all white.}
    \label{fig:Surgering_D}
\end{figure}

    Surger (boundary compress) $D$ along $F'_0$.  (Note that we are in 4D and this requires more care than in a classical 3D argument. Nevertheless, $F'_0$ is a subset of $V$ and hence can be framed in $V$. We rotate the associated framing so that thickening $F'_0$ to $D^2\times I$ causes $D^2\times 0$ and $D^2_1$ to lie on opposite sides of $U\cup V$ in $S^4$. We are only performing boundary compression, so there is no obstruction to ensuring that along $\alpha'$, the thickened disk $D^2\times I$ meets $D$ in $\alpha\times I$.)

    This boundary compression splits $D$ into two disks $D_1$ and $D_2$, each of whose boundary either intersects $K$ exactly once or is homotopically trivial in $\partial (\nu M)$, 
    and each of which intersects $V$ in fewer components than $D$ does. 
    The minimality of $|X\pitchfork D|$ implies that $\partial D_1$ and $\partial D_2$ are both disjoint from $K$ and homotopically trivial in $\partial(\nu M)$, which contradicts $\partial D$ being homotopically essential in $\partial(\nu M)$. 
\end{proof}

\begin{remark}\label{rem:more_generalized_murasugi_sum}
    Although Definition \ref{def:generalized_murasugi_sum} requires both $U$ and $V$ to be handlebodies, the assumption that $U$ is a handlebody is not used in the above proof of Theorem \ref{thm:essential Generalized}. Thus, the theorem also applies in this more general setting. 
\end{remark}

\begin{remark}\label{rem:state_solid_deplumb_gen}
    Remark \ref{rem:state_solid_deplumb} extends naturally to generalized Murasugi sums---if a component $X_0$ of a state $X$ is a Heegaard surface for $S^3$ that satisfies condition \eqref{E:sum_condition}, then $M_X$ decomposes as a generalized Murasugi sum along the solid that caps off $X_0$.
\end{remark}

We end this subsection with some natural questions about $\pi_1$-essential solids. 

\begin{restatable}{question}{essential}\label{Q:Pi_1_Injective}
    Does every 2-knot have a $\pi_1$-injective spanning solid?
\end{restatable}

\begin{restatable}{question}{rationalB}\label{Q:Rational_B3}
Which rational homology 3-balls embed in $S^4$?
\end{restatable}

\begin{restatable}{question}{essentialPiOne}\label{Q:finite_acyclic}
    Does every nontrivial 2-knot have a $\pi_1$-injective spanning solid whose fundamental group is nontrivial and either finite or acyclic?
\end{restatable}

Note that a positive answer to Question \ref{Q:finite_acyclic} would resolve the following open question.

\begin{restatable}{question}{acyclic}\label{Q:cyclic_pi_1}
    Does every nontrivial 2-knot have acyclic fundamental group?
\end{restatable}
Question \ref{Q:cyclic_pi_1} is often referred to as the ``unknotting problem."

\subsection{Essential state solids from pseudo-ribbon diagrams}\label{sec:ribbon_state_solids}

    Remark \ref{rem:state_solid_deplumb_gen} and Theorem \ref{thm:essential Generalized} (or more simply Remark \ref{rem:state_solid_deplumb} and Theorem \ref{thm:essential}) provide a means of obtaining classes of $\pi_1$-essential solids: first, characterize certain $\pi_1$-essential building blocks, and then describe convenient ways to combine them by (generalized) Murasugi sum. This approach motivates the following definition.

\begin{definition}
A \emph{$2n$-tangle} is a properly embedded 1-manifold $T\subset B^3$ comprised of $n$ arcs and any number of circles.  A \emph{spanning surface} for $T$ is an embedded compact surface $F\subset B^3$ with $T\subset \partial F\subset T\cup\partial B^3$.
\end{definition}

\begin{proposition}\label{prop:pi_1_induced}
Suppose $F\subset B^3$ is a spanning surface for a $2n$-tangle $T\subset B^3\subset S^3$, $M\subset S^4$ is the solid obtained by spinning $F$ (without deformation),  $K=\partial M$, and $q:S^4\to B^3$ is the associated quotient map.  
Then the restrictions $q:S^4\setminus K\to B^3\setminus T$ and $q:\mathring{M}\to \mathring{F}$ and $q:\partial(\nu M)\to\partial(\nu F)$ induce isomorphisms of fundamental groups. 
\end{proposition}

\begin{proof}
The spinning construction gives decompositions $S^4=(B^3\times S^1)\cup(S^2\times D^2)$ and $K=(T\times S^1)\cup(\partial T\times D^2)$ and thus $S^4\setminus K=((B^3\setminus T)\times S^1)\cup((S^2\setminus\partial D)\times D^2)$. If $\langle g_1,\dots,g_m~|~r_1,\dots,r_n\rangle$ is a presentation of $\pi_1(B^3\setminus T)$, then $\pi_1((B^3\setminus T)\times S^1)$ has a presentation of the form $\langle g_1,\dots,g_m,a~|~r_1,\dots,r_n\rangle$. Taking the basepoint and a curve $\alpha$ representing $a$ both to lie in $S^2\times S^1=\partial(B^3\times S^1)=\partial(S^2\times D^2)$, we see that $\alpha$ is nullhomotopic in $(S^2\setminus \partial T)\times D^2$.  The proposition thus follows from the Van Kampen theorem.\end{proof}

\begin{corollary}
\label{cor:spin_tangle_surface}
Suppose $F\subset B^3$ is a spanning surface for a $2n$-tangle $T\subset B^3\subset S^3$ and $M\subset S^4$ is the solid obtained by spinning $F$. If $F$ is $\pi_1$-essential in $S^3$, then $M$ is $\pi_1$-essential in $S^4$.  
\end{corollary}

\begin{proof}
Denote $q:S^4\setminus K\to B^3\setminus T$ as in Proposition \ref{prop:pi_1_induced}, and denote the inclusion map $j:\mathring{M}\hookrightarrow S^4\setminus K$ and its restriction $\iota:\partial(\nu M)\hookrightarrow S^4\setminus\mathring{\nu}M$ as in Definition \ref{def:pi1essential}.  Let $g:I\to \partial \nu M$ be a loop, write $\gamma=g(I)$, and suppose that $[g]\in \ker~\iota_*$ and  $|\gamma\cap \partial M|\leq 1$.  Then $|q(\gamma)\cap \partial F|\leq 1$ and, by Proposition \ref{prop:pi_1_induced}, $[q\circ g]\in\ker~(q\circ\iota)_*$. 
Since $F$ is $\pi_1$-essential, it follows $[q\circ g]=0\in\pi_1(\partial(\nu F))$. 
Hence by Proposition \ref{prop:pi_1_induced}, $[g]=0\in \pi_1(\partial\nu M)$.
\end{proof}

\begin{proposition}\label{prop:ess_cb}
    Let $E\subset D^2$ be an alternating diagram of a $2n$-tangle $T\subset B^3$. If no nugatory crossing in $E$ is detected by a circle,\footnote{That is, if no circle $\beta\subset D^2$ satisfies $\beta\cap E=\{p\}$ for a crossing point $p$ of $E$} then both checkerboard surfaces are $\pi_1$-essential in $S^3$, as are the solids obtained by spinning them (without deformation).
\end{proposition}

\begin{proof}
    These checkerboard surfaces are $\pi_1$-essential by Lemma 3.3 of \cite{ozawa11}, so these solids are $\pi_1$-essential by Corollary \ref{cor:spin_tangle_surface}.
\end{proof}

\begin{figure}[ht]
    \begin{center}
        \includegraphics[height=1.5in]{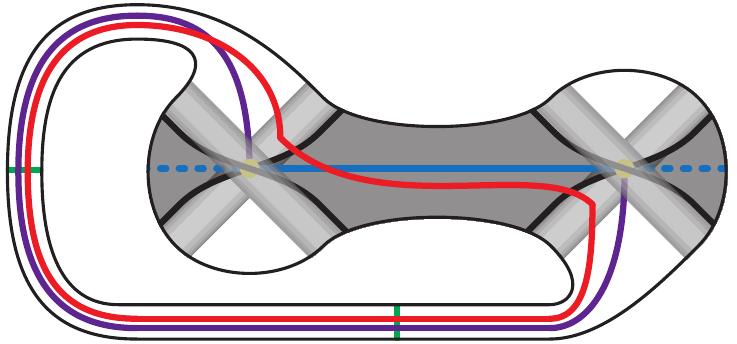}
        \caption{Checkerboard surfaces from certain alternating tangle diagrams are $\pi_1$-essential.}
        \label{fig:ess_cb}
    \end{center}
\end{figure}

Even the trivial implication of Proposition \ref{prop:ess_cb} is intriguing:

\begin{example}\label{ex:trivial_s2_s2}
The trivial {knotted surface} $L$ of two spherical components has $\pi_1$-essential spanning solids with fundamental group of arbitrarily high rank $n$.  To see this, take an alternating diagram $E$ of the trivial 4-tangle that has at least $2n+1$ crossings. The checkerboard surfaces $F_1$ and $F_2$ of $E$ satisfy $\beta_1(F_1)+\beta_1(F_2)=2n$, so at least one of $\beta_1(F_i)\geq n$. The same holds for the solids $M_i$ obtained by spinning $F_i$, so at least one of $\operatorname{rank}(\pi_1(M_i))\geq n$, and Proposition \ref{prop:ess_cb} implies that both $M_i$ are $\pi_1$-essential. 
\end{example}

For each diagram $E$ in Example \ref{ex:trivial_s2_s2}, there is a triangle of $D^2\cut E$ incident to $\partial D^2$ and a single crossing $p$, so there is a properly embedded arc $\beta\subset D^2$ with $\beta\cap E=\{p\}$.  Hence, no such diagram $E$ is reduced, as the arc $\beta$ detects the nugatory crossing $p$.  When we spin everything, $\beta$ yields a sphere which detects the nugatory crossing obtained from $p$.  The same phenomenon holds more generally: if an arc or circle $\beta$ detects a nugatory crossing $p$ in any $2n$-tangle, then, after spinning, the sphere or torus from $\beta$ detects the nugatory crossing circle from $p$.  

It is striking how differently the two kinds of nugatory crossings affect $\pi_1$-essentiality.  In Example \ref{ex:trivial_s2_s2}, for example, every crossing either is nugatory or eventually becomes nugatory as other nugatory crossings are removed, and yet the entire complicated checkerboard surface is $\pi_1$-essential, as is the solid obtained by spinning it.  Morally, there ought to be something ``inessential'' about such a complicated spanning solid for the most trivial 2-component {knotted surface}.  This 
motivates:

\begin{definition}\label{def:pi_2_ess}
    A spanning solid $M$ for a knotted surface $K\subset S^4$ is \emph{$\pi_2$-essential} if inclusion $i:\mathring{M}\hookrightarrow S^4\setminus K$ induces an injective map $i_*:\pi_2(\mathring{M})\hookrightarrow \pi_2(S^4\setminus K)$. 
\end{definition}

Proving that Murasugi sum in dimension 4 respects this property, however, is beyond the scope of this project.  We pose the following question.

\begin{restatable}{question}{piTwo}\label{Q:Murasugi_sum_pi_1_injective}
    Suppose $M=M_1*M_2$ is a spanning solid constructed as the Murasugi sum of spanning solids $M_1$ and $M_2$.  \begin{enumerate}[label=(\arabic*)]
        \item If $M_1$ and $M_2$ are $\pi_2$-injective, must $M$ also be $\pi_2$-injective? 
        \item If $M_1$ and $M_2$ are both $\pi_1$-essential and $\pi_2$-injective, must $M$ also be $\pi_2$-injective?
    \end{enumerate}
\end{restatable}

\begin{example}
 Figure \ref{fi:35_adequately_homogeneous} shows a pseudo-ribbon diagram of the spun $(3,5)$ torus knot whose Seifert state is $\pi_1$-essential for the reasons described at the beginning of this subsection, using Theorem \ref{thm:essential Generalized} and Proposition \ref{prop:ess_cb}.   
\end{example}

\begin{figure}[ht]
    \centering
    \includegraphics[width=40mm]{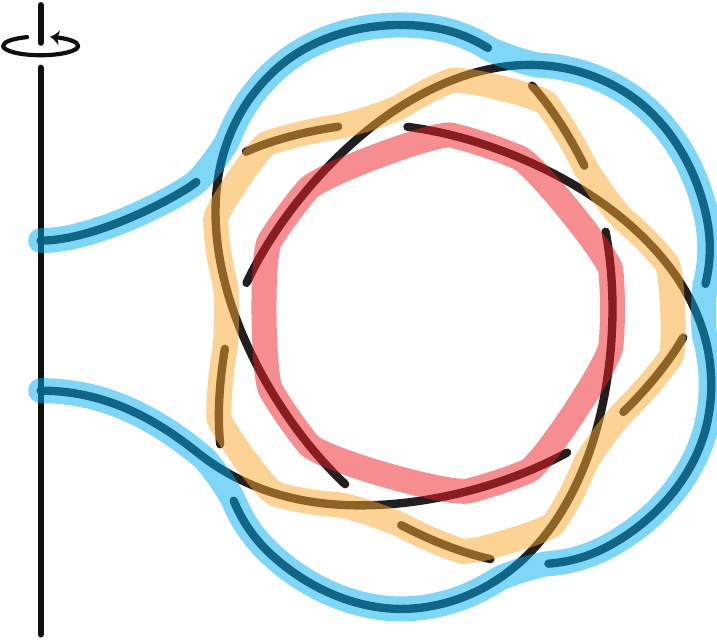}
    \hspace{10mm}
    \includegraphics[width=40mm]{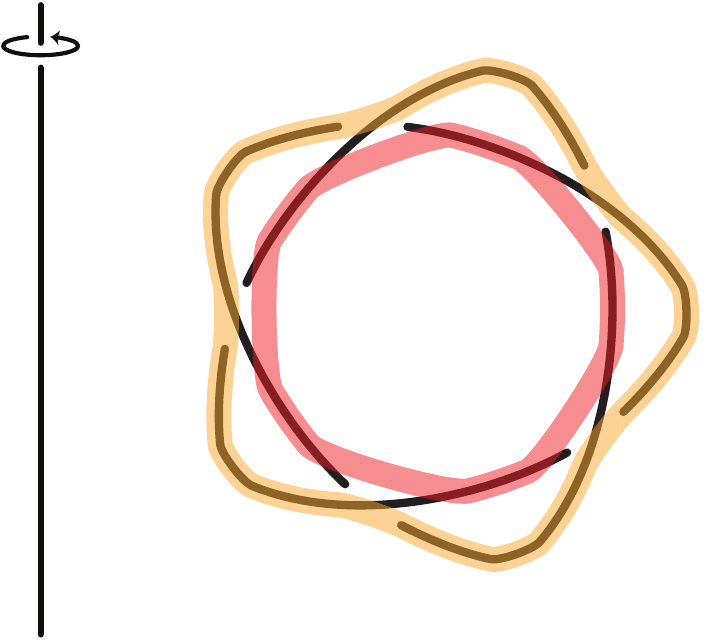}
    \hspace{10mm}
    \includegraphics[width=40mm]{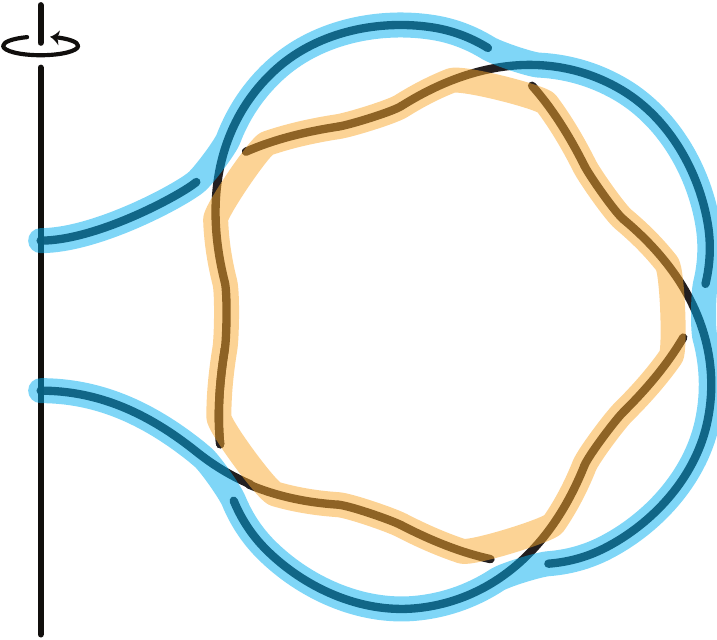}
    \caption{The state solid from the Seifert state of the pseudo-ribbon diagram of the spun $(3,5)$ torus knot shown left is a generalized Murasugi sum of the solids from the states shown center and right.}
    \label{fi:35_adequately_homogeneous}
\end{figure}

\begin{example}\label{ex:T34}
    Consider the spin $M$ of the incompressible Seifert surface for the torus link $T(3,4)$.  Figure \ref{fi:T34} shows that $M$ is a Murasugi sum of the the spins of the fiber surfaces of the torus links $T(2,\pm 4)$, and Figure \ref{fi:T34B} shows $M$ as a different generalized Murasugi sum of the same two solids. 
    The 2-knot $\partial M$ is non ribbon-alternating by Theorem \ref{thm:non_alternating}.
\end{example}

\begin{figure}[ht]
    \centering
    \labellist
    \pinlabel {$=$} at 460 250
    \pinlabel {$*$} at 820 250
    \endlabellist
    \includegraphics[width=100mm]{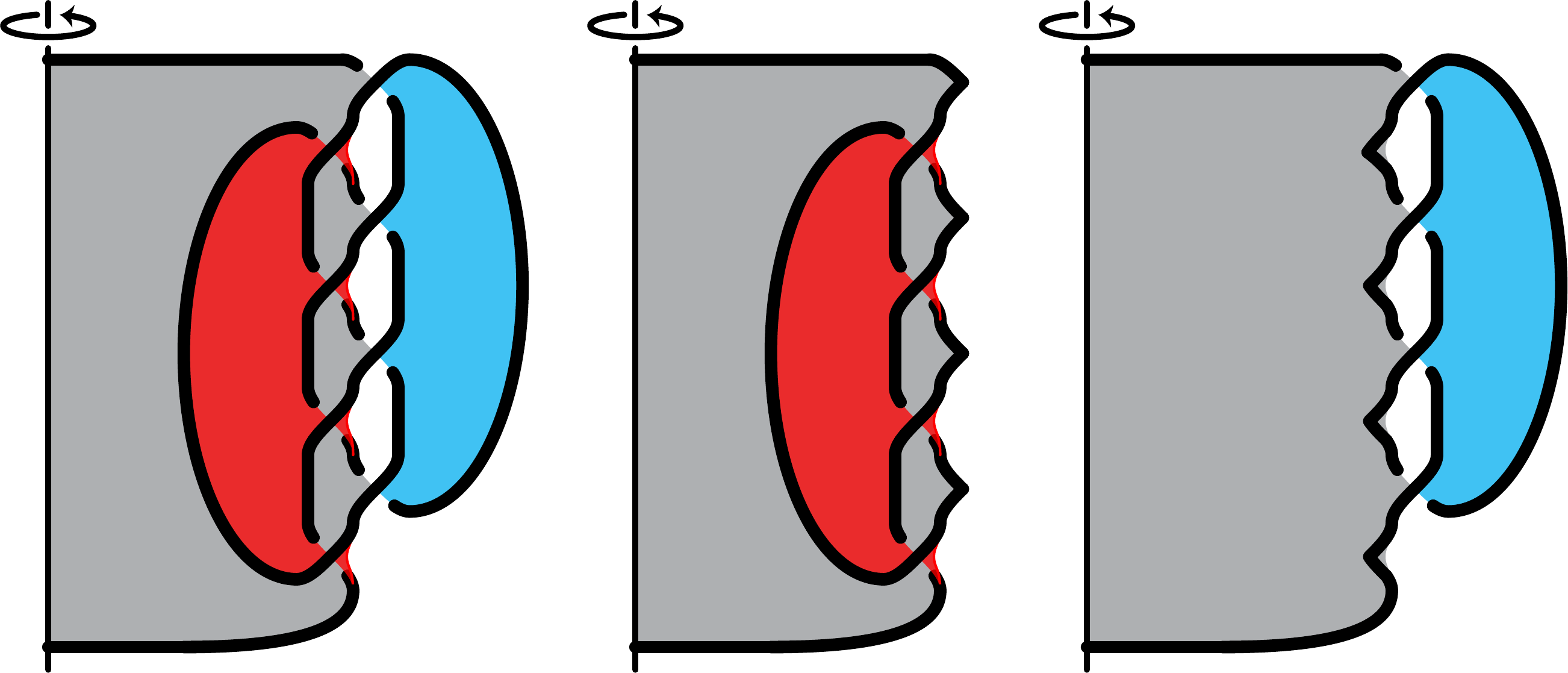}
    \caption{A Murasugi sum of $\pi_1$-essential checkerboard solids from ribbon-alternating diagrams.}
    \label{fi:T34}
\end{figure}

\begin{figure}[ht]
    \centering
    \labellist
    \pinlabel {$=$} at 400 170
    \pinlabel {$\hat{*}$} at 815 170
    \endlabellist
    \includegraphics[width=.8\textwidth]{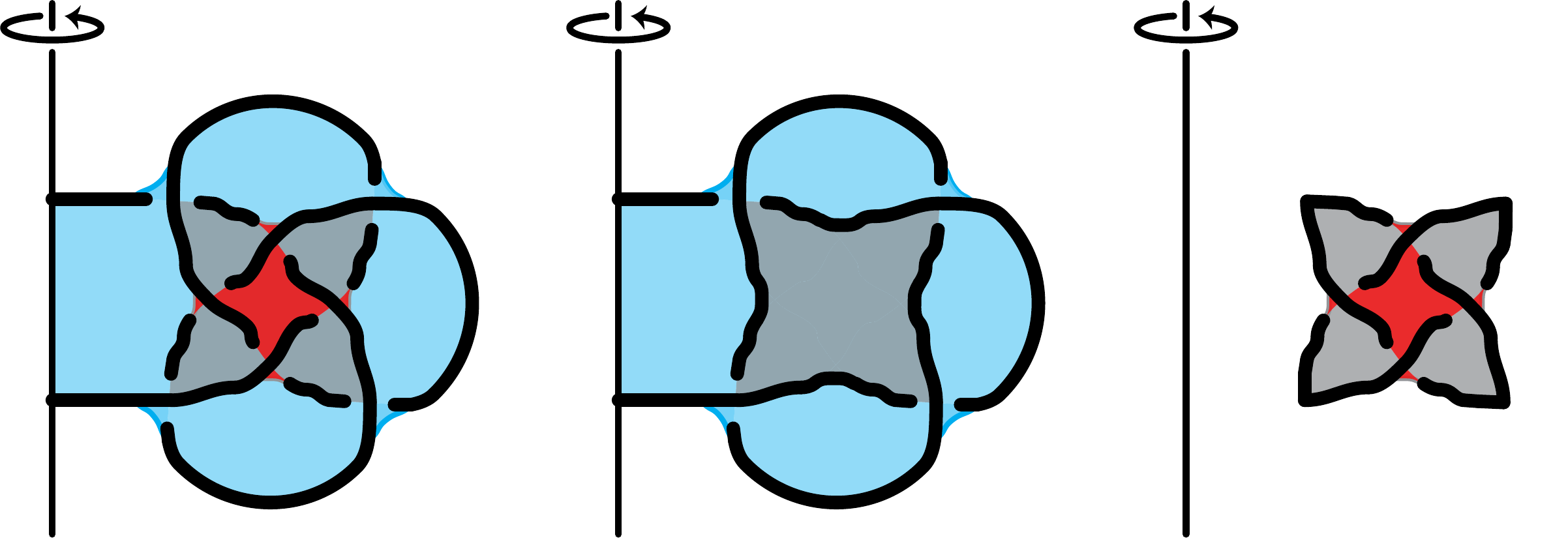}
    \caption{A second way to realize the solid from Figure \ref{fi:T34} as a generalized Murasugi sum of $\pi_1$-essential checkerboard solids from ribbon-alternating diagrams.}
    \label{fi:T34B}
\end{figure}

\section{Collected questions}\label{sec:questions}

We have open questions scattered throughout the paper. We collect them here for the reader's convenience.

\sunrisemoves*

\statewithsolidconnectsum*

\qp*

\connectedstateone*

\connectedstatetwo*

\spunMobius*

\spunNAlink*

\TaitOne*

\TaitFlypes*

\nonsplit*

\connectsum*

\essential*

\rationalB*

\essentialPiOne*

\acyclic*

\piTwo*

\bibliographystyle{alpha}
\bibliography{references}

@article {boyle_torus,
    AUTHOR = {Boyle, Jeffrey},
     TITLE = {The turned torus knot in {$S^4$}},
   JOURNAL = {J. Knot Theory Ramifications},
  FJOURNAL = {Journal of Knot Theory and its Ramifications},
    VOLUME = {2},
      YEAR = {1993},
    NUMBER = {3},
     PAGES = {239--249},
      ISSN = {0218-2165,1793-6527},
   MRCLASS = {57Q45 (57M25)},
  MRNUMBER = {1238874},
MRREVIEWER = {Charles\ Livingston},
       DOI = {10.1142/S0218216593000155},
       URL = {https://doi.org/10.1142/S0218216593000155},
}

@article {cartersaito_solids,
    AUTHOR = {Carter, J. Scott and Saito, Masahico},
     TITLE = {A {S}eifert algorithm for knotted surfaces},
   JOURNAL = {Topology},
  FJOURNAL = {Topology. An International Journal of Mathematics},
    VOLUME = {36},
      YEAR = {1997},
    NUMBER = {1},
     PAGES = {179--201},
      ISSN = {0040-9383},
   MRCLASS = {57Q45 (57N10)},
  MRNUMBER = {1410470},
MRREVIEWER = {Masayuki\ Yamasaki},
       DOI = {10.1016/0040-9383(95)00068-2},
       URL = {https://doi.org/10.1016/0040-9383(95)00068-2},
}

@article {gordon-luecke,
    AUTHOR = {Gordon, C. McA. and Luecke, J.},
     TITLE = {Knots are determined by their complements},
   JOURNAL = {J. Amer. Math. Soc.},
  FJOURNAL = {Journal of the American Mathematical Society},
    VOLUME = {2},
      YEAR = {1989},
    NUMBER = {2},
     PAGES = {371--415},
      ISSN = {0894-0347,1088-6834},
   MRCLASS = {57M25 (57M40)},
  MRNUMBER = {965210},
MRREVIEWER = {Martin\ Scharlemann},
       DOI = {10.2307/1990979},
       URL = {https://doi.org/10.2307/1990979},
}

@article {feustel-witten,
    AUTHOR = {Feustel, C. D. and Whitten, Wilbur},
     TITLE = {Groups and complements of knots},
   JOURNAL = {Canadian J. Math.},
  FJOURNAL = {Canadian Journal of Mathematics. Journal Canadien de
              Math\'{e}matiques},
    VOLUME = {30},
      YEAR = {1978},
    NUMBER = {6},
     PAGES = {1284--1295},
      ISSN = {0008-414X,1496-4279},
   MRCLASS = {57M25},
  MRNUMBER = {511562},
MRREVIEWER = {L.\ Neuwirth},
       DOI = {10.4153/CJM-1978-105-0},
       URL = {https://doi.org/10.4153/CJM-1978-105-0},
}

@inproceedings{gabai_murasugisumnatural,
	author = {Gabai, David},
	booktitle = {Low-dimensional topology ({S}an {F}rancisco, {C}alif., 1981)},
	doi = {10.1090/conm/020/718138},
	mrclass = {57M25 (57N10)},
	mrnumber = {718138},
	pages = {131--143},
	publisher = {Amer. Math. Soc.},
	series = {Contemp. Math.},
	title = {The {M}urasugi sum is a natural geometric operation},
	url = {https://doi.org/10.1090/conm/020/718138},
	volume = {20},
	year = {1983}}

@article{murasugi,
	author = {Murasugi, Kunio},
	doi = {10.2307/2373107},
	fjournal = {American Journal of Mathematics},
	issn = {0002-9327},
	journal = {Amer. J. Math.},
	mrclass = {55.20},
	mrnumber = {157375},
	mrreviewer = {L. Neuwirth},
	pages = {544--550},
	title = {On a certain subgroup of the group of an alternating link},
	url = {https://doi.org/10.2307/2373107},
	volume = {85},
	year = {1963}}

@article{ozbagci_popescupampu,
	author = {Ozbagci, Burak and Popescu-Pampu, Patrick},
	doi = {10.1007/s40598-015-0033-3},
	fjournal = {Arnold Mathematical Journal},
	issn = {2199-6792},
	journal = {Arnold Math. J.},
	mrclass = {57R17 (57R75)},
	mrnumber = {3460043},
	mrreviewer = {Richard Keith Hind},
	number = {1},
	pages = {69--119},
	title = {Generalized plumbings and {M}urasugi sums},
	url = {https://doi.org/10.1007/s40598-015-0033-3},
	volume = {2},
	year = {2016}}

@article{thompson1994,
	author = {Thompson, Abigail},
	journal = {Pacific J. Math},
	number = {2},
	pages = {393--395},
	title = {A note on {M}urasugi sums},
	volume = {163},
	year = {1994}}

@article{bonahon_siebenmann2010,
	author = {Bonahon, Francis and Siebenmann, Laurent},
	journal = {preprint},
	note = {Available at: \url{https://dornsife.usc.edu/francis-bonahon/wp-content/uploads/sites/205/2023/06/BonSieb-compressed.pdf}},
	title = {New geometric splittings of classical knots and the classification and symmetries of arborescent knots},
	year = {2010}}

@inproceedings{stallings1978,
	author = {Stallings, John R.},
	booktitle = {{A}lgebraic and {G}eometric {T}opology, {P}art 2,},
	pages = {55--60},
	series = {Proc. Sympos. Pure Math.},
	title = {Constructions of fibred knots and links},
	volume = {32.1},
	year = {1978}}

@article {menasco,
    AUTHOR = {Menasco, W.},
     TITLE = {Closed incompressible surfaces in alternating knot and link
              complements},
   JOURNAL = {Topology},
  FJOURNAL = {Topology. An International Journal of Mathematics},
    VOLUME = {23},
      YEAR = {1984},
    NUMBER = {1},
     PAGES = {37--44},
      ISSN = {0040-9383},
   MRCLASS = {57M25},
  MRNUMBER = {721450},
MRREVIEWER = {Cameron\ McA.\ Gordon},
       DOI = {10.1016/0040-9383(84)90023-5},
       URL = {https://doi.org/10.1016/0040-9383(84)90023-5},
}

@article{lines1985,
	author = {Lines, Daniel},
	journal = {Journal of the London Mathematical Society},
	number = {3},
	pages = {557--571},
	publisher = {Oxford University Press},
	title = {On Odd-Dimensional Fibred Knots Obtained by Plumbing and Twisting},
	volume = {2},
	year = {1985}}

@article{lines1986,
	author = {Lines, Daniel},
	doi = {10.1017/S0305004100065919},
	journal = {Mathematical Proceedings of the Cambridge Philosophical Society},
	number = {1},
	pages = {117--131},
	title = {On even-dimensional fibred knots obtained by plumbing},
	volume = {100},
	year = {1986}}

@article{lines1987,
	author = {Lines, Daniel},
	journal = {Canadian Mathematical Bulletin},
	number = {4},
	pages = {429--435},
	publisher = {Cambridge University Press},
	title = {Stable plumbing for high odd-dimensional fibred knots},
	volume = {30},
	year = {1987}}

@inproceedings{conway1970,
	author = {Conway, John H.},
	booktitle = {Computational Problems in Abstract Algebra (Proc. Conf., Oxford, 1967)},
	pages = {329--358},
	publisher = {Pergamon Press},
	title = {An enumeration of knots and links, and some of their algebraic properties},
	year = {1970}}

@article{gabai1986genera,
	author = {Gabai, David},
	journal = {Memoirs of the American Mathematical Society},
	number = {59},
	pages = {1-98},
	publisher = {American Mathematical Soc.},
	title = {Genera of the arborescent Links},
	volume = {339},
	year = {1986}}

@article{montesinos1973,
	author = {Montesinos, Jose M},
	journal = {Bol. Soc. Mat. Mexicana},
	number = {1},
	pages = {32},
	title = {Variedades de {S}eifert que son recubridores ciclicos ramificados de dos hojas},
	volume = {18},
	year = {1973}}

@article{hatcher_thurston1985,
	author = {Hatcher, Allen and Thurston, William},
	journal = {Inventiones mathematicae},
	number = {2},
	pages = {225--246},
	publisher = {Springer-Verlag Berlin/Heidelberg},
	title = {Incompressible surfaces in 2-bridge knot complements},
	volume = {79},
	year = {1985}}

@article {trisections_seifertsolids,
    AUTHOR = {Joseph, Jason and Meier, Jeffrey and Miller, Maggie and Zupan,
              Alexander},
     TITLE = {Bridge trisections and {S}eifert solids},
   JOURNAL = {Algebr. Geom. Topol.},
  FJOURNAL = {Algebraic \& Geometric Topology},
    VOLUME = {25},
      YEAR = {2025},
    NUMBER = {3},
     PAGES = {1501--1522},
      ISSN = {1472-2747,1472-2739},
   MRCLASS = {57K45 (57K10)},
  MRNUMBER = {4930569},
       DOI = {10.2140/agt.2025.25.1501},
       URL = {https://doi.org/10.2140/agt.2025.25.1501},
}

@article{cappell_shaneson76,
	author = {Sylvain E. Cappell and Julius L. Shaneson},
	issn = {0003486X, 19398980},
	journal = {Annals of Mathematics},
	number = {1},
	pages = {61--72},
	publisher = {[Annals of Mathematics, Trustees of Princeton University on Behalf of the Annals of Mathematics, Mathematics Department, Princeton University]},
	title = {Some New Four-Manifolds},
	url = {http://www.jstor.org/stable/1971056},
	urldate = {2025-10-23},
	volume = {104},
	year = {1976}}

@article{kindred_essential,
	adsurl = {https://ui.adsabs.harvard.edu/abs/2024arXiv240816948K},
	archiveprefix = {arXiv},
	author = {{Kindred}, Thomas},
	doi = {10.48550/arXiv.2408.16948},
	eid = {arXiv:2408.16948},
	eprint = {2408.16948},
	journal = {arXiv e-prints},
	month = aug,
	note = {\href{https://arxiv.org/abs/2408.16948}{\UrlFont{arXiv:2408.16948}}},
	primaryclass = {math.GT},
	title = {{How essential is a spanning surface?}},
	year = 2024}

@article{adams_kindred2013,
	author = {Colin C. Adams and Thomas Kindred},
	journal = {Algebraic and Geometric Topology},
	number = {5},
	pages = {2967--3007},
	title = {A classification of spanning surfaces for alternating links},
	volume = {13},
	year = {2013}}

@article{yajima1964,
  title={On simply knotted spheres in ${R}^4$.},
  author={Yajima, Takeshi},
    journal={Osaka J. Math.},
volume={1},
pages= {133-152},
  year={1964}
}

@book{carter_kamada_saito,
  title={Surfaces in 4-space},
  author={Carter, Scott and Kamada, Seiichi and Saito, Masahico},
  volume={142},
  year={2013},
  publisher={Springer Science \& Business Media}
}

@incollection {kanenobu1987,
    AUTHOR = {Kanenobu, Taizo},
     TITLE = {Untwisted deform-spun knots: examples of symmetry-spun
              {$2$}-knots},
 BOOKTITLE = {Transformation groups ({O}saka, 1987)},
    SERIES = {Lecture Notes in Math.},
    VOLUME = {1375},
     PAGES = {145--167},
 PUBLISHER = {Springer, Berlin},
      YEAR = {1989},
      ISBN = {3-540-51218-7},
   MRCLASS = {57Q45},
  MRNUMBER = {1006689},
MRREVIEWER = {Alexander\ I.\ Suciu},
       DOI = {10.1007/BFb0085606},
       URL = {https://doi.org/10.1007/BFb0085606},
}

@article {teragaito1989,
    AUTHOR = {Teragaito, Masakazu},
     TITLE = {Twisting symmetry-spins of {$2$}-bridge knots},
   JOURNAL = {Kobe J. Math.},
  FJOURNAL = {Kobe Journal of Mathematics},
    VOLUME = {6},
      YEAR = {1989},
    NUMBER = {1},
     PAGES = {117--125},
      ISSN = {0289-9051,2760-5760},
   MRCLASS = {57Q45},
  MRNUMBER = {1023531},
MRREVIEWER = {Alexander\ I.\ Suciu},
}

@article {teragaito1990,
    AUTHOR = {Teragaito, Masakazu},
     TITLE = {Twisting symmetry-spins of pretzel knots},
   JOURNAL = {Proc. Japan Acad. Ser. A Math. Sci.},
  FJOURNAL = {Japan Academy. Proceedings. Series A. Mathematical Sciences},
    VOLUME = {66},
      YEAR = {1990},
    NUMBER = {7},
     PAGES = {179--183},
      ISSN = {0386-2194},
   MRCLASS = {57Q45},
  MRNUMBER = {1078403},
MRREVIEWER = {Jonathan\ A.\ Hillman},
       URL = {http://projecteuclid.org/euclid.pja/1195512399},
}

@article {shima1998,
    AUTHOR = {Shima, Akiko},
     TITLE = {An unknotting theorem for tori in {$S^4$}},
   JOURNAL = {Rev. Mat. Complut.},
  FJOURNAL = {Revista Matem\'atica Complutense},
    VOLUME = {11},
      YEAR = {1998},
    NUMBER = {2},
     PAGES = {299--309},
      ISSN = {1139-1138,1988-2807},
   MRCLASS = {57Q45 (57Q35)},
  MRNUMBER = {1666485},
MRREVIEWER = {J.\ P.\ Levine},
       DOI = {10.5209/rev\_REMA.1998.v11.n2.17249},
       URL = {https://doi.org/10.5209/rev_REMA.1998.v11.n2.17249},
}

@misc {cklsv26,
    AUTHOR = {Cohen, Moshe and Kindred, Thomas and Lowrance, Adam and Shanahan Patrick and Van Cott, Cornelia},
     TITLE = {Surgery moves connecting all essential spanning surfaces for a 2-bridge link},
     note = {\textit{Forthcoming}}
}

@article{satoh02,
  title={Positive, alternating, and pseudo-ribbon surface-knots},
  author={Satoh, Shin and Shima, A},
  journal={Kobe J. Math},
  volume={19},
  pages={51--59},
  year={2002}
}

@article {ozawa11,
    AUTHOR = {Ozawa, Makoto},
     TITLE = {Essential state surfaces for knots and links},
     JOURNAL = {J. Aust. Math. Soc.},
  FJOURNAL = {Journal of the Australian Mathematical Society},
    VOLUME = {91},
      YEAR = {2011},
     PAGES = {391–-404}
}

\end{document}